\documentclass[12pt]{amsart}
\usepackage{amssymb}
\usepackage{mathrsfs}
\usepackage{bm}
\usepackage{graphicx}
\usepackage{color}
\usepackage[all]{xy}

\allowdisplaybreaks[4]

\newtheorem{THEOREM}{Theorem}
\newtheorem{thm}{Theorem}[section]
\newtheorem{lem}[thm]{Lemma}

\newtheorem{prop}[thm]{Proposition}
\newtheorem{definition}[thm]{Definition}

\def \para{\refstepcounter{thm} \par\medskip\noindent
                \textbf{\thethm .} }

\def \remark{\refstepcounter{thm} \par\medskip\noindent
                \textbf{Remark \thethm .} }

\numberwithin{equation}{thm}

\newenvironment{red}
{\relax\color{red}}
{\hspace*{.5ex}\relax}

\newcommand{\ber}{\begin{red}}
\newcommand{\er}{\end{red}}

\newcommand\NN{\mathbb N}

\newcommand\ZZ{\mathbb Z}

\newcommand\bC{\mathbf C}
\newcommand\bD{\mathbf D}

\newcommand\cC{\mathcal{C}}

\newcommand\cT{\mathcal{T}}

\newcommand\fs{\mathfrak s}
\newcommand\ft{\mathfrak t}
\newcommand\fu{\mathfrak u}
\newcommand\fv{\mathfrak v}

\renewcommand\a{\alpha}  %#####
\renewcommand\b{\beta}   %#####
\newcommand\g{\gamma}  
\renewcommand\d{\delta}  %#####

\newcommand\la{\lambda}

\renewcommand\xi{\xi}
\renewcommand\pi{\pi}

\newcommand\D{\Delta}

\renewcommand\Xi{\Xi}
\renewcommand\Pi{\Pi}

\newcommand\vL{\varLambda}

\newcommand\wh{\widehat}

\newcommand\ol{\overline}

\newcommand\ra{\rightarrow}

\newcommand\lan{\langle}
\newcommand\ran{\rangle}

\newcommand\Hom{\operatorname{Hom}}
\newcommand\End{\operatorname{End}}

\newcommand\Ext{\operatorname{Ext}}
\newcommand\Top{\operatorname{Top}}

\newcommand\Soc{\operatorname{Soc}}
\newcommand\cha{\operatorname{char}}

\newcommand\rad{\operatorname{rad}}

\newcommand\cmod{\operatorname{-mod}}

\newcommand{\isom}{\,\raise2pt\hbox{$\underrightarrow{\sim}$}\,}

\newcounter{ichi}
\newcommand{\roi}{\roman{ichi}}
\newcounter{ni}
\newcommand{\roii}{\roman{ni}}
\newcounter{san}
\newcommand{\roiii}{\roman{san}}
\newcounter{yon}
\newcommand{\roiv}{\roman{yon}}
\newcounter{go}
\newcounter{roku}
\newcounter{nana}
\newcounter{hachi}
\newcounter{kyu}
\begin{document}

\setlength{\baselineskip}{4.9mm}
\setlength{\abovedisplayskip}{4.5mm}
\setlength{\belowdisplayskip}{4.5mm}

%%%%%%%%%%%%%%%%%%%%%%%%%%%%%%%%%%%%%%%%%%%%

\renewcommand{\theenumi}{\roman{enumi}}
\renewcommand{\labelenumi}{(\theenumi)}
\renewcommand{\thefootnote}{\fnsymbol{footnote}}
\renewcommand{\thefootnote}{\fnsymbol{footnote}}
\parindent=20pt

%%%%%%%%%%%%%%%%%%%%%%%%%%%%%%%%%%%%%%%%%%%%

\setcounter{section}{0}

%%%%%%%%%%%%%%%%%%%%%%%%%%%%%%%%%%%%%%%%%%%%

%%%%%%%%%%%%%%%%%%%%%%%%%%%%%%%%%%%%%%%%%%%%%%%%%%%%%%%%%%%%%%%

%%%%%%%%%%%%%%%%%%%%%%%%%%%%%%%%%%%%%%%%%%%%%%%%%%%%%%%%%%%%%%%

\title[Self-injective cellular algebras]{Self-injective cellular algebras of polynomial growth representation type}

\address{S.Ariki : Department of Pure and Applied Mathematics, Graduate School of Information Science and Technology, Osaka University, 1-5 Yamadaoka, Suita, Osaka 565-0871, Japan}
\thanks{$^1$S.A. is supported by the JSPS Grant-in-Aid for Scientific Research 15K04782.}
\email{ariki@ist.osaka-u.ac.jp}

\address{R.Kase : Faculty of Informatics, Okayama University of Science, 1-1 Ridaicho, Kita-ku, Okayama-shi 700-0005, Japan}
\email{r-kase@mis.ous.ac.jp}

\address{K.Miyamoto :  Department of Pure and Applied Mathematics, Graduate School of Information Science and Technology, Osaka University, 1-5 Yamadaoka, Suita, Osaka 565-0871, Japan}
\email{k-miyamoto@ist.osaka-u.ac.jp}

\address{K.Wada : Department of Mathematics, Faculty of Science, Shinshu University, 
		Asahi 3-1-1, Matsumoto 390-8621, Japan}
\thanks{$^2$K.W. is supported by the JSPS Grant-in-Aid for Scientific Research 16K17565.}
\email{wada@math.shinshu-u.ac.jp} 

\maketitle

%%%%%%%%%%%%%%%%%%%%%%%%%%%%%%%%%%%%%%%%%%%%%%%%%%%%%%%%%%%%%%%
\begin{center}
S.Ariki$^1$, R.Kase, K.Miyamoto and K.Wada$^2$
\end{center}
%%%%%%%%%%%%%%%%%%%%%%%%%%%%%%%%%%%%%%%%%%%%%%%%%%%%%%%%%%%%%%%

\begin{abstract}
We classify Morita equivalence classes of indecomposable self-injective cellular algebras 
which have polynomial growth representation type, 
assuming that the base field has an odd characteristic. 
This assumption on the characteristic is for the cellularity to be
a Morita invariant property.
\end{abstract}

%%%%%%%%%%%%%%%%%%%%%%%%%%%%%%%%%%%%%%%%%%%%%%%%%%%%%%%%%%%%%%%
\section{Introduction}
%%%%%%%%%%%%%%%%%%%%%%%%%%%%%%%%%%%%%%%%%%%%%%%%%%%%%%%%%%%%%%%

Cellular algebras often appear as finite dimensional algebras arising from various setting in Lie theory. 
A fundamental example is a block algebra of the Hecke algebra associated with a finite Weyl group 
-- in both classical and exceptional type \cite{G},
or the cyclotomic Hecke algebra associated with the complex reflection group $G(m,1,n)$. 
There are more examples from diagram algebras. 
Among them, algebras of tame representation type can receive a detailed study. 
For example, if we consider tame block algebras of the Hecke algebras associated with the symmetric group,
\footnote{They appear only when the quantum characteristic of the Hecke algebras is $2$.}
we have only two Morita equivalence classes
\footnote{They are Brauer graph algebras whose Brauer graphs are straight lines of length two 
such that two of the three vertices are exceptional vertices and their multiplicities are both two.},
which are derived equivalent. Note that those block algebras are symmetric cellular algebras. 

A natural question is whether this classification is achievable only because we consider algebras of very special kind, or 
the property of being symmetric cellular of tame type is strong enough to determine 
Morita equivalence classes of algebras. 
To assure that cellularity is a Morita invariant property, 
we assume that the base field has an odd characteristic throughout the paper.

If we consider block algebras of finite type, we meet Brauer tree algebras and their Brauer trees are always straight lines. 
The second half is explained by a result from K\"onig and Xi \cite{KX1}, which states that a Brauer tree algebra is cellular 
if and only if the Brauer tree is a straight line (with at most one exceptional vertex). 
Indeed, we prove that the property of being indecomposable self-injective cellular algebras 
of finite type is already a sufficient condition to be Brauer tree algebras whose Brauer trees are straight lines. 
Hence, that block algebras of Hecke algebras are algebras of very special kind plays almost 
no role for the classification of Morita equivalence classes of the block algebras of finite type.
\footnote{Note however that we do not know whether it is possible to determine the exceptional vertex of the Brauer tree algebras 
without appealing to results which are developed for Hecke algebras.}

After considering finite representation type, we investigate how far the property of being indecomposable self-injective cellular algebras 
of tame type put restrictions. As we do not know how to analyze algebras of tame type in general, 
we focus on tame of polynomial growth in this paper, and
we classify Morita equivalence classes of indecomposable self-injective cellular algebras of polynomial growth 
assuming that the base field has an odd characteristic. Our main result is as follows.
The algebras (1) are from Theorem \ref{F Theorem}, the algebra (2) is from Proposition \ref{D prop local} and Proposition \ref{D I-2 cellular}, 
the algebra (3) is from Proposition \ref{D I-2 cellular}, the algebras (4) are from Proposition \ref{D II} 
and Lemma \ref{D V-5}, the algebra (5) is from Lemma \ref{D V-5}, 
the algebras (6) are from Proposition \ref{D IV}, and the classification results from (2)-(6) are
summarized in Theorem \ref{D main result}. The classification (7) is Theorem \ref{P main}. 
It is worth mentioning that they are all standard algebras. 

\begin{THEOREM}
Let $K$ be an algebraically closed field of odd characterisitic, and let $A$ be an indecomposable self-injective cellular algebra over $K$. 
If $A$ is tame of polynomial growth, then
$A$ is Morita equivalent to one of the following symmetric algebras. Further, the algebras (1) are of finite type, 
the algebras (2)-(6) are of domestic infinite type, the algebras (7) are nondomestic of polynomial growth.

\begin{itemize}
\item[(1)]
A Brauer tree algebra whose Brauer tree is a straight line with at most one exceptional vertex. 
\item[(2)]
The Kronecker algebra $K[X,Y]/(X^2,Y^2)$.
\item[(3)]
The algebra $K Q/I$, where the quiver $Q$ is
\[
\begin{xy}
(0,0) *[o]+{\circ}="A", (10,0) *[o]+{\circ}="B",
\ar @(ur,dr) "B";"B"^{\gamma}
\ar @<1mm> "A";"B"^\alpha
\ar @<1mm> "B";"A"^\beta
\end{xy}
\]
and the relations given by the admissible ideal $I$ are
\[
\alpha\beta=0, \;\; \gamma^2=0, \;\;
\alpha\gamma\beta\alpha=\beta\alpha\gamma\beta=0, \;\; \gamma\beta\alpha=\beta\alpha\gamma.
\]
\item[(4)]
The algebras $KQ/I$, where the quiver $Q$ is
\[
\begin{xy}
(0,0) *[o]+{\circ}="A", (10,0) *[o]+{\circ}="B", (20,0) *[o]+{\circ}="C",
\ar @<1mm> "A";"B"^{\alpha}
\ar @<1mm> "B";"C"^{\beta}
\ar @<1mm> "B";"A"^{\delta}
\ar @<1mm> "C";"B"^{\gamma}
\end{xy}
\]
and the relations given by the admissible ideal $I$ is either (a) or (b) given below. (The relations 
(a) and (b) define different algebras.)

\begin{itemize}
\item[(a)]
$\alpha\beta\gamma\delta\alpha=\beta\gamma\delta\alpha\beta=\gamma\delta\alpha\beta\gamma=\delta\alpha\beta\gamma\delta=0$, 
$\gamma\beta=0$, $\alpha\delta\alpha=\alpha\beta\gamma$ and $\delta\alpha\delta=\beta\gamma\delta$.
\item[(b)]
$\alpha(\delta\alpha)^2=\gamma(\delta\alpha)^2=0$, $(\delta\alpha)^2\delta=(\delta\alpha)^2\beta=0$, 
$\alpha\delta\alpha\beta=0$, $\gamma\beta\gamma\delta=0$ and $\delta\alpha=\beta\gamma$.
\end{itemize}
\item[(5)]
The algebra $KQ/I$, where the quiver $Q$ is
\[
\begin{xy}
(0,0) *[o]+{\circ}="A", (10,0) *[o]+{\circ}="B", (20,0) *[o]+{\circ}="C",
\ar @<1mm> "A";"B"^{\alpha}
\ar @<1mm> "B";"C"^{\beta}
\ar @<1mm> "B";"A"^{\delta}
\ar @<1mm> "C";"B"^{\gamma}
\ar @(ru,lu) "B";"B"_{\epsilon}
\end{xy}
\]
and the relations given by the admissible ideal $I$ are
\begin{center}
$\alpha\beta=\gamma\delta=0$, $\alpha\epsilon=\epsilon\beta=\gamma\epsilon=\epsilon\delta=0$ and $\delta\alpha=\epsilon^2=\beta\gamma$.
\end{center}
\item[(6)]
A Brauer graph algebra whose Brauer graph is a straight line with exactly two exceptional vertices such that 
their multiplicities are both two.
\item[(7)]
The algebras $A_1(\la), A_2(\la)$, where $\la\in K\setminus\{0,1\}$, and $A_4, A_7, A_{11}$ 
from the table in Theorem \ref{P theorem standard nondomestic}.
\end{itemize}

In particular, any indecomposable self-injective cellular algebra of polynomial growth is of the form $fM_n(A)f$, 
where $A$ is an algebra from the above list, $n\in\NN$, and $f$ is an idempotent of the matrix algebra $M_n(A)$.
\end{THEOREM}

\medskip
{\bf Notation :} 
For a finite dimensional algebra $A$ over a field $K$, 
we denote  the category of finite dimensional left $A$-modules by $A \cmod$. 
% For  $M \in A \cmod$ , we denote by $[M]$ the image of $M$ in the Grothendieck group of $A\cmod$. 
For an $A$-module $M$ and a simple $A$-module $L$, 
we denote the composition multiplicity of $L$ in $M$ by $[M:L]$. 

Let $Q =(Q_0,Q_1)$ be a finite quiver, 
where $Q_0$ is the set of vertices and $Q_1$ is the set of arrows. 
Let $I$ be an admissible ideal of the path algebra $K Q$ of $Q$ over a field $K$, so that $A:=K Q/I$ is a finite dimensional $K$-algebra.
For a vertex $i \in Q_0$, we denote by $e_i$ the corresponding primitive idempotent of $A$, 
and by $L_i$ (resp. $P_i$) the corresponding simple (resp. indecomposable projective) $A$-module.

%%%%%%%%%%%%%%%%%%%%%%%%%%%%%%%%%%%%%%%%%%%%%%%%%%%%%%%%%%%%%%%
\section{Cellular algebras} 
%%%%%%%%%%%%%%%%%%%%%%%%%%%%%%%%%%%%%%%%%%%%%%%%%%%%%%%%%%%%%%%

The notion of cellular algebra was introduced in \cite{GL},  
and basic properties of the cellular algebra were studied in \cite{GL}, \cite{KX1}, \cite{KX2} \cite{KX3},  \cite{KX4} and \cite{C}. 
Those properties are summarized in \cite{M} and \cite{X}. 
In this paper, we follow the definition given in \cite{M}: note that the partial order in \cite{GL} is the reversed one.
In this section, we recall the definition and the basic properties of cellular algebras.

%%%%%

\begin{definition}
An associative algebra $A$ over a field $K$ is called a \emph{cellular algebra} with cell datum $(\vL, \cT, \cC, \iota)$ 
if the following conditions are satisfied: 
\begin{description}
\item[(C1)] 
$\vL$ is a finite partially ordered set with a partial order $\geq$.  
For each $\la \in \vL$, there exists a finite set $\cT(\la)$ such that
the algebra $A$ has a $K$-basis 
\begin{align*} 
\cC =\{c_{\fs \ft}^{\la} \mid \fs, \ft \in \cT(\la), \, \la \in \vL\}. 
\end{align*}

\item[(C2)] 
The linear map $\iota : A \ra A$ which sends $c_{\fs \ft}^{\la}$ to $c_{\ft \fs}^{\la}$ ($\fs, \ft \in \cT(\la)$, $\la \in \vL$) 
gives an algebra anti-automorphism of $A$. 
(Clearly, $\iota$ is an involution.)

\item[(C3)] 
For any $\la \in \vL$, $\fs \in \cT(\la)$ and $a \in A$, there exist $r_{\fu}^{(a, \fs)} \in K$ ($\fu\in \cT(\la)$) such that 
\begin{align}
\label{C a cst}
a c_{\fs \ft}^{\la} \equiv \sum_{\fu \in\cT(\la)} r_{\fu}^{(a, \fs)} c_{\fu \ft}^{\la} \mod A^{>\la}, 
\end{align}
where $A^{> \la}$ is a subspace of $A$ spanned by 
$\{ c_{\fs' \ft'}^{\la'} \mid \fs', \ft' \in \cT(\la'), \,  \la' > \la, \la' \in \vL\}$, 
and $r_{\fu}^{(a,\fs)}$ does not depend on $\ft$. 
\end{description}
The above basis $\cC$ is called a \emph{cellular basis}. 
\end{definition}

Applying the anti-involution $\iota$ to \eqref{C a cst}, we have 
\begin{align}
\label{C cts a} 
c_{\ft \fs}^{\la} a \equiv \sum_{\fu \in \cT(\la)} r_{\fu}^{(\iota(a), \fs)} c_{\ft \fu}^{\la} \mod A^{>\la}.
\end{align}
%%%%%

\para {\bf Cell modules.} 
Let $A$ be a cellular algebra with cell datum $(\vL, \cT, \cC, \iota)$. 
For each $\la \in \vL$,  
let $\D(\la)$ be a $K$-vector space with a basis $\{c_{\fs}^{\la} \mid \fs \in \cT(\la)\}$. 
We can define an action of $A$ on $\D(\la)$ by 
\begin{align*}
a \cdot c_{\fs}^{\la} = \sum_{\fu \in \cT(\la)} r_{\fu}^{(a, \fs)} c_{\fu}^{\la}, 
\end{align*}
where $r_{\fu}^{(a, \fs)} \in K$ are those coefficients given in \eqref{C a cst}. 
We call the $A$-modules $\D(\la)$ ($\la \in \vL$) \emph{cell modules}. 

By \eqref{C a cst} and \eqref{C cts a}, 
for $\la \in \vL$ and $\fs, \ft, \fu, \fv \in \cT(\la)$, 
we see that  
\begin{align*}
c_{\fu \fs}^{\la} c_{\ft \fv}^{\la} \equiv r_{\fs \ft} c_{\fu \fv}^{\la} \mod A^{> \la}, 
\end{align*}
where $r_{\fs \ft}\in K $ does not depend on $\fu$ and $\fv$. 
Then we define a bilinear form $\lan \, , \, \ran$ on $\D(\la)$ by $\lan c_{\fs}^{\la}, c_{\ft}^{\la} \ran = r_{\fs \ft}$.
Let $\rad_{\lan \,,\,\ran} \D(\la)$ be the radical of the bilinear form $\lan \,,\, \ran$ on $\D(\la)$. 
Then $\rad_{\lan \,,\, \ran} \D(\la)$ is an $A$-submodule of $\D(\la)$ by \cite[Proposition 3.2]{GL}. 
Put $L(\la) = \D(\la) / \rad_{\lan \,,\, \ran} \D(\la)$, 
and $ \vL^+ = \{\la \in \vL \mid L(\la) \not=0\}$. 
Then we have the following proposition. 

%%%%%

\begin{prop}[{\cite[Proposition 3.2, Theorem 3.4]{GL}, \cite[Lemma 2.5]{C}}]\ 
\label{C Prop simple}
\begin{enumerate}
\item 
For $\la \in \vL^+$, $L(\la)$ is absolutely irreducible. 

\item 
The set $\{L(\la) \mid \la \in \vL^+\}$ is a complete set of representatives of isomorphism classes of simple $A$-modules. 

\item 
For $\la, \mu \in \vL^+$, we have 
$\dim_K \Ext_{A}^i (L(\la), L(\mu)) = \dim_K \Ext_{A}^i (L(\mu), L(\la))$ for all $i \geq 0$. 

\end{enumerate}
\end{prop}

Proposition \ref{C Prop simple}(iii) implies that if a cellular algebra $A$ has the bounded quiver presentation $KQ_A/I$ 
for the Gabriel quiver $Q_A$ of $A$ and an admissible ideal $I$, 
then each arrow of $Q_A$ must have the inverse arrow. We note that the five classes of algebras we will consider in \S3 are 
defined in the form $KQ/I$ but $I$ are not admissible in the first place, so that we must rewrite them when we apply Proposition \ref{C Prop simple}(iii).
However, in our application, we will see that the rewriting is merely deletion of the loops in the quiver $Q$ which come from vertices of valency 
one in Brauer graphs we start with.
Hence, Proposition \ref{C Prop simple}(iii) implies that if one of the algebras we consider in \S3 is cellular, then 
either there is no branch in the Brauer graph or 
there is one loop and the endpoint of the loop is the only branch in the Brauer graph whose valency at most four.
We use this criterion intensively in \S3.

%%%%%

\para {\bf The decomposition matrix and the Cartan matrix.}
For $\la \in \vL$ and $\mu \in \vL^+$, 
let $d_{\la\mu}=[\D(\la):L(\mu)]$ be the composition multiplicity of $L(\mu)$ in $\D(\la)$, 
and we call them \emph{decomposition numbers} of the cellular algebra $A$. 
Put $\bD=(d_{\la \mu})_{\la \in \vL, \mu \in \vL^+}$, and we call it the \emph{decomposition matrix}. 
For $\la\in \vL^+$, let $P(\la)$ be the projective cover of $L(\la)$. 
Let $c_{\la\mu} = [P(\la) : L(\mu)]$, for $\la, \mu \in \vL^+$, and put $\bC =(c_{\la \mu})_{\la,\mu \in \vL^+}$. 
Then $\bC$ is the \emph{Cartan matrix} of $A$. 

We may construct a sequence of two-sided ideals for each of the linear extensions of the partial order. First, 
we enumerate the elements of $\vL$ as $\{\mu_1,\mu_2,\dots \}$ following a chosen linear extension. Hence $i < j$ if $\mu_i < \mu_j$. Next,
let $A(\mu_i)$ be the subspace of $A$ spanned by $\{c_{\fs \ft}^{\mu_k} \mid \fs, \ft \in \cT(\mu_k), \, k \geq i\}$, for $\mu_i \in \vL$. 
Then $A(\mu_i)$, for $1\le i \le |\vL|$, where $|\vL|$ is the size of $\vL$, is a two-sided ideal of $A$ by \eqref{C a cst} and \eqref{C cts a}. Defining $ A(\mu_{|\vL|+1})$ to be $0$, 
we have the following sequence of two-sided ideals associated with the linear extension
\begin{align}
\label{C seq A}
A= A(\mu_1) \supset A(\mu_2) \supset \dots \supset A(\mu_{|\vL|}) \supset A(\mu_{|\vL|+1})=0.
\end{align}  
For each $\la \in \vL^+$, applying the exact functor $- \otimes_{A} P(\la)$ to \eqref{C seq A}, 
we have a sequence of $A$-submodules 
\begin{align} 
\label{C seq Pla}
P(\la) = P(\la)_{\mu_1} \supset P(\la)_{\mu_2} \supset \dots \supset P(\la)_{\mu_{|\vL|}} \supset P(\la)_{\mu_{|\vL|}+1}=0, 
\end{align}
where we put $P(\la)_{\mu_i} = A(\mu_i) \otimes_{A} P(\la)$. 
Then we have the following proposition. 

%%%%%

\begin{prop}[{\cite[Proposition 3.6, Theorem 3.7]{GL}, \cite[Proposition 1.2]{KX3}}]\ 
\label{C Prop decom matrix}
\begin{enumerate} 
\item 
For $\la \in \vL$ and $\mu \in \vL^+$, 
if $d_{\la\mu} \not=0$, then we have $\la \geq \mu$. 
Moreover, if $\la \in \vL^+$, we have $d_{\la \la} =1$. 

\item 
For $\la \in \vL^+$ and the $A$-submodules $P(\la)_{\mu_i}$ in the sequence \eqref{C seq Pla}, we have 
$P(\la)_{\mu_i}/ P(\la)_{\mu_{i+1}} \cong \D(\mu_i)^{\oplus d_{\mu_i \la}}$ 
as $A$-modules. 

\item 
We have 
$\bC= \bD^T \bD$, where $\bD^T$ is the transpose of $\bD$. 

\item 
We have $\det \bC >0$. 

\end{enumerate}
\end{prop}

%%%%%

We also use the following results proved in \cite{KX2} and \cite{KX4}. 

\begin{prop}[{\cite[Theorem 8.1]{KX2}, \cite[Theorem 1.1]{KX4}}] 
\label{C Prop Morita equiv} 
Let $A$ be a cellular algebra over a field $K$. 
Then we have the following. 
\begin{enumerate} 
\item
If $\cha K \not=2$ then the basic algebra $B$ of $A$ is also cellular and 
there is a Morita equivalence between $A$ and $B$ which respects the cellular structure.

\item 
If $A$ is self-injective then $A$ is weakly symmetric. 

\end{enumerate}
\end{prop}

%%%%%

Thanks to Proposition \ref{C Prop Morita equiv} (\roi), 
it is enough to consider basic algebras in order to determine Morita equivalence classes of cellular algebras  
if the characteristic of the base field is odd.  
Another useful result is the following lemma. 

\begin{lem}[{\cite[Corollary 8.3]{KX2}}]
\label{C Lem idem}
Suppose that $A$ is a cellular algebra over a field $K$ of odd characteristic. If $P$ is a finitely generated projective $A$-module, 
then $\End_A(P)$ is a cellular algebra.
\end{lem}

\remark
For the proof, let $Q$ be the direct sum of pairwise inequivalent indecomposable projective $A$-modules that appear 
as direct summands of $P$. Then the opposite algebra of $\End_A(Q)$ is of the form $eAe$, for an idempotent $e\in A$. 
We may prove that $eAe$ is a cellular algebra. As $\End_A(Q)$ is the basic algebra of $\End_A(P)$, we have the isomorphism of algebras
$\End_A(P)\cong fM_n(\End_A(Q))f$, for an idempotent $f$ of the matrix algebra $M_n(\End_A(Q))$. 
Noting that a matrix algebra over a cellular algebra is a cellular algebra, we deduce that $\End_A(P)$ is a cellular algebra by the same argument 
which we used to prove that $eAe$ is a cellular algebra. 

%%%%%

\bigskip

Lemma \ref{C Lem idem} has the following three consequences. We will frequently use (a) to prove that an algebra is not cellular. 

\begin{itemize}
\item[(a)]
If $A$ is a cellular algebra over a field of odd characteristic, then $e A e$ is a cellular algebra for any idempotent $e \in A$.  
\item[(b)]
Let $A$ be a finite dimensional algebra over a field of odd characteristic, 
and suppose that $eAe$, for an idempotent $e \in A$, and $A$ are Morita equivalent. 
Then, $A$ is a cellular algebra if $eAe$ is a cellular algebra. 
\item[(c)]
Suppose that $\cha K \not=2$ and let $A$ be a finite dimensional $K$-algebra. Then, $A$ is a cellular algebra if and only if 
its basic algebra is a cellular algebra.
\end{itemize}
For the second statement, we recall that the functor $F:=eA\otimes_A-$ induces 
the equivalence and $Ae\otimes_{eAe}eA\cong A$ holds. Thus, $eA=F(A)$ is a finitely generated projective $eAe$-module and 
$\End_{eAe}(eA)\cong \End_A(A)\cong A^{\rm op}$. 
The third statement follows from Proposition \ref{C Prop Morita equiv} (\roi) and (b).

%%%%%

\bigskip
In this paper, we are interested in self-injective cellular algebras. 
For a self-injective cellular algebra, we have the following lemma. 

\begin{lem} 
\label{C Lem min}
Let $A$ be a cellular algebra over a field $K$ with cell datum $(\vL, \cT, \cC, \iota)$. 
Then we have the following. 
\begin{enumerate}
\item 
If $\la \in \vL$ is minimal, 
then $\la \in \vL^+$, $d_{\la \la }=1$ and $d_{\la \mu}=0$ for all $\mu \in \vL^+ \setminus \{\la\}$. 

\item 
Assume that $A$ is self-injective. 
If $\la \in \vL$ is maximal, 
then there exists $\mu \in \vL^+$ such that 
$d_{\la \mu}=1$ and $d_{\la \nu}=0$ for all $\nu \in \vL^+ \setminus \{ \mu\}$. 
In particular, if $\la \in \vL^+$ then $d_{\la \nu}=0$ for all $\nu \in \vL^+ \setminus \{ \la\}$.

\end{enumerate} 
\end{lem} 

\begin{proof} 
(\roi) follows from Proposition \ref{C Prop decom matrix} (\roi) immediately. 
We prove (\roii). 
Since we assume that $A$ is a self-injective cellular algebra, 
$A$ is a weakly symmetric algebra by Proposition \ref{C Prop Morita equiv} (\roii). 
Let $\la \in \vL$ be maximal and choose a linear extension which ends with $\la$. 
Assume that there exists $\mu \in \vL^+$ such that $d_{\la \mu} \geq 2$. 
Then, by Proposition \ref{C Prop decom matrix} (\roii), 
we see that $P(\mu)$ contains $\D(\la)^{\oplus d_{\la \mu}}$ as an $A$-submodule, 
which implies that $\Soc P(\mu)$ is not irreducible. It contradicts with 
$\Top P(\mu) \cong \Soc P(\mu) \cong L(\mu)$. 
Thus, we have $d_{\la \mu} \leq 1$ for all $\mu \in \vL^+$. 
Next, assume that there exist $\mu, \nu \in \vL^+$ ($\mu \not=\nu$) 
such that $d_{\la \mu}= d_{\la \nu}=1$. 
Then, by the same reasoning, both $P(\mu)$ and $P(\nu)$ contain $\D(\la)$ as an $A$-submodule. It follows that 
we have 
$L(\mu) \cong \Soc P(\mu) \cong \Soc \D(\la)$ 
and 
$L(\nu) \cong \Soc P(\nu) \cong \Soc \D(\la)$. 
However, $L(\mu)$ and $L(\nu)$ are inequivalent $A$-modules by Proposition \ref{C Prop simple} (ii).
Thus, we have (\roii). 
\end{proof}

%%%%%

\para {\bf Strategy.} 
The aim of this paper is to classify self-injective cellular algebras of polynomial growth 
over an algebraically closed field $K$ of $\cha K \not=2$ up to Morita equivalence. 
By the properties of cellular algebras given above, 
we have the following strategy for the classification. 

\begin{description} 
\item[Strategy 1]
By Proposition \ref{C Prop Morita equiv} (\roi) and our assumption that $\cha K \not=2$, 
it is enough to classify basic self-injective cellular algebras of polynomial growth 
up to isomorphism of algebras.  
Since basic self-injective algebras over an algebraically closed field of polynomial growth 
have been classified by several authors, 
and they are listed in \cite{S}, 
it is enough to specify cellular algebras in the list of algebras from \cite{S}.

\item[Strategy 2] 
In order to prove that a basic algebra listed in \cite{S} is cellular,
it is enough to give a cellular basis explicitly. 

\item[Strategy 3] 
In order to prove that a basic algebra listed in \cite{S} is not cellular, 
it is enough to give a contradiction to a property of cellular algebras given in this section 
for a cellular algebra with a small number of simple modules which is obtained from the original basic algebra 
by applying Lemma \ref{C Lem idem}.

\end{description}

Throughout the paper, 
we always assume that $K$ is an algebraically closed field of $\cha K \not=2$.

%%%%%%%%%%%%%%%%%%%%%%%%%%%%%%%%%%%%%%%%%%%%%%%%%%%%%%%%%%%%%%%
\section{Self-injective cellular algebras of finite type} 
%%%%%%%%%%%%%%%%%%%%%%%%%%%%%%%%%%%%%%%%%%%%%%%%%%%%%%%%%%%%%%%
In this section, 
we classify basic self-injective cellular algebras of finite type over $K$. 
In fact, it is known that a weakly symmetric algebra of finite type is symmetric (see Proposition \ref{F Prop weak}), 
and classification of basic symmetric cellular algebras of finite type was obtained in \cite{O}. 
However, the paper \cite{O} is written in Japanese, and it will not appear anywhere. Because of this, 
we give a proof for the classification of basic self-injective cellular algebras of finite type, for the readers' convenience. 

By \cite[Theorem 3.8]{S} and $\cha K \not=2$, the self-injective algebras we consider in this section are standard, 
that is, it admits a simply connected Galois covering whose Galois group is an admissible torsion-free group of $K$-linear automorphisms.

%%%%%

\para 
First, we classify basic symmetric cellular algebras of finite type over $K$. 
By \cite[Theorem 3.11]{S}, 
a nonsimple basic standard symmetric algebra of finite type is isomorphic to one of the following algebras: 
\begin{itemize}
\item 
A Brauer tree algebra whose Brauer tree has at most one exceptional vertex. 

\item 
A modified Brauer tree algebra $D(T_S)$ 
associated to a Brauer tree $T_S$ 
with at least two edges and an extreme exceptional vertex $S$ of multiplicity $2$. 

\item 
A trivial extension $T(B)$ of a tilted algebra $B$ of Dynkin type. 
\end{itemize}
For the definitions of a Brauer tree algebra, which we denote by $A(T^m_S)$ if it has an exceptional vertex $S$ 
of multiplicity $m\ge2$, and a modified Brauer tree algebra $D(T_S)$, 
see \cite[paragraphs 2.8 and 3.6]{S}. The assumptions on the Brauer trees are stated there. 
For the Brauer tree algebras, we know when it admits a cellular algebra structure as follows.

%%%%%

\begin{prop}[{\cite[Proposition 5.3]{KX1}}] 
\label{F Prop BT}
A Brauer tree algebra is cellular 
if and only if
the Brauer tree is a straight line with at most one exceptional vertex. 
\end{prop}

In subsequent subsections, we examine when a modified Brauer tree algebra or the trivial extension of 
a tilted algebra of Dynkin type admits a cellular algebra structure.

\remark 
In \cite[Proposition 5.3]{KX1}, 
Brauer trees are allowed to have an arbitrary number of exceptional vertices. 
However, in the above classification of nonsimple basic standard symmetric algebras of finite type, 
we consider only  Brauer trees with at most one exceptional vertex $S$. 

%%%%%

\para 
\label{F BD}
We consider a modified Brauer tree algebra $D(T_S)$. 
By the definition of $D(T_S)$, 
if the Brauer tree $T_S$ has a branch, 
then the quiver of $D(T_S)$ contains a cycle whose length is at least $3$, 
and this cycle has no inverse arrow. 
This contradicts Proposition \ref{C Prop simple} (\roiii) if $D(T_S)$ is cellular. 
Thus, we see that 
the Brauer tree $T_S$ is a straight line if $D(T_S)$ is cellular.   

Suppose that the Brauer tree $T_S$ is a straight line with $n$ edges. 
Then the quiver $Q_{T_S}$ of $D(T_S)$ is given by 
\begin{align}
\label{F QTS}
\xymatrix{
Q_{T_S} = & 1 \ar@(ul,dl)_{\g } \ar[r]<0.5 ex>^{\a_1}  & 2 \ar[r]<0.5 ex>^{\a_2}   \ar[l]<0.5 ex>^{\b_1} 
	& 3 \ar[r]<0.5 ex>^{\a_3}  \ar[l]<0.5 ex>^{\b_2} & \cdots \ar[r]<0.5ex>^{\a_{n-1}} \ar[l]<0.5 ex>^{\b_3} 
		& n \ar[l]<0.5 ex>^{\b_{n-1}}.
} 
\end{align} 
We define $I_{T_S}$ to be the two-sided ideal of $K Q_{T_S}$ generated by 
\begin{align}
\label{F ITS}
\begin{split}
&\a_i \a_{i+1}, \, \b_{i+1} \b_i \, (1 \leq i \leq n-2), \, 
\b_1 \a_1, \, \a_{n-1} \b_{n-1} \a_{n-1}, \, \b_{n-1} \a_{n-1} \b_{n-1}, \, 
\\
&\g^2 - \a_1 \b_1, \, \a_2 \b_2 - \b_1 \g \a_1, \, \b_i \a_i - \a_{i+1} \b_{i+1}\, (2 \leq  i \leq n-2). 
\end{split}
\end{align}
Then we have $D(T_S) = K Q_{T_S} / I_{T_S}$. 
In the case where $n=2$, we have the following lemma.

%%%%%

\begin{lem} 
\label{F Lemma DB 2}
Let $A=KQ / I$ where
\[
\xymatrix{Q = & 1 \ar@(ul,dl)_{\g } \ar[r]<0.5 ex>^{\a_1}  & 2 \ar[l]<0.5 ex>^{\b_1}}
\]
and $I$ is the two-sided ideal generated by 
$\{ \b_1 \a_1, \, \g^2 - \a_1 \b_1\}$. 
Then, $\dim A = 10$ and $A$ is not cellular.
\end{lem}

\begin{proof} 
We see that 
\begin{align*}
& A e_1 = K e_1 \oplus K \g \oplus K \b_1 \oplus K \g^2 \oplus K \b_1 \g \oplus K \g^3, 
\\
& A e_2 = K e_2 \oplus K \a_1 \oplus K \g \a_1 \oplus K \b_1 \g \a_1, 
\end{align*} 
and the Cartan matrix of $A$ is
\[
\bC= \begin{pmatrix} 4 & 2 \\ 2 & 2 \end{pmatrix}. 
\]
Suppose that $A$ is a cellular algebra with cell datum $(\vL, \cT, \cC, \iota)$. 
Let $\bD$ be the decomposition matrix of the cellular structure of $A$. 
Then, by Proposition \ref{C Prop decom matrix} (\roiii), 
we have 
\[
\bD^T \bD = \begin{pmatrix} 4 & 2 \\ 2 & 2 \end{pmatrix}. 
\]
As Proposition \ref{C Prop decom matrix} (\roi) tells us that every column of $D$ must contain $1$ as an entry, this equation implies that 
\begin{align*}
\bD= (d_{ij})_{1\leq i \leq 4, 1 \leq j \leq 2}=\begin{pmatrix} 1 & 1 \\ 1 & 1 \\ 1& 0 \\ 1 & 0 \end{pmatrix}
\end{align*} 
modulo rearrangement of rows.  
Let $\vL=\{\la_1,\la_2, \la_3,\la_4\}$ and $\vL^+=\{ \mu_1, \mu_2\}$ be  such that 
$[ \D(\la_i) : L(\mu_j) ] = d_{ij}$. If $\la_1$ or $\la_2$ was maximal or minimal with respect to the partial order, then 
the corresponding row of $D$ would be a unit vector by Lemma \ref{C Lem min}, a contradiction. Thus, 
we can assume that 
a total order on $\vL$ which extends the partial order is 
$3<1<2<4$. In particular, we do not have $\la_1>\la_2$ in the partial order. 
We also note that $\mu = \min\{ \la \in \vL \mid d_{\la\mu}\ne0 \}$, for $\mu\in \vL^+$, by 
Proposition \ref{C Prop decom matrix} (\roi). As $d_{\la_i\mu_2}\ne0$ only for $i=1,2$, $\mu_2\in\{\la_1,\la_2\}$, but
$\mu_2=\la_2$ contradicts $\la_1\not>\la_2$, so that $\mu_2=\la_1$. The minimality of $\la_3$ implies
$\mu_1=\la_3$ by Lemma \ref{C Lem min}(i).

We determine $A(\la_4)$ and $A(\la_2)$. By Proposition \ref{C Prop decom matrix} (\roii), 
we see that 
\begin{align}
\label{F DB 2}
A(\la_4) e_1 \cong P(\mu_1)_{\la_4} \cong \D(\la_4), 
\quad 
A(\la_4) e_2 \cong P(\mu_2)_{\la_4} =0. 
\end{align}
As $A (\la_4)$ is a two-sided-ideal and $\D(\la_4) \cong L(\mu_1)$, 
\eqref{F DB 2} implies that $A(\la_4) = K \g^3$. 
In a similar way, we see that 
\begin{align}
\label{F DB 2-2}
\begin{split}
&(A(\la_2)/ A(\la_4) )e_1 \cong P(\mu_1)_{\la_2} / P(\mu_1)_{\la_4} \cong \D(\la_2), 
\\
&(A(\la_2) / A(\la_4)) e_2 \cong P(\mu_2)_{\la_2} / P(\mu_2)_{\la_4} \cong \D(\la_2). 
\end{split}
\end{align}
Note that $A (\la_2)$ is a two-sided-ideal and $[ \D(\la_2) ] = [L(\mu_1)] + [L(\mu_2)]$.  
Then, since $\dim e_1(A(\la_2)/ A(\la_4) )e_1 = 1$, there is a nonzero element $ae_1+b\g+c\g^2$ in $A(\la_2)e_1$. If $a\ne0$ 
then $a\g+b\g^2 \in A(\la_2)e_1$, which implies $\dim e_1(A(\la_2)/ A(\la_4) )e_1>1$. 
Thus, $a=0$. The same argument shows $b=0$ and we conclude that $\g^2\in A(\la_2)e_1$. On the other hand, $\dim e_2(A(\la_2)/ A(\la_4) )e_1 = 1$
implies that there is a nonzero element $a\b_1+b\b_1\g$ in $A(\la_2)e_1$. 
We may also prove $\beta_1\gamma\alpha_1 \in A(\la_2)e_2$ and $\gamma\alpha_1 \in A(\la_2)e_2$ in the similar manner. Hence,
\begin{align*}
&(A(\la_2)/ A(\la_4)) e_1 \equiv K \g^2 \oplus K (a\b_1 + b\b_1\g) \mod A (\la_4), 
\\ 
&(A(\la_2)/ A(\la_4))e_2 \equiv K \g \a_1 \oplus K \b_1 \g \a_1 \mod A (\la_4).
\end{align*}
We observe that $\a_1$ acts as zero on $(A(\la_2)/ A(\la_4))e_2$. Then, $\a_1(a\b_1 + b\b_1\g) \equiv a\g^2$ implies $a=0$. 
However, if $a=0$ then we meet another contradiction
\begin{align*}
&\Soc \D(\la_2) \cong \Soc (A(\la_2)/ A(\la_4)) e_1 \cong L(\mu_1) \oplus L(\mu_2), 
\\ 
&\Soc \D(\la_2) \cong  \Soc (A(\la_2)/ A(\la_4))e_2 \cong  L(\mu_2).
\end{align*}
We have proved that $A$ is not cellular. 
\end{proof}

%%%%%

Lemma \ref{F Lemma DB 2} implies the following proposition. 

\begin{prop}
\label{F Prop MBD}
A modified Brauer tree algebra $D(T_S)$ 
associated to any Brauer tree $T_S$ 
with at least two edges and an extreme exceptional vertex $S$ of multiplicity $2$ 
is not cellular. 
\end{prop}
\begin{proof}
If $D (T_S)$ is a cellular algebra, 
the Brauer tree $T_S$ has to be a straight line by the argument in \ref{F BD}. 
On the other hand, 
if $T_S$ is a straight line with $n$ edges, 
then we have $D(T_S) = KQ_{T_S}/ I_{T_S}$, 
where $Q_{T_S}$ and $I_{T_S}$ are given by \eqref{F QTS} and \eqref{F ITS} respectively. 
Moreover, if we denote the algebra from Lemma \ref{F Lemma DB 2} by $A$,
we see that $(e_1+e_2)D(T_S)(e_1+e_2) \cong A$. 
Indeed, if we compute bases of $D(T_S)e_1$ and $D(T_S)e_2$ explicitly, we know that 
$\dim (e_1+e_2)D(T_S)(e_1+e_2) = 10$ 
and that we have a surjective algebra homomorphism $A \to (e_1+e_2)D(T_S)(e_1+e_2)$. 
Thus, $(e_1+e_2)D(T_S)(e_1+e_2)$ is not cellular by Lemma \ref{F Lemma DB 2}, 
so that $D(T_S)$ is not cellular by Lemma \ref{C Lem idem}. 
\end{proof}

%%%%%

\para 
We consider the trivial extension $T(B)$ of a tilted algebra $B$ of Dynkin type. 
By \cite[Corollary 3.4 in Chapter VIII]{ASS}, 
The quiver of $B$ is acyclic. 
Thus, we consider the trivial extension $T(A)$ of a finite dimensional basic algebra $A$ whose quiver is acyclic. 

%%%%%

\begin{lem}
\label{F Lemma TA}
Let $Q$ be a finite acyclic quiver, 
$I$ an admissible ideal of $K Q$. 
Put $A= KQ/I$. 
If the trivial extension $T(A)$ of $A$ is cellular, 
then we have $I= R^2_Q$, 
where $R_Q$ is the arrow ideal of $KQ$. 
\end{lem}

\begin{proof}
We consider a bounded quiver presentation of $T(A)$. 
For this, we need add new arrows to $Q$ which generate $\Hom_K(A,K)$. 
Then, we obtain 
the quiver $Q_{T(A)}$ of $T(A)$ from the quiver of $Q$ by the following 
recipe\footnote{Since $T(A)$ is the quotient of the repetitive algebra 
by the Nakayama automorphism $\nu_A$, 
we may also obtain the presentation by 
using the bounded quiver presentation of the repetitive algebra of $A$ 
given in \cite[Lemma 1.3]{A}.}: 
\begin{quote}
\it
Let $\a$ be a path in $Q$ such that $\a \not \in I$ and $\b \a, \a \b \in I$ for any arrow $\b$ in $Q$.
Then we add an arrow from $t(\a)$ to $s(\a)$, 
where $t(\a)$ is the target of $\a$ and $s(\a)$ is the source of $\a$.  
\end{quote}
Thanks to this construction, 
if there exists a path $\a$ in $Q$ such that the length of $\a$ is at least $2$ and $\a \not\in I$, 
then $Q_{T(A)}$ contains a cycle whose length is at least $3$, 
and this cycle has no inverse arrow. 
This contradicts Proposition \ref{C Prop simple} (\roiii) if $T(A)$ is cellular. 
We have proved $R_Q^2\subset I$. 
The opposite inclusion is clear since $I$ is an admissible ideal.
\end{proof}

%%%%%

\remark 
In \cite{O}, Lemma \ref{F Lemma TA} is proved in more elementary but longer argument. 
We thank Professor O. Iyama for informing us \cite[Lemma 1.3]{A}.  

%%%%% 

\para 
Let $Q=(Q_0,Q_1)$ be a finite acyclic quiver, 
$R_Q$ the arrow ideal of $KQ$. 
We consider the trivial extension $T(KQ/ R_Q^2)$ of $KQ/R_Q^2$. 

By \cite[Lemma 1.3]{A} (see the proof of Lemma \ref{F Lemma TA}), 
we see that the quiver of $T(KQ/R_Q^2)$ is the double quiver 
$\wh{Q}=(\wh{Q}_0, \wh{Q}_1)$ of $Q$, 
namely $\wh{Q}_0=Q_0$ and $\wh{Q}_1$ is obtained from $Q_1$ by adding an additional arrow $\a^*$ 
for each $\a \in Q_1$ such that $s(\a^*) = t(\a)$ and $t(\a^*)=s (\a)$.  

Let $I_{\wh{Q}}$ be the two-sided ideal of $K \wh{Q}$ generated by 
\begin{align*}
&\{ \a \b, \, \b^* \a^* \mid \a, \b \in Q_1\} 
\cup 
\{ \a \b^*, \, \b^* \a \mid \a \not=\b \in Q_1 \} 
\\
&\cup 
\{ \a \a^* - \b^* \b \mid \a, \b \in Q_1 \text{ such that } s (\a) = t (\b)\} 
\\
&\cup 
\{ \a \a^* - \b \b^* \mid \a, \b \in Q_1 \text{ such that } s (\a) = s (\b)\} 
\\
&\cup 
\{ \a^* \a - \b^* \b \mid \a, \b \in Q_1 \text{ such that } t(\a) = t(\b)\}. 
\end{align*}
Then we have the following lemma. 

%%%%%

\begin{lem} 
\label{F Lemma KQ RQ2}
Let $Q=(Q_0,Q_1)$ be a finite acyclic quiver, 
and $R_Q$ be the arrow ideal of $KQ$. 
Then the trivial extension $T(KQ/ R_Q^2)$ of $KQ/ R_Q^2$ 
is isomorphic to $K \wh{Q}/ I_{\wh{Q}}$.
\end{lem}

\begin{proof}
Let $\{f_{i} \mid i \in Q_0\} \cup \{ f_{\a} \mid \a \in Q_1\} \subset \Hom_{K} (KQ/ R_Q^2, K)$ 
be the dual basis of the basis
$\{e_i \mid i \in Q_0\} \cup \{\a \in Q_1\}$ of $KQ/ R_Q^2$,
namely $f_{i}(e_j)=\d_{ij}$, $f_{i}(\a)=0$, $f_{\a}(e_i)=0$ and $f_{\a}(\b) = \d_{\a \b}$, 
for $i,j \in Q_0$ and $\a, \b \in Q_1$. 
Then
\begin{align*}
\{e_i \mid i \in Q_0\} \cup \{\a \in Q_1\} \cup \{f_{i} \mid i \in Q_0\} \cup \{ f_{\a} \mid \a \in Q_1\}
\end{align*}
gives a basis of $T(KQ/ R_Q^2)$. 
Checking the well-definedness by direct calculation,
we can define an algebra homomorphism $\Phi: K \wh{Q}/ I_{\wh{Q}} \ra T (KQ/ R_Q^2)$ by
\begin{align*} 
e_i \mapsto e_i \, (i \in Q_0), \, 
\a \mapsto \a, \, \a^* \mapsto f_{\a} \, (\a \in Q_1). 
\end{align*} 
In particular, we used
$\a f_{\a} = f_{s(\a)}$ and $f_\a \a = f_{t (\a)}$ in $T(KQ/ R_Q^2)$, for $\a \in Q_1$, 
in the calculation, and they also show that
$\Phi$ is surjective. 
By comparing the dimensions of $K \wh{Q}/ I_{\wh{Q}}$ and $T(KQ/ R_Q^2)$, 
we see that $\Phi$ is an isomorphism.  
\end{proof} 

%%%%

Note that the algebra $K \wh{Q}/ I_{\wh{Q}}$ is determined only by the underlying graph of $Q$. 
We denote the underlying graph of $Q$ by $\ol{Q}$. 
Then, we have the following proposition. 

%%%%%

\begin{prop} 
\label{F Prop hatQ hatI}
Let $Q=(Q_0,Q_1)$ be a finite connected acyclic quiver. 
Then, 
the algebra 
$K \wh{Q}/ I_{\wh{Q}}$ is cellular 
if and only if 
the underlying graph $\ol{Q}$ of $Q$ is a straight line with no multiple edge.
Furthermore, if the condition holds then
$K \wh{Q}/ I_{\wh{Q}}$ is isomorphic to the Brauer tree algebra whose
Brauer tree is a straight line with no exceptional vertex.
\end{prop}

\begin{proof} 
Suppose that $K \wh{Q}/ I_{\wh{Q}}$ is cellular.
If $\ol{Q}$ has a multiple edge, 
then $Q$ contains a subquiver 
\[
\xymatrix{Q' =  i \ar[rr]<-1.5ex>   \ar[rr]<1.8ex> & \vdots  & j }
\]
with $k$ arrows ($k \geq 2$) by Proposition \ref{C Prop simple}.  
By the defining relations, 
we see that $(e_i + e_j)(K\wh{Q} / I_{\wh{Q}})(e_i + e_j) \cong K \wh{Q'} / I_{\wh{Q'}}$, 
and its Cartan matrix is
\[
\bC'=\begin{pmatrix} 2 & k \\ k & 2 \end{pmatrix}. 
\]
By Proposition \ref{C Prop decom matrix} (\roiv), 
$(e_i + e_j)(K\wh{Q} / I_{\wh{Q}})(e_i + e_j)$ is not cellular 
since $\det \bC' \leq 0$. 
Then, by Lemma \ref{C Lem idem}, 
$K \wh{Q}/ I_{\wh{Q}}$ is not cellular, contradicting our assumption. Hence
$\ol{Q}$ has no multiple edge. 

Note that $Q$ does not have an oriented cycle since $Q$ is acyclic. In particular 
the underlying graph of $Q$ does not contain a loop.
Suppose that $Q$ contains a subquiver $Q^{\sharp}$ 
whose underlying graph is a cycle of length $l\ge2$. We label its vertices as follows. 
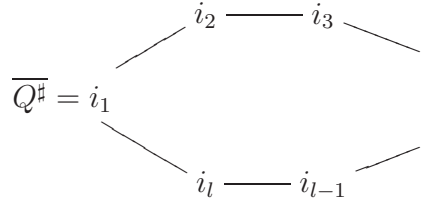
\begin{align*}
\xymatrix@R=3pt{
&i_2 \ar@{-}[r] & i_3 \ar@{-}[dr]  &
\\
&&& \ar@{.}[dd]
\\
\ol{Q^{\sharp}} = i_1\ar@{-}[uur] &&& & 
\\
&&& \ar@{-}[dl]
\\
& i_l \ar@{-}[uul]& i_{l-1} \ar@{-}[l] &
}
\end{align*}
Then we see that 
$(e_{i_1} + \dots + e_{i_l}) (K\wh{Q}/ I_{\wh{Q}}) (e_{i_1} + \dots + e_{i_l}) 
	\cong K \wh{Q^{\sharp}}/ I_{\wh{Q^{\sharp}}}$, 
and its Cartan matrix is 
\begin{align*}
\bC^{\sharp} = 
	\begin{pmatrix} 
	2 & 1 & 0 & 0 & 0 & \dots & 0 & 0 & 0 & 1 
	\\ 
	1 & 2 & 1 & 0 & 0 & \dots & 0 & 0 & 0 & 0 
	\\
	0 & 1 & 2 & 1 & 0 & \dots & 0 & 0 & 0 & 0 
	\\
	&&&&& \vdots &&&&
	\\
	0 & 0 & 0 & 0 & 0 & \dots & 0 & 1 & 2 & 1
	\\
	1 & 0 & 0 & 0 & 0 & \dots & 0 & 0 & 1 & 2
	\end{pmatrix}. 
\end{align*}
Suppose that $K \wh{Q^{\sharp}}/ I_{\wh{Q^{\sharp}}}$ is a cellular algebra with cell datum $(\vL,\cT, \cC, \iota)$,  
and let $\bD^{\sharp}$ be its decomposition matrix. 
Then, by Proposition \ref{C Prop decom matrix} (\roiii), 
we have 
\begin{align*}
\bD^{\sharp} = 
	\begin{pmatrix} 
	1 & 0 & 0 & 0 & \dots & 0 & 0 & 1 
	\\
	1 & 1 & 0 & 0 & \dots & 0 & 0 & 0 
	\\
	0 & 1 & 1 & 0 & \dots & 0 & 0 & 0 
	\\
	&&&& \vdots &&&
        \\
        0 & 0 & 0 & 0 & \dots & 1 & 1 & 0
	\\
	0 & 0 & 0 & 0 & \dots & 0 & 1 & 1 
	\end{pmatrix}. 
\end{align*}
modulo rearrangement of rows.

Put $\vL =\{ \la_1, \la_2,\dots, \la_{l}\}$
such that $([\D(\la_i) : L(\la_j)])_{1\leq i,j \leq l} = \bD^{\sharp}$.  
Then, by Proposition \ref{C Prop decom matrix} (\roi), 
we have 
$\la_l < \la_1 < \la_2 < \dots < \la_{l-1} < \la_l$ in the partial order, which is impossible. Hence, 
$(e_{i_1} + \dots + e_{i_l}) (K\wh{Q}/ I_{\wh{Q}}) (e_{i_1} + \dots + e_{i_l})$ 
is not cellular, and 
$K \wh{Q}/ I_{\wh{Q}}$ is not cellular by Lemma \ref{C Lem idem}, contradicting our assumption again.   

Assume that $\ol{Q}$ has no multiple edge and no cycle. 
If $\ol{Q}$ has a branch, 
$Q$ contains a subquiver $Q^{\natural}$ whose underlying graph is of the following form
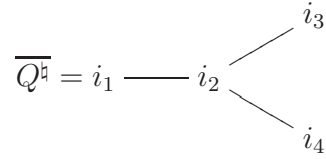
\begin{align*}
&\xymatrix@R=5pt{ & & i_3 \\ \ol{Q^{\natural}} = i_1 \ar@{-}[r] & i_2 \ar@{-}[ur] \ar@{-}[dr] & & \\ && i_4}
\end{align*}
We see that 
$(e_{i_1}+ \dots + e_{i_4}) K \wh{Q}/ I_{\wh{Q}} (e_{i_1}+ \dots + e_{i_4}) 
	\cong K \wh{Q^{\natural}}/ I_{\wh{Q^{\natural}}}$, 
and its Cartan matrix is 
\begin{align*}
\bC^{\natural} = 
	\begin{pmatrix} 
	2 & 1 & 0 & 0 
	\\
	1 & 2 & 1 & 1 
	\\
	0 & 1 & 2 & 0 
	\\
	0 & 1 & 0 & 2
	\end{pmatrix}.
\end{align*} 
However, 
there is no matrix $\bD$ with non-negative integral entries that satisfies 
$\bD^T \bD = \bC$. 
Thus, 
$(e_{i_1}+ \dots + e_{i_4}) K \wh{Q}/ I_{\wh{Q}} (e_{i_1}+ \dots + e_{i_4})$ 
is not cellular by Proposition \ref{C Prop decom matrix} (\roiii), and 
$K \wh{Q}/ I_{\wh{Q}}$ is not cellular by Lemma \ref{C Lem idem}. 

Now the only graph $\ol{Q}$ without cycles or loops, 
and branch points, is a straight line. Hence we have proved that 
$\ol{Q}$ is a straight line with no multiple edge if $K \wh{Q} / I_{\wh{Q}}$ is cellular. 

Conversely, suppose that $\ol{Q}$ is a straight line with no multiple edge. Then we easily see that 
$K \wh{Q}/ I_{\wh{Q}}$ is isomorphic to the Brauer tree algebra 
whose Brauer tree is a straight line  
with no exceptional vertex, which is cellular by Proposition \ref{F Prop BT}. 
\end{proof}

%%%%%

Note that 
the Brauer tree algebra whose Brauer tree is a straight line with no exceptional vertex is
isomorphic to the trivial extension of a tilted algebra of type $A_n$ for some $n\in\NN$.
Hence, we obtain the following proposition as a special case of Proposition \ref{F Prop hatQ hatI}. 

%%%%%

\begin{prop} 
\label{F Prop TE}
The trivial extension $T(B)$ of a tilted algebra $B$ of Dynkin type 
is cellular 
if and only if 
$T(B)$ is isomorphic to the Brauer tree algebra whose Brauer tree is 
a straight line with no exceptional vertex.
\end{prop} 

%%%%%

By \cite[Theorem 3.11]{S}, Proposition \ref{F Prop BT}, Proposition \ref{F Prop MBD} and Proposition \ref{F Prop TE}, 
we have the following proposition. 

%%%%%

\begin{prop}[{\cite{O}}] 
\label{F Prop symmetric}
An algebra $A$ is a basic symmetric cellular algebra of finite type 
if and only if 
$A$ is isomorphic  to a Brauer tree algebra
whose Brauer tree is a straight line with at most one exceptional vertex. 
\end{prop} 

%%%%%

Recall from Proposition \ref{C Prop Morita equiv} (\roii) that
a self-injective cellular algebra is weakly symmetric. 
If it is of finite type, we may apply the following result.

%%%%%

\begin{prop} 
\label{F Prop weak}
A weakly symmetric algebra over $K$ of finite type is symmetric. 
\end{prop}

\begin{proof}
Let $A$ be a basic weakly symmetric algebra over $K$ of finite type, 
and $A'$ be its Stamm-algebra (see \cite[Definition 4.4]{W} for the definition of Stamm-algebra). 
By \cite[Coroollary 4.7 (\roiii) in Part I]{W}, $A'$ is symmetric if and only if $A$ is weakly symmetric. 
If $A$ was not a symmetric algebra, then 
$A$ would be isomorphic to $A'$ by \cite[Theorem C]{KS}, 
which contradicts with $A'$ being a symmetric algebra.
\end{proof}

%%%%%

By Proposition \ref{C Prop Morita equiv} (\roii), 
Proposition \ref{F Prop symmetric}
and 
Proposition \ref{F Prop weak}, 
we have the following theorem. 

%%%%%

\begin{thm}\label{F Theorem}
Let $A$ be a basic algebra over $K$. 
The following are equivalent. 
\begin{enumerate} 
\item 
$A$ is a self-injective cellular algebra of finite type. 

\item 
$A$ is isomorphic  to a Brauer tree algebra 
whose Brauer tree is
a straight line with at most one exceptional vertex. 
\end{enumerate}
\end{thm}

%%%%%%%%%%%%%%%%%%%%%%%%%%%%%%%%%%%%%%%%%%%%%%%%%%%%%%%%%%%%%%%
\section{Self-injective cellular algebras of domestic type} 
%%%%%%%%%%%%%%%%%%%%%%%%%%%%%%%%%%%%%%%%%%%%%%%%%%%%%%%%%%%%%%%

 The aim of this section is to determine basic self-injective cellular algebras of domestic type. 
By Proposition \ref{C Prop decom matrix} and Proposition \ref{C Prop Morita equiv},  
it suffices to consider weakly symmetric algebras of domestic type whose Cartan matrices are nonsingular. 
 
\para {\bf Standard self-injective algebras of domestic type}
First of all, we classify standard self-injective cellular algebras of domestic type. 
Nonstandard self-injective algebras of domestic type will be treated in Proposition \ref{D Prop dom} and Lemma \ref{D nondom}.
The main result in this section is Theorem \ref{D main result}. 

By \cite[Theorem 4.3]{S}, a self-injective algebra is standard domestic of infinite type 
if and only if it is a self-injective algebra of Euclidean type.
Then, non-local weakly symmetric algebras of Euclidean type which have nonsingular Cartan matrices are classified 
in \cite[Theorem 1]{BS} and the result is recalled in \cite[Theorem 1.1]{BHS}. They are
algebras $\Lambda'(T), \Gamma^{(0)}(T,v), \Gamma^{(1)}(T,v), \Gamma^{(2)}(T,v_1,v_2)$ and $\Lambda(T,v_1,v_2)$. 
On the other hand, the classification of local self-injective algebras of Euclidean type is known and 
it is recalled in \cite[Theorem 1.2]{BHS}. They are algebras $A(\lambda)$ with $\lambda\in K\setminus\{0\}$. 

\begin{definition}
Let $\lambda \in K\backslash \{0\}$ and $Q$ the following quiver. 
\[\begin{xy}
(0,0)*[o]+{1}="1",
\SelectTips{eu}{}
\ar @(lu,ld)"1";"1"_{\alpha}
\ar @(ru,rd)"1";"1"^{\beta}
\end{xy}\] 
Then we define $A(\lambda) = KQ/I(\lambda)$, where $I(\lambda)$ is the admissible ideal generated 
by the elements $\alpha^2,\beta^2,\alpha\beta-\lambda\beta\alpha$.
\end{definition}

The five classes of non-local algebras are given by Brauer graphs satisfying some conditions. 
Our convention is that we read the cyclic ordering around each vertex clockwise when we draw a Brauer graph in a plane. 
The common assumption for the five classes of the Brauer graphs is that 
\begin{quote}
\it
the Brauer graphs have at most one cycle, which may be or may not be a loop, and if the cycle exists then it must be of odd length.
\end{quote}
Furthermore, we fix at most one distinguished vertex in the cycle. 
Given a Brauer graph $T$, we denote by $Q_{T}$ the Brauer quiver associated with $T$. 
The cyclic ordering of each vertex in $T$ defines two simple cycles in $Q_{T}$ if it is the endpoint of a loop, and
one simple cycle in $Q_{T}$ otherwise. In particular, 
if a distinguished vertex of $T$ is chosen, then it defines one or two distinguished simple cycles, say $\gamma$-cycle or $\delta$-cycle, in $Q_{T}$, 
and we divide the remaining simple cycles of $Q_{T}$ into $\alpha$-parts and $\beta$-parts 
in such a way that if two cycles intersect at a vertex in $Q_{T}$, then the two cycles belong to different parts. 
The cycles in the $\alpha$-parts (respectively, the $\beta$-parts) 
are called $\alpha$-cycles (respectively, $\beta$-cycles). 
We denote an arrow in the $\alpha$-parts (respectively, the $\beta$-parts, the $\gamma$-cycle, the $\delta$-cycle) 
which starts at a vertex $i$ by $\alpha _{i}$ (respectively, $\beta_{i}$, $\gamma_{i}$, $\delta_{i}$), 
and we denote the target of $\alpha_{i}$ (respectively, $\beta_{i}$, $\gamma_{i}$, $\delta_{i}$) by $\alpha(i)$ (respectively, $\beta(i)$, $\gamma(i)$, $\delta(i)$). We will use the symbol $A_{i}$ (respectively, $B_{i}$, $C_{i}$, $D_{i}$) to denote the $\alpha$-cycle (respectively, $\beta$-cycle, $\gamma$-cycle, $\delta$-cycle) from $i$ to $i$, that is,
\[ \begin{array}{cc}
A_{i}=\alpha_{i}\alpha_{\alpha(i)}\cdots\alpha_{\alpha^{-1}(i)}, &  B_{i}=\beta_{i}\beta_{\beta(i)}\cdots\beta_{\beta^{-1}(i)} ,\\
C_{i}=\gamma_{i}\gamma_{\gamma(i)}\cdots\gamma_{\gamma^{-1}(i)}, &  D_{i}=\delta_{i}\delta_{\delta(i)}\cdots\delta_{\delta^{-1}(i)}.
\end{array}\]

We are ready to introduce the algebras $\Lambda'(T), \Gamma^{(0)}(T,v), \Gamma^{(1)}(T,v), \Gamma^{(2)}(T,v_1,v_2)$ and $\Lambda(T,v_1,v_2)$.
Note however that although we use Brauer graphs, they are not necessarily Brauer graph algebras and 
there are some modifications in the defining relations.

\begin{definition}\label{D defn I}
Let $T$ be a Brauer graph with exactly one cycle, which is of an odd length, and we define $\Lambda'(T)$. 

Suppose that the cycle is not a loop, and that a distinguished vertex $v$ on the cycle is chosen. 
We call the simple cycle in $Q_T$ associated with the cyclic ordering around $v$ the $\gamma$-cycle. 
Then, $\Lambda'(T) = KQ_T/I'(T)$, where $I'(T)$ is the ideal generated by the following elements:
\begin{enumerate}
\item $\alpha_i\beta_{\alpha(i)},$ $\beta_i\alpha_{\beta(i)},$ $\alpha_i\gamma_{\alpha(i)},$ $\gamma_i\alpha_{\gamma(i)},$  
$\gamma_i\beta_{\gamma(i)},$ $\beta_i\gamma_{\beta(i)}$, where $i$ are vertices of $Q_T$,
\item $A_j-B_j$, where $j$ is the intersection of an $\alpha$-cycle and a $\beta$-cycle,
\item $A_j-C_j$, where $j$ is the intersection of an $\alpha$-cycle and the $\gamma$-cycle,
\item $B_j-C_j$, where $j$ is the intersection of a $\beta$-cycle and the $\gamma$-cycle.
\end{enumerate} 

Suppose that the cycle is a loop, and the vertex of the loop is chosen as the distinguished vertex $v$. 
Then we have distinguished simple cycles the $\gamma$-cycle and the $\delta$-cycle in $Q_{T}$ 
which are associated with the cyclic ordering around $v$. 
We denote the intersection of the $\gamma$-cycle and the $\delta$-cycle by $a$.
Then, $\Lambda'(T) = KQ_T/I'(T)$, where $I'(T)$ is the ideal generated by the following elements:
\begin{enumerate}
\item $\alpha_i\beta_{\alpha(i)},$ $\beta_i\alpha_{\beta(i)},$ $\alpha_i\gamma_{\alpha(i)},$ $\gamma_i\alpha_{\gamma(i)},$ 
$\gamma_i\beta_{\gamma(i)},$ $\beta_i\gamma_{\beta(i)},$ $\alpha_i\delta_{\alpha(i)},$ $\delta_i\alpha_{\delta(i)},$  
$\beta_i\delta_{\beta(i)},$ $\delta_i\beta_{\delta(i)}$, where $i$ are vertices of $Q_T$,
\item $A_j-B_j$, where $j$ is the intersection of an $\alpha$-cycle and a $\beta$-cycle,
\item $A_j-\gamma_j\gamma_{\gamma(j)}\cdots\gamma_{\gamma^{-1}(a)}D_{a}\gamma_{a}\cdots\gamma_{\gamma^{-1}(j)}$, where $j$ is the intersection of an $\alpha$-cycle and the $\gamma$-cycle,
\item $A_j-\delta_j\delta_{\delta(j)}\cdots\delta_{\delta^{-1}(a)}C_{a}\delta_{a}\cdots\delta_{\delta^{-1}(j)}$, where $j$ is the intersection of an $\alpha$-cycle and the $\delta$-cycle,
\item $B_j-\gamma_j\gamma_{\gamma(j)}\cdots\gamma_{\gamma^{-1}(a)}D_{a}\gamma_{a}\cdots\gamma_{\gamma^{-1}(j)}$, where $j$ is the intersection of an $\beta$-cycle and the $\gamma$-cycle,
\item $B_j-\delta_j\delta_{\delta(j)}\cdots\delta_{\delta^{-1}(a)}C_{a}\delta_{a}\cdots\delta_{\delta^{-1}(j)}$, where $j$ is the intersection of an $\beta$-cycle and the $\delta$-cycle,
\item $\gamma_{\gamma^{-1}(a)}\gamma_{a}$, $\delta_{\delta^{-1}(a)}\delta_{a}$,
\item $C_{a}D_{a} - D_{a}C_{a}$.
\end{enumerate}
\end{definition}

\begin{definition}\label{D defn II}
Let $T$ be a Brauer graph with exactly one loop. 
We denote the loop and its vertex by $b$ and $u$ respectively. 
We assume that there is an edge $a$ of $T$ which is different from $b$ such that
\begin{itemize}
\item
the endpoints of $a$ are $u$ and $v$,
\item
no other edge than $a$ has $v$ as an endpoint,
\item
the edges $a$ and $b$ are direct successors of the other in the cyclic ordering around the vertex $u$,
\item
the loop $b$ is not a direct successor of itself in the cyclic ordering around the vertex $u$.
\end{itemize}
Then, we choose $u$ as the distinguished vertex and we define $\Gamma^{(0)}(T,v)$. 

Let two simple cycles in $Q_T$ associated with the cyclic ordering around $u$ be 
the $\gamma$-cycle and the $\delta$-cycle. They intersect at the vertex $b$ of $Q_T$. 
By the second condition, there is the unique loop in $Q_T$ which has the vertex $a$ as the endpoint. 
We assume that this loop belongs to $\alpha$-parts. We also assume that 
the simple cycle $b\to a \to b$ is the $\delta$-cycle, so that $D_b=\delta_b\delta_a$. Then,
$\Gamma^{(0)}(T,v) = KQ_T/I^{(0)}(T,v)$, where $I^{(0) }(T,v)$ is the ideal generated by the following elements:
\begin{enumerate}
\item $\alpha_i\beta_{\alpha(i)},$ $\beta_i\alpha_{\beta(i)},$ $\alpha_i\gamma_{\alpha(i)},$ $\gamma_i\alpha_{\gamma(i)},$  
$\gamma_i\beta_{\gamma(i)},$ $\beta_i\gamma_{\beta(i)}$, where $i$ are vertices of $Q_T$, and 
$\alpha_a\delta_a,$ $\delta_b\alpha_a$,
\item $A_j-B_j$, where $j$ is the intersection of an $\alpha$-cycle and a $\beta$-cycle,
\item $A_j-\gamma_j\gamma_{\gamma(j)}\cdots\gamma_{\gamma^{-1}(b)}D_b\gamma_b\cdots\gamma_{\gamma^{-1}(j)}$, 
where $j$ is the intersection of an $\alpha$-cycle and the $\gamma$-cycle,
\item $\alpha_{a} - \delta_{a}C_{b}\delta_{b}$,
\item $B_j-\gamma_j\gamma_{\gamma(j)}\cdots\gamma_{\gamma^{-1}(b)}D_b\gamma_b\cdots\gamma_{\gamma^{-1}(j)}$, 
where $j$ is the intersection of an $\beta$-cycle and the $\gamma$-cycle,
\item $\gamma_{\gamma^{-1}(b)}\gamma_{b}$,
\item $\delta_{a}C_{b}-\delta_{a}D_{b}$,
\item $C_{b}\delta_{b}-D_{b}\delta_{b}$.
\end{enumerate}
\end{definition}

\begin{definition}\label{D defn III}
Let $T$ be a Brauer graph with a unique exceptional vertex $v$, whose multiplicity is two, 
and a unique cycle, which is of length three. When we draw the Brauer graph in a plane, 
we assume that
\begin{itemize}
\item
the label of the edges of the cycle is given as
 
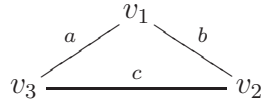
\[\begin{xy}
(0,0)*[o]+{v_{3}}="2",(30,0)*[o]+{v_{2}}="3",(15,10)*[o]+{v_{1}}="1",
\ar @{-}"1";"2"_{a}
\ar @{-}"2";"3"^{c}
\ar @{-}"1";"3"^{b}
\end{xy}\] 
and other edges are located outside the triangle, that is, 
$a$ (respectively, $b$, $c$) is the direct successor of $b$ (respectively, $c$, $a$) 
in the cyclic ordering around $v_{1}$ (respectively, $v_{2}$, $v_{3}$),
\item
among the vertices $v_{1}$, $v_{2}$ and $v_{3}$, $v_{3}$ is the nearest vertex
to the exceptional vertex $v$, where $v=v_{3}$ may occur. 
\end{itemize}
Then, we choose $v_1$ as the distinguished vertex and we define $\Gamma^{(1)}(T,v)$.

Note that there is no cycle outside the triangle by the common assumption which we made for the Brauer graphs. 
We call the simple cycle in $Q_{T}$ associated with the cyclic ordering around $v_{1}$ the $\gamma$-cycle, 
and assume that the simple cycle associated with $v_{2}$ (respectively, $v_{3}$) belongs to $\beta$-parts (respectively, $\alpha$-parts). 
We call the simple cycle associated with the cyclic ordering around $v$ \emph{the exceptional cycle}. 
Then, $\Gamma^{(1)}(T,v) = KQ_T/I^{(1)}(T,v)$, where $I^{(1)}(T,v)$ is the ideal 
generated by the following elements:
\begin{enumerate}
\item $\beta_i\alpha_{\beta(i)},$ $\alpha_i\gamma_{\alpha(i)},$ $\gamma_i\beta_{\gamma(i)}$, where $i$ are vertices of $Q_T$,
\item $\alpha_{i}\beta_{\alpha(i)}$, for $i\ne a$,
\item $\gamma_{i}\alpha_{\gamma(i)}$, for $i\ne b$,
\item $\beta_{i}\gamma_{\beta(i)}$, for $i\ne c$,
\item $A_j-B_j$, if the intersection of an $\alpha$-cycle and a $\beta$-cycle is a vertex $j\neq c$ and the two cycles are not exceptional,
\item $A_j-C_j$, if the intersection of an $\alpha$-cycle and the $\gamma$-cycle is a vertex $j\neq a$ and the two cycles are not exceptional,
\item $B_j-C_j$, if the intersection of a $\beta$-cycle and the $\gamma$-cycle is a vertex $j\neq b$ and the two cycles are not exceptional, 
\item $A_j^{2}-B_j$, if the intersection of an $\alpha$-cycle and a $\beta$-cycle is a vertex $j\neq c$ and the $\alpha$-cycle is exceptional,
\item $A_j-B_j^{2}$, if the intersection of an $\alpha$-cycle and a $\beta$-cycle is a vertex $j$ and the $\beta$-cycle is exceptional,
\item $\gamma_{\gamma^{-1}(b)}\gamma_{b}\alpha_{a},$ $\alpha_{\alpha^{-1}(a)}\alpha_{a}\beta_{c}, \beta_{\beta^{-1}(c)}\beta_{c}\gamma_{b}$,
\item $\beta_{b}\beta_{\beta(b)}\cdots\beta_{\beta^{-1}(c)}-\gamma_{b}\alpha_{a},$ $\gamma_{a}\gamma_{\gamma(a)}\cdots\gamma_{\gamma^{-1}(b)}-\alpha_{a}\beta_{c}$,
\item $\alpha_{c}\alpha_{\alpha(c)}\cdots\alpha_{\alpha^{-1}(a)}-\beta_{c}\gamma_{b}$, if the $\alpha$-cycle through the vertex $c$ is not exceptional,
\item $\alpha_{c}\alpha_{\alpha(c)}\cdots\alpha_{\alpha^{-1}(a)}A_{a}-\beta_{c}\gamma_{b}$, if the $\alpha$-cycle through the vertex $c$ is exceptional.
\end{enumerate}
\end{definition}
\remark
In fact, we have $j\ne a, b, c$ in the conditions (v)-(ix): if $j=a$ or $b$ in (v) then one of the cycles is the $\gamma$-cycle, that is, 
it is not an $\alpha$-cycle nor a $\beta$-cycle. The same reasoning excludes $j=b, c$ in (vi), $j=a, c$ in (vii) and $j=a, b$ in (viii). 
We have $j\ne c$ in (ix) because if $j=c$ then the $\beta$-cycle is not exceptional. 
Another remark is that the condition for 
$\alpha_i\beta_{\alpha(i)}$ (respectively, $\gamma_i\alpha_{\gamma(i)}$, $\beta_i\gamma_{\beta(i)}$) appears in (ii) and (xi) 
(respectively, (iii) and (xi), (iv) and (xii), (xiii)).

\begin{definition}\label{D defn IV}
Let $T$ be a Brauer tree  $T$ (in the sense of \cite[Proposition 5.3]{KX1}) with two different exceptional vertices $v_{1}$ and  $v_{2}$ of multiplicity two. 
Then, we do not specify the distingushed vertex, so that there are no $\gamma$-cycle and $\delta$-cycle, 
and we define $\Lambda(T,v_{1},v_{2})$.

We call simple cycles associated with the cyclic ordering around the exceptional vertices $v_{1}$ and $v_{2}$ 
\emph{exceptional cycles}. Then, $\Lambda(T,v_{1},v_{2}) = KQ_T/I(T,v_{1},v_{2})$, 
where $I(T,v_{1},v_{2})$ is the ideal generated by the following elements:
\begin{enumerate}
\item $\alpha_{i}\beta_{\alpha(i)},$ $\beta_{i}\alpha_{\beta(i)}$, where $i$ are vertices of $Q_{T}$,
\item $A_j-B_j$, if the intersection of an $\alpha$-cycle and a $\beta$-cycle is a vertex $j$ and 
the two cycles are not exceptional,
\item $A_j^{2}-B_j$, if the intersection of an $\alpha$-cycle and a $\beta$-cycle is a vertex $j$ and 
only the $\alpha$-cycle is exceptional, 
\item $A_j-B_j^{2}$, if the intersection of an $\alpha$-cycle and a $\beta$-cycle is a vertex $j$ and 
only the $\beta$-cycle is exceptional,
\item $A_j^{2}-B_j^{2}$, if the intersection of an $\alpha$-cycle and a $\beta$-cycle is a vertex $j$ and 
both the $\alpha$-cycle and the $\beta$-cycle are exceptional.
\end{enumerate}
\end{definition}

\begin{definition}\label{D defn V}
Let $T$ be a Brauer tree with 
a unique exceptional vertex $v_{2}$
\footnote{In page 51 of \cite{BHS}, the definition starts with stating that $v_1$ and $v_2$ are exceptional vertices. 
However, it is stated in the middle of the page that the exceptional cycle is the one which is associated with the cyclic ordering around $v_2$. 
Thus, only $v_2$ is exceptional, and we should follow the definition in \cite{BS}.}
, which is of multiplicity two. 
We assume the following:
\begin{itemize}
\item
There is an edge $a$ of $T$ which has endpoints $v_1$ and $u$.
\item
The vertex $v_1$ is different from $v_2$ and $v_1$ is not an endpoint of any other edge than $a$.
\item
Let $b$ be the direct successor of $a$ and $c$ the direct predecessor of $a$ in the cyclic ordering around $u$. 
Then, $b\ne a$ and $c\ne a$.
\item
The vertices $u$, $v_{1}$, $v_{2}$ and the edges $a$, $b$, $c$ determine a subtree 
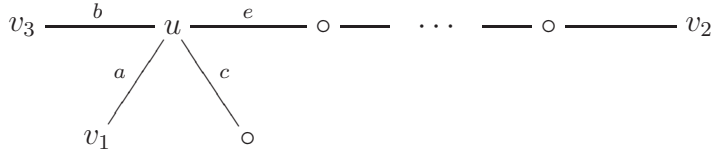
\[ 
 \begin{xy}
(0,0)*[o]+{v_{3}}="3",(20,0)*[o]+{u}="u",(10,-15)*[o]+{v_{1}}="1",(30,-15)*[o]+{\circ}="000",(40,0)*[o]+{\circ}="001",(50,0)*[o]+{}="dammy1",
(55,0)*[o]+{\cdots}="cdots",(60,0)*[o]+{}="dammy2",(70,0)*[o]+{\circ}="002",(90,0)*[o]+{v_{2}}="2"
\ar @{-}"3";"u"^{b}
\ar @{-}"1";"u"^{a}
\ar @{-}"000";"u"_{c}
\ar @{-}"u";"001"^{e}
\ar @{-}"001";"dammy1"
\ar @{-}"dammy2";"002"
\ar @{-}"002";"2"
\end{xy}
\]
of the Brauer tree $T$, where one of $u=v_{2}$, $v_{3}=v_{2}$ may occur (but cannot be both as $u\ne v_3$), 
and any of $b=e$, $c=e$ and $b=c=e$ may hold. 
\end{itemize}
Then, we do not specify the distinguished vertex and we define $\Gamma^{(2)}(T,v_{1},v_{2})$.

We call the simple cycle of $Q_T$ associated with the cyclic ordering around the vertex $v_2$ \emph{the exceptional cycle}. 
We assume that the simple cycle in $Q_{T}$ associated with the cyclic ordering around the vertex $u$ is an $\alpha$-cycle. 
Then, $\Gamma^{(2)}(T,v_{1},v_{2}) = KQ_T^{(2)}/I^{(2)}(T,v_{1},v_{2})$, where the vertices and the arrows of $Q_{T}^{(2)}$ are given by
\[ Q_{T}^{(2)}=((Q_{T})_{0}\cup \{ w\},\ (Q_{T})_{1}\cup\{ \gamma_{1}:c\to w,\ \gamma_{2}: w\to b,\ \gamma_{3}:w\to w\}),\]
and $I^{(2)}(T,v_{1},v_{2})$ is the ideal generated by the following elements:
\begin{enumerate}
\item $\alpha_{i}\beta_{\alpha(i)}$, $\beta_{i}\alpha_{\beta(i)}$, for all vertices $i$ of $Q_{T}$,
\item $A_j-B_j$, if the intersection of an $\alpha$-cycle and a $\beta$-cycle is a vertex $j$ and the two cycles are not exceptional,
\item $A_j^{2}-B_j$, if the intersection of an $\alpha$-cycle and a $\beta$-cycle is a vertex $j$ and only the $\alpha$-cycle is exceptional, 
\item $A_j-B_j^{2}$, if the intersection of an $\alpha$-cycle and a $\beta$-cycle is a vertex $j$ and only the $\beta$-cycle is exceptional,
\item $\gamma_{2}\beta_{b}$, $\beta_{\beta^{-1}(c)}\gamma_{1}$, $\gamma_{1}\gamma_{3}$, $\gamma_{3}\gamma_{2}$,
\item $\gamma_{2}\alpha_{b}\cdots \alpha_{c}$, $\alpha_{a}\alpha_{b}\cdots\alpha_{\alpha^{-1}(c)}\gamma_{1}$, if $u\ne v_2$, 
\\
(We understand that $\gamma_{2}\alpha_{b}$, $\alpha_{a}\gamma_{1}$, if $u\ne v_2$ and $b=c=e$),
\item $\gamma_{2}\alpha_{b}\cdots \alpha_{\alpha^{-1}(c)}\gamma_{1}\gamma_{2}\alpha_{b}\cdots\alpha_{c}$, 
$\alpha_{a}\alpha_{b}\cdots\alpha_{\alpha^{-1}(c)}\alpha_c\alpha_a\alpha_{b}\cdots\alpha_{\alpha^{-1}(c)}\gamma_{1}$, if $u=v_2$,
\\ (We understand that $\gamma_{2}\gamma_{1}\gamma_{2}\alpha_{b}$, $\alpha_{a}\alpha_{b}\alpha_{a}\gamma_{1}$, if $u=v_2$ and $b=c=e$),
\item $\alpha_{c}\alpha_{a}-\gamma_{1}\gamma_{2}$,
\item $\gamma_{2}\alpha_{b}\alpha_{\alpha(b)}\cdots\alpha_{\alpha^{-1}(c)}\gamma_{1}-\gamma_{3}$, if $u\ne v_2$,
\item $(\gamma_{2}\alpha_{b}\alpha_{\alpha(b)}\cdots\alpha_{\alpha^{-1}(c)}\gamma_{1})^{2}-\gamma_{3}$, if $u=v_2$.
\end{enumerate}
\end{definition}

Now, we recall the classification of weakly symmetric algebras of Euclidean type which have nonsingular Cartan matrices. 

 \begin{thm}[{\cite[Theorem 1,1 and Theorem 1.2]{BHS}}]\label{D thm class}
 Let $A$ be a basic self-injective algebra with a nonsingular Cartan matrix. Then the following statements hold:
 \begin{enumerate}
 \item[(1)]  $A$ is a local self-injective algebra of Euclidean type if and only if $A$ is isomorphic to $A(\lambda)$ for some $\lambda\in K\backslash \{0\}$.
 \item[(2)] $A$ is a non-local weakly symmetric algebra of Euclidean type if and only if $A$ is isomorphic to an algebra of the form $\Lambda (T,v_1,v_2)$, $\Lambda'(T)$, $\Gamma^{(0)}(T,v)$, $\Gamma^{(1)}(T,v)$, or $\Gamma^{(2)}(T,v_1,v_2)$.
 \end{enumerate}
 \end{thm}

From now on, we classify standard self-injective cellular algebras of domestic type. 
By Theorem \ref{D thm class}, 
it suffices to check whether an algebra from the six classes of algebras 
$A(\lambda)$, $\Lambda (T,v_1,v_2)$, $\Lambda'(T)$, $\Gamma^{(0)}(T,v)$, $\Gamma^{(1)}(T,v)$, and $\Gamma^{(2)}(T,v_1,v_2)$
admits a cellular algebra structure.

%%%%

\para
First, we consider the local case. Let $A$ be a local self-injective algebra of Euclidean type whose Cartan matrix is nonsingular. 
By Theorem \ref{D thm class} (1), we may assume $A=A(\lambda)$ for some $\lambda \in K\backslash\{0\}$. 

Note that $A(1)$ is isomorphic to the Kronecker algebra $K[X,Y]/(X^2,Y^2)$. 
Then, the Kronecker algebra is cellular. To see this, 
let the anti-involution $\iota$ be the identity map, 
$\Lambda=\{\la_1<\la_2<\la_3<\la_4\}$ 
and $\cT(\la_i)=\{1\}$ for $i=1,2,3,4$. Then,
\[
(c^{\la_1}_{1,1})=(1),\quad (c^{\la_2}_{1,1})=(X), \quad (c^{\la_3}_{1,1})=(Y), \quad (c^{\la_4}_{1,1})=(XY)
\]
is a cellular basis.

%%%%

\begin{prop}
\label{D prop local}
The algebra $A(\lambda)$ is cellular if and only if $\lambda=1$. 
In other words, any standard local self-injective cellular algebras of domestic type is isomorphic to the Kronecker algebra $K[X,Y]/(X^2,Y^2)$. 
\end{prop} 
\begin{proof}
Recall that $A(\lambda)=KQ/I$, where 
\[Q=\begin{xy}
(0,0)*[o]+{1}="1",
\SelectTips{eu}{}
\ar @(lu,ld)"1";"1"_{\alpha}
\ar @(ru,rd)"1";"1"^{\beta}
\end{xy}\] 
and $I$ is the admissible ideal generated by $\alpha^{2},$ $\beta^{2},$ $\alpha\beta-\lambda\beta\alpha$.

If $\lambda=1$, then $A(1)$ is cellular. 
Suppose that $A(\lambda)$ is a cellular algebra with cell datum $(\Lambda,\cT,\cC,\iota)$. 
Since $\lambda\neq 0$ and $A$ is local, $A(\lambda)$ is the unique indecomposable projective module, and its basis is given by
\[ A(\la) = Ke_1\oplus K\alpha \oplus K\beta\oplus K\alpha\beta. \]
Thus, the Cartan matrix of $A(\lambda)$ is $(4)$. 
Let $\bD$ be the decomposition matrix of $A(\lambda)$. 
By (i) and (iii) from Proposition \ref{C Prop decom matrix}, we have 
$\bD=(1\ 1\ 1\ 1)^T$. It implies that $\Lambda=\{\lambda_1, \lambda_2, \lambda_3, \lambda_4\}$ and 
we may assume without loss of generality that $\lambda_1$ is minimal in the partial order, and we may extend the partial order 
to the total order $1 < 2 < 3 < 4$. The equality $\dim A(\lambda)=\sum_{i=1}^4 |\cT(\lambda_i)|^{2}$ implies that we may set
\[ \cT(\lambda_1)=\cT(\lambda_2)=\cT(\lambda_3)=\cT(\lambda_4)=\{1\}. \]
The cellular basis is $\cC=\{c^{\lambda_i}_{1,1}\ |\ i=1,2,3,4\}$ and $\iota(c^{\lambda_i}_{1,1})=c^{\lambda_i}_{1,1}$
implies that the anti-involution $\iota$ fixes all elements of $A(\lambda)$. It follows that
\[ \alpha\beta=\iota(\alpha\beta)=\iota(\lambda\beta\alpha)=\lambda\iota(\alpha)\iota(\beta)=\lambda\alpha\beta. \]
Therefore, we obtain $\lambda=1$.
\end{proof}

%%%%

\para
Next, we consider the non-local case. We begin by the algebras $\Lambda'(T)$ from Definition \ref{D defn I}.
Suppose that the Brauer graph has a cycle which is not a loop. 
Since $T$ does not have a loop, Proposition \ref{C Prop simple}(iii) implies that $T$ does not have a branch. 
Hence, $T$ coincides with the cycle of odd length. Its Brauer quiver $Q_{T}$ is given by
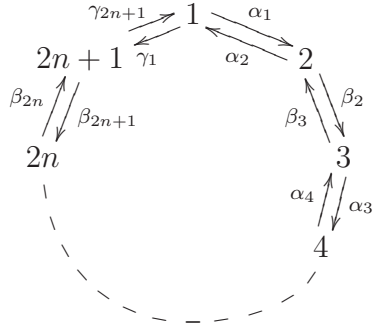
\[
\begin{xy}
(0,11)*[o]+{1}="0",(15,5)*[o]+{2}="1",(20,-8)*[o]+{3}="2", (-15,5)*[oo]+{2n+1}="3",(-20,-8)*[o]+{2n}="4",
(17,-20)*[o]+{4}="5",(0,-30)*{}="dammy",
\ar @<1mm> "0";"1"^{\alpha_1}
\ar @<1mm> "1";"0"^{\alpha_2}
\ar @<1mm> "1";"2"^{\beta_2}
\ar @<1mm> "2";"1"^{\beta_3}
\ar @<1mm> "0";"3"^{\gamma_{1}}
\ar @<1mm> "3";"0"^{\gamma_{2n+1}}
\ar @<1mm> "4";"3"^{\beta_{2n}}
\ar @<1mm> "3";"4"^{\beta_{2n+1}}
\ar @<1mm> "2";"5"^{\alpha_{3}}
\ar @<1mm> "5";"2"^{\alpha_{4}}
\ar @{--}@(l,d) "dammy";"4"
\ar @{--}@(l,d) "dammy";"5"
\end{xy}
\]
and the admissible ideal is generated by the following elements.
\begin{enumerate}
\item $\alpha_{2k-1}\beta_{2k}$, $\beta_{2k+1}\alpha_{2k}$, for $1\leq k\leq n$, 
\item $\alpha_{2k+1}\alpha_{2k+2}-\beta_{2k+1}\beta_{2k}$, for $1\leq k\leq n-1$,
\item$\beta_{2n}\gamma_{2n+1}$, $\gamma_1\beta_{2n+1}$, $\gamma_{2n+1}\alpha_1$ and $\alpha_2\gamma_1$, 
\item $\alpha_{2k}\alpha_{2k-1}-\beta_{2k}\beta_{2k+1}$, for $1\leq k\leq n$,
\item $\alpha_{1}\alpha_{2}-\gamma_{1}\gamma_{2n+1}$ and $\beta_{2n+1}\beta_{2n}-\gamma_{2n+1}\gamma_{1}$.
\end{enumerate}

\begin{lem}
\label{D lem I-1}
Let $T$ be a Brauer graph as above. Then the algebra $\Lambda'(T)$ is not cellular.
\end{lem}
\begin{proof}
We see that the Cartan matrix of $\Lambda'(T)$ is
\[ \bC = \begin{pmatrix}
2 & 1 & 0 & 0 & \cdots & 0 & 0 & 0 & 1 \\ 
1 & 2 & 1 & 0 & \cdots & 0 & 0 & 0 & 0 \\
0 & 1 & 2 & 1 & \cdots & 0 & 0 & 0 & 0 \\ 
 &  &  &  & \vdots &  &  &  &  \\ 
0 & 0 & 0 & 0 & \cdots & 0 & 1 & 2 & 1 \\ 
1 & 0 & 0 & 0 & \cdots & 0 & 0 & 1 & 2 
\end{pmatrix}. \]
The same Cartan matrix appeared in the proof of Proposition \ref{F Prop hatQ hatI} and 
the argument there shows that $\Lambda'(T)$ cannot be cellular.
\end{proof}

%%%%
\para
Suppose that the Brauer graph has a loop. Then, Proposition \ref{C Prop simple}(iii) implies that
the Brauer graph $T$ is of the form:
\begin{equation}
\label{D I-2 Brauer graph}
\begin{xy}
(0,0)*[o]+{\circ}="0",(10,0)*[o]+{\circ}="1",(20,0)*[o]+{\circ}="2", (30,0)*[o]+{\cdots}="d1",(40,0)*[o]+{\circ}="3",
(50,0)*[o]+{v}="v",(60,0)*[o]+{\circ}="4",(70,0)*[o]+{\circ}="5", (80,0)*[o]+{\cdots}="d3",(90,0)*[o]+{\circ}="6",
(-10,0)*{}="dammy",(100,0)*[o]+{\circ}="7"
\ar @{-}"0";"1"^{-l}
\ar @{-}"1";"2"^{-(l-1)}
\ar @{-}"2";"d1"
\ar @{-}"d1";"3"
\ar @{-}"3";"v"^{-1}
\ar @{-}"v";"4"^{1}
\ar @{-}"4";"5"^{2}
\ar @{-}"5";"d3"
\ar @{-}"d3";"6"
\ar @{-}"6";"7"^{m}
\ar @{-} @(u,u)"v";"dammy"_{0}
\ar @{-} @(d,d)"dammy";"v"
\end{xy}
\end{equation}
with no exceptional vertex.
We denote the above Brauer graph by $T(l,m)$. We introduce the corresponding Brauer quiver as follows.

\begin{definition}
For $l,m\geq 0$, we denote the following quiver by $Q(l,m)$.
\[
\begin{xy}
(-55,0)*[o]+{-l}="-l",(-43,0)*[o]+{}="-l+1", (-40,0)*{\cdots}="d1",(-36,0)*{}="-3",(-25,0)*[o]+{-2}="-2",(-12,0)*[o]+{-1}="-1",
(0,0)*[o]+{0}="0",(10,0)*[o]+{1}="1",(20,0)*[o]+{2}="2", (30,0)*[o]+{}="3",(34,0)*{\cdots}="d1",(38,0)*[o]+{}="m-1",
(50,0)*[o]+{m}="m"
\ar @<1mm> "-l";"-l+1"^{\alpha_{-l}}
\ar @<1mm> "-l+1";"-l"^{\beta_{-l+1}}
\ar @<1mm> "-3";"-2"^{\alpha_{-3}}
\ar @<1mm> "-2";"-3"^{\beta_{-2}}
\ar @<1mm> "-2";"-1"^{\alpha_{-2}}
\ar @<1mm> "-1";"-2"^{\beta_{-1}}
\ar @<1mm> "-1";"0"^{\delta_{-1}}
\ar @<1mm> "0";"-1"^{\delta_{0}}
\ar @<1mm> "0";"1"^{\gamma_{0}}
\ar @<1mm> "1";"0"^{\gamma_{1}}
\ar @<1mm> "1";"2"^{\alpha_{1}}
\ar @<1mm> "2";"1"^{\beta_{2}}
\ar @<1mm> "2";"3"^{\alpha_{2}}
\ar @<1mm> "3";"2"^{\beta_{3}}
\ar @<1mm> "m-1";"m"^{\alpha_{m-1}}
\ar @<1mm> "m";"m-1"^{\beta_{m}}
\SelectTips{eu}{}
\ar @(dl,ul)"-l";"-l"^{\beta_{-l}}
\ar @(ur,dr)"m";"m"^{\alpha_m}
\end{xy}
\]
Here, $\beta_{0}=\delta_{0}$ if $l=0$ and $\alpha_0=\gamma_0$ if $m=0$. 
We label arrows other than the $\gamma$-cycle and the $\delta$-cycle with $\alpha_i$ and $\beta_i$, but 
we do not follow the convention to divide the arrows into $\alpha$-parts and $\beta$-parts in such a way that 
each of the relevant vertices of the Brauer graph is the intersection of an $\alpha$-cycle and a $\beta$-cycle.
\end{definition}

We write $\Lambda'(T(l,m))=KQ(l,m)/I$ and recall the generators of the ideal $I$ for the following five cases. 
Then, from those explicit generators, it is easy to see that $\Lambda'(T(l,m))\cong \Lambda'(T(m,l))$, so that the five cases suffice. 

\begin{enumerate}
\item[(a)] $m\ge l\ge 2$ or $m>l\geq 1$.
\item[(b)] $m=l=1$.
\item[(c)] $l=0$, $m>1$.
\item[(d)] $l=0$, $m=1$.
\item[(e)] $l=m=0$.
\end{enumerate}

%%%%
\para
Suppose that we are in the case (a). Then $I$ is the ideal generated by the following elements:
\begin{enumerate}
\item $\alpha_i\alpha_{i+1}$, for $-l\leq i\leq -3$ or $1\leq i\leq m-1$ (if $l=1$, then $1\leq i\leq m-1$), 
\item $\beta_i\beta_{i-1}$, for $-l+1\leq i\leq -1$ or $3\leq i\leq m$ (if $l=1$, then $3\leq i\leq m$), 
\item $\alpha_m\beta_m$, $\gamma_0\alpha_1$, $\beta_2\gamma_1$ and $\delta_0\beta_{-1}$,
\item $\beta_{-l}\alpha_{-l}$, $\alpha_{-2}\delta_{-1}$ (if $l=1$, then $\beta_{-1}\delta_{-1}$),
\item $\alpha_i\beta_{i+1}-\beta_i\alpha_{i-1}$, for $-l+1\leq i\leq -2$ or $2\leq i\leq m-1$ (if $l=1$, then $2\leq i\leq m-1$), 
\item $\alpha_m-\beta_m\alpha_{m-1}$ and $\beta_{-l}-\alpha_{-l}\beta_{-l+1}$, (if $l=1$, then $\beta_{-l}-\alpha_{-l}\beta_{-l+1}$ does not occur.)
\item $\alpha_{1}\beta_{2}-\gamma_{1}\delta_{0}\delta_{-1}\gamma_{0}$,
\item $\beta_{-1}\alpha_{-2}-\delta_{-1}\gamma_0\gamma_1\delta_0$ (if $l=1$, then $\beta_{-1}-\delta_{-1}\gamma_0\gamma_1\delta_0$),
\item $\gamma_{1}\gamma_{0}$, $\delta_{-1}\delta_{0}$, $\gamma_{0}\gamma_{1}\delta_{0}\delta_{-1}-\delta_{0}\delta_{-1}\gamma_{0}\gamma_{1}$.
\end{enumerate}
By using (vi) or (viii), we delete the arrows $\alpha_{m}$ and $\beta_{-l}$ from $Q(l,m)$ 
to obtain the Gabriel quiver, 
and replace $\alpha_m$ (respectively, $\beta_{-l}$) in the above generators of $I$ with 
$\beta_m\alpha_{m-1}$ (respectively, $\alpha_{-l}\beta_{-l+1}$ or $\delta_{-1}\gamma_0\gamma_1\delta_0$ if $l=1$). 
Explicit computation shows that the indecomposable projective modules are given by
\[\begin{array}{lcl}
P_{-l} &=& Ke_{-l}\oplus K\beta_{-l+1} \oplus K\alpha_{-l}\beta_{-l+1} \;\text{(if $l\ge 2$)}, \\
P_{i} &=& Ke_{i} \oplus  K\alpha_{i-1} \oplus K\beta_{i+1}\oplus K\alpha_{i}\beta_{i+1} \quad (i\notin \{-l, -1,0,1,m\}), \\
P_{-1} &=& \begin{cases}
Ke_{-1}\oplus K\alpha_{-2}\oplus K\delta_{0} \oplus K\gamma_{1}\delta_{0}
\oplus K\gamma_{0}\gamma_{1}\delta_{0} \oplus K\delta_{-1}\gamma_{0}\gamma_1\delta_0 \;&\text{(if $l\ge2$)},\\
Ke_{-1}\oplus K\delta_{0} \oplus K\gamma_{1}\delta_{0}
\oplus K\gamma_{0}\gamma_{1}\delta_{0} \oplus K\delta_{-1}\gamma_{0}\gamma_1\delta_0 \;&\text{(if $l=1$)},
\end{cases}\\
P_{0} &=& Ke_0\oplus K\delta_{-1}\oplus K\gamma_{1} \oplus K\delta_{0}\delta_{-1} \oplus K\gamma_{0}\gamma_{1}\oplus K\gamma_{1}\delta_{0}\delta_{-1} \\
        & & \oplus K\delta_{-1}\gamma_{0}\gamma_{1}\oplus K\gamma_0\gamma_1\delta_0\delta_{-1}, \\
P_1  &=& Ke_1\oplus K\beta_2\oplus K\gamma_0\oplus K\delta_{-1}\gamma_0\oplus K\delta_{0}\delta_{-1}\gamma_0\oplus K\gamma_{1}\delta_0\delta_{-1}\gamma_{0}, \\
P_m &=& Ke_m\oplus K\alpha_{m-1} \oplus K\beta_m\alpha_{m-1}.     
\end{array} \]

Suppose that we are in the case (b). Then $I$ is the ideal generated by the following elements:
\begin{enumerate}
\item  $\alpha_1\gamma_1$, $\gamma_0\alpha_1$, $\beta_{-1}\delta_{-1}$ and $\delta_0\beta_{-1}$, 
\item $\alpha_{1}-\gamma_{1}\delta_{0}\delta_{-1}\gamma_{0}$ and $\beta_{-1}-\delta_{-1}\gamma_0\gamma_1\delta_0$,
\item $\gamma_{1}\gamma_{0}$, $\delta_{-1}\delta_{0}$, $\gamma_{0}\gamma_{1}\delta_{0}\delta_{-1}-\delta_{0}\delta_{-1}\gamma_{0}\gamma_{1}$.
\end{enumerate}
By using (ii), we delete the arrows $\alpha_{1}$ and $\beta_{-1}$ from $Q(1,1)$ to obtain the Gabriel quiver  
and replace $\alpha_1$ (respectively, $\beta_{-1}$) in the above generators of $I$ 
by $\gamma_{1}\delta_{0}\delta_{1}\gamma_{0}$ (respectively, $\delta_{-1}\gamma_0\gamma_1\delta_0$). 
Explicit computation shows 
\[\begin{array}{lcl}
P_{-1} &=&Ke_{-1}\oplus K\delta_{0}\oplus K\gamma_1\delta_{0} \oplus K\gamma_0\gamma_{1}\delta_{0}\oplus K\delta_{-1}\gamma_{0}\gamma_{1}\delta_{0}, \\
P_{0} &=& Ke_0\oplus K\delta_{-1}\oplus K\gamma_{1} \oplus K\delta_{0}\delta_{-1} \oplus K\gamma_{0}\gamma_{1}\oplus K\gamma_{1}\delta_{0}\delta_{-1} \\
        & & \oplus K\delta_{-1}\gamma_{0}\gamma_{1}\oplus K\gamma_0\gamma_1\delta_0\delta_{-1}, \\
P_1  &=& Ke_1\oplus K\gamma_0\oplus K\delta_{-1}\gamma_0\oplus K\delta_{0}\delta_{-1}\gamma_0\oplus K\gamma_{1}\delta_0\delta_{-1}\gamma_{0}. \\   
\end{array} \]

\begin{lem}
\label{D I-2 (a)}
If $m\ge l\geq 1$ then the algebra $\Lambda'(T(l,m))$ is not cellular.
\end{lem}
\begin{proof}
Put $e=e_{-1}+e_0+e_1$. By Lemma \ref{C Lem idem}, it suffices to prove that $e\Lambda'(T(l,m))e$ is not cellular. 
Suppose to the contrary that $e\Lambda'(T(l,m))e$ is cellular. 
The computation of indecomposable projective $\Lambda'(T(l,m))$-modules in the cases (a) and (b) 
shows that the Cartan matrix of $e\Lambda'(T(l,m))e$ is
\[
\bC=\begin{pmatrix}
2 & 2 & 1 \\
2 & 4 & 2 \\
1 & 2 & 2 \end{pmatrix}.
\]
Then, the decomposition matrix $\bD$, which satisfies $\bD^T\bD=\bC$, is
\[
\bD=\begin{pmatrix}
 1 & 1 & 0\\
 1 & 1 & 1\\
 0 & 1 & 1\\
 0 & 1 & 0\end{pmatrix}
\]
modulo rearrangement of rows. Let $\Lambda=\{\lambda_1, \lambda_2, \lambda_3, \lambda_4\}\supset \Lambda^+=\{\mu_1,\mu_2,\mu_3\}$ 
such that $[\Delta(\lambda_i):L(\mu_j)]=d_{ij}$. 
Since there are more than three nonzero entries in $\bD$, there exist pairwise distinct elements in $\Lambda$ such that 
one is minimal and the other is maximal in the partial order. 
However, Lemma \ref{C Lem min} tells us that the rows which correspond to the minimal elements and the maximal elements in the partial order 
must have unique nonzero entries, so that they must coincide, a contradiction. 
\end{proof}

%%%%
\para
Assume that we are in the case (c). Then $\Lambda'(T(0,m))=KQ(0,m)/I$, where 
$I$ is the ideal generated by the following elements:
\begin{enumerate}
\item $\alpha_i\alpha_{i+1}$, for $1\leq i\leq m-1$, 
\item $\beta_i\beta_{i-1}$, for  $3\leq i\leq m$, 
\item $\alpha_m\beta_m$, $\gamma_0\alpha_1$ and $\beta_2\gamma_1$,
\item $\alpha_i\beta_{i+1}-\beta_i\alpha_{i-1}$, for $2\leq i\leq m-1$, 
\item $\alpha_m-\beta_m\alpha_{m-1}$,
\item $\alpha_{1}\beta_{2}-\gamma_{1}\delta_{0}\gamma_{0}$,
\item $\gamma_{1}\gamma_{0}$, $\delta_{0}^2$, $\gamma_{0}\gamma_{1}\delta_{0}-\delta_{0}\gamma_{0}\gamma_{1}$.
\end{enumerate}
By using (v), we delete the arrows $\alpha_{m}$ from $Q(0,m)$ to obtain the Gabriel quiver  
and replace $\alpha_m$ in the above generators of $I$ by $\beta_m\alpha_{m-1}$. Then, we have 
\[\begin{array}{lcl}
P_{0} &=& Ke_0\oplus K\delta_{0}\oplus K\gamma_{1} \oplus K\gamma_1\delta_0\oplus K\gamma_{0}\gamma_{1}\oplus K\gamma_{0}\gamma_{1}\delta_{0}, \\
P_1  &=& Ke_1\oplus K\beta_2\oplus K\gamma_0\oplus K\delta_{0}\gamma_0\oplus K\gamma_1\delta_0\gamma_0, \\
P_{i} &=& Ke_{i} \oplus  K\alpha_{i-1} \oplus K\beta_{i+1}\oplus K\alpha_{i}\beta_{i+1} \quad (2\leq i \leq m-1), \\
P_m &=& Ke_m\oplus K\alpha_{m-1} \oplus K\beta_m\alpha_{m-1}.     
\end{array} \]

\begin{lem}
\label{D I-2 (c)}
If $m>1$, then the algebra $\Lambda'(T(0,m))$ is not cellular.
\end{lem}
\begin{proof}
Put $e=e_0+e_1+e_2$ and suppose that $\Lambda'(T(0,m))$ is cellular. Then, $e\Lambda'(T(0,m))e$ is cellular by Lemma \ref{C Lem idem}. 
The computation of indecomposable projective $\Lambda'(T(0,m))$-modules 
shows that the Cartan matrix of $e\Lambda'(T(0,m))e$ is 
\[ 
\bC=\begin{pmatrix}
4 & 2 & 0  \\
2 & 2 & 1  \\
0 & 1 & 2  \\
\end{pmatrix},
\]
and there should exist the decomposition matrix $\bD$ satisfying $\bD^T\bD=\bC$. 
However, we find that $\bD$ does not exist, a contradiction.
\end{proof}

In the last two cases (d) and (e), the following lemma holds. 
Note that the algebra $\Lambda'(T(0,0))$ is generated by $\gamma_0$ and $\delta_0$ subject to the relations
\[
\gamma_0^2=0, \;\; \delta_0^2=0, \;\; \gamma_0\delta_0=\delta_0\gamma_0, 
\]
so that $\Lambda'(T(0,0))\cong K[X,Y]/(X^2, Y^2)$. The quiver $Q(0,1)$ is given by
\[
\begin{xy}
(0,0)*[o]+{0}="A", (15,0)*[o]+{1}="B",
\ar @<1mm> "A";"B"^{\gamma_0}
\ar @<1mm> "B";"A"^{\gamma_1}
\SelectTips{eu}{}
\ar @(ul,dl) "A";"A"_{\delta_0}
\ar @(ur,dr) "B";"B"^{\alpha_1}
\end{xy}
\]
and $\Lambda'(T(0,1))=KQ(0,1)/I$, where $I$ is the ideal generated by 
\begin{enumerate}
\item $\gamma_0\alpha_1$ and $\alpha_1\gamma_1$,
\item $\gamma_1\gamma_0$ and $\delta_0^2$,
\item $\alpha_1 - \gamma_1\delta_0\gamma_0$,
\item $\delta_0\gamma_0\gamma_1-\gamma_0\gamma_1\delta_0$.
\end{enumerate}
By using (iii), we delete the arrow $\alpha_1$ of $Q(0,1)$ to obtain the Gabriel quiver  
and replace $\alpha_1$ in the above generators of $I$ with $\gamma_1\delta_0\gamma_0$. 
Explicit computation shows 
\[\begin{array}{lcl}
P_{0} &=& Ke_0\oplus K\delta_{0}\oplus K\gamma_{1} \oplus K\gamma_1\delta_0\oplus K\gamma_{0}\gamma_{1}\oplus K\gamma_{0}\gamma_{1}\delta_{0}, \\
P_1  &=& Ke_1\oplus K\gamma_0\oplus K\delta_{0}\gamma_0\oplus K\gamma_1\delta_0\gamma_0.      
\end{array} \]

\begin{lem}
\label{D I-2 (d)(e)}
The algebras $\Lambda'(T(0,1))$ and $\Lambda'(T(0,0))$ are cellular.
\end{lem}
\begin{proof}
The algebra $\Lambda'(T(0,0))\cong K[X,Y]/(X^2,Y^2)$ is a cellular algebra as we have proved before. 
Next, we construct a cellular basis of $\Lambda'(T(0,1))$. 
First of all, we can see that the anti-involution on $KQ(0,1)$ defined by
$e_0^*=e_0$, $e_1^*=e_1$, $\delta_0^{\ast}=\delta_0$, $\gamma_0^{\ast}=\gamma_1$ and $\gamma_1^{\ast}=\gamma_0$ 
preserves the ideal $I$. Thus, it induces an anti-involution $\iota:\Lambda'(T(0,1)) \to \Lambda'(T(0,1))$. 
Let $\Lambda=\{\lambda_1 < \lambda_2 < \lambda_3 < \lambda_4\}$ and define
\[  \cT(\lambda_1)=\cT(\lambda_4)=\{1\}, \quad \cT(\lambda_2)=\cT(\lambda_3)=\{1,2\}. \]
Then the following $K$-basis of $\Lambda'(T(0,1))$ is a cellular basis:
\begin{align*}
 & \qquad (c^{\lambda_1}_{1,1})=(e_0), \quad (c^{\lambda_2}_{i,j})_{i,j\in\cT(\la_2)}=
\begin{pmatrix} e_1 & \gamma_1 \\ \gamma_0 & \gamma_0\gamma_1 \end{pmatrix}, \\
&(c^{\lambda_3}_{i,j})_{i,j\in\cT(\la_3)}=
\begin{pmatrix} \delta_0 & \delta_0\gamma_0 \\ \gamma_1\delta_0 & \gamma_1\delta_0\gamma_0 \end{pmatrix},  
\quad (c^{\lambda_4}_{1,1})=(\gamma_0\gamma_1\delta_0). 
\end{align*}
The conditions (C1) and (C2) are easy to verify. To see that the condition (C3) holds, 
it suffices to observe that 
$K\delta_0\oplus K\gamma_1\delta_0\oplus K\gamma_0\gamma_1\delta_0$ and $K\gamma_0\gamma_1\delta_0$ are 
submodules of $P_0$.
Therefore, $\Lambda'(T(0,1))$ is cellular.
\end{proof}

By Lemma \ref{D lem I-1}, Lemma \ref{D I-2 (a)}, Lemma \ref{D I-2 (c)} and Lemma \ref{D I-2 (d)(e)},
we obtain the following proposition.
\begin{prop}
\label{D I-2 cellular}
The algebra $\Lambda'(T)$ is cellular if and only if $T$ is one of $T(0,1)$, $T(1,0)$ and $T(0,0)$.
\end{prop}

%%%%
\para
We consider the algebras $\Gamma^{(0)}(T,v)$ from Definition \ref{D defn II}. Then, by Proposition \ref{C Prop simple}(iii), 
the Brauer graph $T$ is one of the following $T(m)$ or $T(-m)$:
\[
\label{D II Brauer graph}
T(m)=\begin{xy}
(0,0)*[o]+{v}="3", (10,0)*[o]+{u}="v",(20,0)*[o]+{\circ}="4", (30,0)*[o]+{\cdots}="d3",(40,0)*[o]+{\circ}="6",
(-5,0)*{}="dammy",(50,0)*[o]+{\circ}="7"
\ar @{-}"3";"v"^{a}
\ar @{-}"v";"4"^{1}
\ar @{-}"4";"d3"
\ar @{-}"d3";"6"
\ar @{-}"6";"7"^{m}
\ar @{-} @(u,u)"v";"dammy"_{b}
\ar @{-} @(d,d)"dammy";"v"
\end{xy} \;\; T(-m)=\begin{xy}
(0,0)*[o]+{\circ}="0",(10,0)*[o]+{\circ}="1",(20,0)*[o]+{\cdots}="d1",(30,0)*[o]+{\circ}="3",
(40,0)*[o]+{u}="v",(50,0)*[o]+{v}="4",
(-5,0)*{}="dammy",
\ar @{-}"0";"1"^{m}
\ar @{-}"1";"d1"
\ar @{-}"d1";"3"
\ar @{-}"3";"v"^{1}
\ar @{-}"v";"4"^{a}
\ar @{-} @(u,u)"v";"dammy"_{b}
\ar @{-} @(d,d)"dammy";"v"
\end{xy}
\]
where $m\ge 1$. The defining relations for $\Gamma^{(0)}(T(m),v)$, which we will give below, show that  
there exists an isomorphism of algebras $\Gamma^{(0)}(T(m),v)\cong \Gamma^{(0)}(T(-m),v)$. 
Hence, it suffices to consider $T=T(m)$ for $m\ge 1$. We denote the corresponding Brauer quiver by $Q(m)$, which we label 
the arrows as follows:
\[
\begin{xy}
(-12,0)*[o]+{a}="-1",
(0,0)*[o]+{b}="0",(10,0)*[o]+{1}="1",(20,0)*[o]+{2}="2", (30,0)*[o]+{}="3",(34,0)*{\cdots}="d1",(38,0)*[o]+{}="m-1",
(50,0)*[o]+{m}="m"
\ar @<1mm> "-1";"0"^{\delta_a}
\ar @<1mm> "0";"-1"^{\delta_b}
\ar @<1mm> "0";"1"^{\gamma_b}
\ar @<1mm> "1";"0"^{\gamma_{1}}
\ar @<1mm> "1";"2"^{\alpha_{1}}
\ar @<1mm> "2";"1"^{\beta_{2}}
\ar @<1mm> "2";"3"^{\alpha_{2}}
\ar @<1mm> "3";"2"^{\beta_{3}}
\ar @<1mm> "m-1";"m"^{\alpha_{m-1}}
\ar @<1mm> "m";"m-1"^{\beta_{m}}
\SelectTips{eu}{}
\ar @(ul,dl)"-1";"-1"_{\alpha_a}
\ar @(ur,dr)"m";"m"^{\alpha_m}
\end{xy}
\]

Suppose that $m>1$. Then $\Gamma^{(0)}(T(m),v)=KQ(m)/I$, where $I$ is the ideal generated by the following elements:
\begin{enumerate}
\item $\alpha_{i}\alpha_{i+1}$, for $1\leq i\leq m-1$, 
\item $\beta_{i}\beta_{i-1}$, for $3\leq i\leq m$, 
\item $\alpha_{m}\beta_{m}$, $\beta_{2}\gamma_{1}$, $\gamma_b\alpha_1$, $\alpha_a\delta_a$, $\delta_b\alpha_a$,
\item $\alpha_{i}\beta_{i+1}-\beta_{i}\alpha_{i-1}$, for $2\leq i\leq m-1$, 
\item $\alpha_{m}-\beta_{m}\alpha_{m-1}$ and $\alpha_a-\delta_a\gamma_b\gamma_{1}\delta_b$,
\item $\alpha_{1}\beta_{2}-\gamma_{1}\delta_{b}\delta_{a}\gamma_{b}$,
\item $\gamma_{1}\gamma_{b}$,
\item $\delta_{a}\gamma_{b}\gamma_{1}-\delta_{a}\delta_{b}\delta_{a}$,
\item $\gamma_{b}\gamma_{1}\delta_{b}-\delta_{b}\delta_{a}\delta_{b}.$
\end{enumerate}
By using (v), we delete the arrows $\alpha_m$ and $\alpha_a$ from $Q(m)$ to obtain the Gabriel quiver  
and replace $\alpha_m$ (respectively, $\alpha_a$) in the above generators of 
$I$ with $\beta_{m}\alpha_{m-1}$ (respectively, $\delta_{a}\gamma_{b}\gamma_{1}\delta_{b}$). 
Then explicit computation shows 
\[\begin{array}{lcl}
P_{a} &=&Ke_{a}\oplus K\delta_{b} \oplus K\gamma_{1}\delta_{b}\oplus K\delta_{a}\delta_{b} 
\oplus K\gamma_{b}\gamma_{1}\delta_{b} \oplus K\delta_{a}\gamma_{b}\gamma_1\delta_{b}, \\
P_{b} &=& Ke_{b}\oplus K\delta_{a}\oplus K\gamma_{1} \oplus K\delta_{b}\delta_{a} 
\oplus K\gamma_{b}\gamma_{1}\oplus K\gamma_{1}\delta_{b}\delta_{a} \\
        & & \oplus K\delta_{a}\gamma_{b}\gamma_{1}\oplus K\gamma_{b}\gamma_1\delta_{b}\delta_{a}, \\
P_1  &=& Ke_1\oplus K\beta_2\oplus K\gamma_{b}\oplus K\delta_{a}\gamma_{b}
\oplus K\delta_{b}\delta_{a}\gamma_{b}\oplus K\gamma_{1}\delta_{b}\delta_{a}\gamma_{b}, \\
P_{i} &=& Ke_{i} \oplus  K\alpha_{i-1} \oplus K\beta_{i+1}\oplus K\alpha_{i}\beta_{i+1} \quad (i\notin \{-1,0,1,m\}), \\
P_m &=& Ke_m\oplus K\alpha_{m-1} \oplus K\beta_m\alpha_{m-1}.     
\end{array} \]

Suppose that $m=1$. Then, $\Gamma^{(0)}(T(1),v)=KQ/I$, where $Q$ is 
\[
Q=\begin{xy}
(0,0)*[o]+{a}="-1",(15,0)*[o]+{b}="0",(30,0)*[o]+{1}="1", 
\ar @<1mm> "0";"-1"^{\delta_{b}}
\ar @<1mm> "-1";"0"^{\delta_{a}}
\ar @<1mm> "0";"1"^{\gamma_{b}}
\ar @<1mm> "1";"0"^{\gamma_{1}}
\end{xy}
\]
and $I$ is the admissible ideal generated by
\begin{enumerate}
\item $\gamma_{b}\gamma_1\delta_{b}\delta_{a}\gamma_{b}$, 
$\gamma_1\delta_{b}\delta_{a}\gamma_{b}\gamma_1$, 
$\delta_{a}\gamma_{b}\gamma_1\delta_{b}\delta_{a}$, $\delta_{b}\delta_{a}\gamma_{b}\gamma_1\delta_{b}$,
\item $\gamma_1\gamma_{b}$,
\item $\delta_{a}\gamma_{b}\gamma_1-\delta_{a}\delta_{b}\delta_{a}$,
\item $\gamma_{b}\gamma_1\delta_{b}-\delta_{b}\delta_{a}\delta_{b}$.
\end{enumerate}
Moreover, the indecompsable projective modules are given by
\[\begin{array}{lcl}
P_a &=&Ke_a\oplus K\delta_b \oplus K\gamma_{1}\delta_b \oplus K\delta_a\delta_b \oplus K\delta_b\delta_a\delta_b \oplus K\delta_a\delta_b\delta_a\delta_b, \\
P_b &=& Ke_b \oplus K\delta_a \oplus K\gamma_{1} \oplus K\delta_b\delta_a 
\oplus K\gamma_b\gamma_{1} \oplus K\gamma_{1}\delta_b\delta_a \\
        & & \oplus K\delta_a\delta_b\delta_a \oplus K\delta_b\delta_a\delta_b\delta_a, \\
P_1  &=& Ke_1\oplus K\gamma_b \oplus K\delta_a\gamma_b \oplus K\delta_b\delta_a\gamma_b \oplus K\gamma_{1}\delta_b\delta_a\gamma_b.
\end{array} \]

\begin{prop}
\label{D II}
The algebra $\Gamma^{(0)}(T(m),v)$ is cellular if and only if $m=1$.
\end{prop}
\begin{proof} 
Assume that $m>1$ and put $e=e_a+e_b+e_1+e_2$. Then the Cartan matrix of $e\Gamma^{(0)}(T(m),v)e$ is 
\[ \bC=\begin{pmatrix}
3 & 2 & 1 & 0 \\
2 & 4 & 2 & 0 \\
1 & 2 & 2 & 1  \\
0 & 0 & 1 & 2 
\end{pmatrix}. \]
Suppose that $\Gamma^{(0)}(T(m),v)$ is cellular. Then so is $e\Gamma^{(0)}(T(m),v)e$ by Lemma \ref{C Lem idem}. 
However, we find that there is no matrix $\bD$ with entries in $\ZZ_{\ge0}$ that satisfies the condition $\bD^T\bD=\bC$, a contradiction.

Next, we show that the algebra $\Gamma^{(0)}(T(1),v)$ is cellular. It is easy to see that the anti-involution $\iota$  
induced by swapping $\delta_{a}$ and $\delta_{b}$ (respectively, $\gamma_{b}$ and $\gamma_1$) is well-defined. 
We take a totally ordered set 
$\Lambda=\{\lambda_1<\lambda_2<\lambda_3<\lambda_4<\lambda_5\}$ and define $\cT(\lambda_k)$, for $1 \le k \le 5$, by
\[  \cT(\lambda_1)=\cT(\lambda_5)=\{1\}, \quad \cT(\lambda_2)=\cT(\lambda_4)=\{1,2\},\quad \cT(\lambda_3)=\{1,2,3\}. \]
For each $\lambda_k\in\Lambda$, we define $(c^{\lambda_k}_{i,j})_{i,j\in\cT(k)}$ as follows:
\begin{align*}
& (c^{\lambda_1}_{1,1})=(e_{b}), \quad (c^{\lambda_2}_{i,j})_{i,j\in\cT(\la_2)}=
\begin{pmatrix} e_1 & \gamma_1 \\ \gamma_{b} & \gamma_{b}\gamma_1 \end{pmatrix}, \\[7pt]
& (c^{\lambda_3}_{i,j})_{i,j\in\cT(\la_3)}=
\begin{pmatrix} e_{a} & \delta_{a} & \delta_{a}\gamma_{b} \\ \delta_{b} & \delta_{b}\delta_{a} & \delta_{b}\delta_{a}\gamma_{b} \\ 
\gamma_1\delta_{b} & \gamma_1\delta_{b}\delta_{a} & \gamma_1\delta_{b}\delta_{a}\gamma_{b} \end{pmatrix}, \\[7pt]
& (c^{\lambda_4}_{i,j})_{i,j\in\cT(\la_4)}=
\begin{pmatrix} \delta_{a}\delta_{b} & \delta_{a}\delta_{b}\delta_{a} \\ \delta_{b}\delta_{a}\delta_{b} & \delta_{b}\delta_{a}\delta_{b}\delta_{a} \end{pmatrix}, \\[9pt]
& (c^{\lambda_5}_{1,1})=(\delta_{a}\delta_{b}\delta_{a}\delta_{b}). 
\end{align*}
It is clear that the conditions (C1) and (C2) hold. The condition (C3) holds because 
$K\delta_a\delta_b\oplus K\delta_b\delta_a\delta_b\oplus K(\delta_a\delta_b)^2$ and $K(\delta_a\delta_b)^2$ 
are submodules of $P_a$. Thus, this basis is a cellular basis of $\Gamma^{(0)}(T(1),v)$.
\end{proof}

%%%%
\para
We consider the algebras $\Gamma^{(1)}(T,v)$ from Definition \ref{D defn III}. By Proposition \ref{C Prop simple}(iii), the Brauer graph $T$ is of the form
 
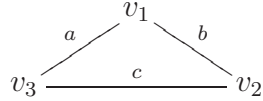
\[\begin{xy}
(0,0)*[o]+{v_{3}}="2",(30,0)*[o]+{v_{2}}="3",(15,10)*[o]+{v_{1}}="1",
\ar @{-}"1";"2"_{a}
\ar @{-}"2";"3"^{c}
\ar @{-}"1";"3"^{b}
\end{xy}\] 
with the exceptional vertex $v_{3}$. Then $\Gamma^{(1)}(T,v) = KQ_T/I^{(1)}(T,v)$, where $Q_T$ is the following quiver
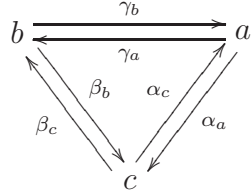
\[ \begin{xy}
(0,0)*[oo]+{b}="b",(30,0)*[oo]+{a}="a",(15,-20)*[oo]+{c}="c",
\ar @<1mm> "b";"a"^{\gamma_{b}}
\ar @<1mm> "a";"b"^{\gamma_{a}}
\ar @<1mm> "a";"c"^{\alpha_{a}}
\ar @<1mm> "c";"a"^{\alpha_{c}}
\ar @<1mm> "c";"b"^{\beta_{c}}
\ar @<1mm> "b";"c"^{\beta_{b}}
\end{xy}\] 
and $I^{(1)}(T,v)$ is the ideal generated by the following elements:
\begin{enumerate}
\item $\beta_{b}\alpha_{c}$, $\alpha_{c}\gamma_{a}$, $\gamma_{a}\beta_{b}$,
\item $\gamma_{a}\gamma_{b}\alpha_{a}$, $\alpha_{c}\alpha_{a}\beta_{c}$, $\beta_{b}\beta_{c}\gamma_{b}$,
\item $\beta_{b}-\gamma_{b}\alpha_{a}$, $\gamma_{a}-\alpha_{a}\beta_{c}$,
\item $\alpha_{c}\alpha_{a}\alpha_{c}-\beta_{c}\gamma_{b}$.
\end{enumerate} 
By using $(\mathrm{iii})$, we delete the arrows $\beta_{b}$ and $\gamma_{a}$ from $Q_T$ to obtain the Gabriel quiver. Then, 
Proposition \ref{C Prop simple}(iii) implies the following.

\begin{prop}
\label{D III}
The algebras $\Gamma^{(1)}(T,v)$ are not cellular.
\end{prop}

%%%%
\para
We consider the algebras $\Lambda(T,v_1,v_2)$ from Definition \ref{D defn IV}. 
Since they are nothing but Brauer tree algebras with two different exceptional vertices, 
we may apply \cite[Proposition 5.3]{KX1}.

%%%% 

\begin{prop}
\label{D IV}
The algebra $\Lambda(T,v_1,v_2)$ is cellular if and only if the Brauer graph $T$ is a straight line.
\end{prop}

By the assumption on $\Lambda(T,v_1,v_2)$ in Definition \ref{D defn IV}, the multiplicities of the exceptional vertices
$v_1$ and $v_2$ are both two.

%%%%
\para
Finally, we consider the algebras $\Gamma^{(2)}(T,v_1,v_2)$ from Definition \ref{D defn V}.
Then, by Proposition \ref{C Prop simple}(iii), the Brauer tree $T$ is of the form
\[T_{l}^{m}=
 \begin{xy}
(0,0)*[o]+{v_{1}}="3",(15,0)*[o]+{u}="u",(30,0)*[o]+{v_3}="001",(40,0)*[o]+{}="dammy1",
(45,0)*[o]+{\cdots}="cdots",(50,0)*[o]+{}="dammy2",(60,0)*[o]+{\circ}="002",(75,0)*[o]+{v_{2}}="2",(85,0)*[o]+{}="dammya",(90,0)*[o]+{\cdots}="\cdots",(95,0)*[o]+{}="dammyb",(105,0)*[o]+{\circ}="02",(120,0)*[o]+{\circ}="m",
\ar @{-}"3";"u"^{a}
\ar @{-}"u";"001"^{b}
\ar @{-}"001";"dammy1"^{1}
\ar @{-}"dammy2";"002"^{l-1}
\ar @{-}"002";"2"^{l}
\ar @{-}"2";"dammya"^{l+1}
\ar @{-}"dammyb";"02"^{m-1}
\ar @{-}"m";"02"_{m}
\end{xy}
\]
where $b=c=e$, and $v_2$ is the unique exceptional vertex, whose multiplicity is two. 
We always have $v_1\ne v_2$ by the assumption, but $u=v_2$ or $v_3=v_2$ may occur. When $u=v_2$ (respectively, $v_3=v_2$) we understand that $l=-1$ 
(respectively $l=0$). 

%%%%
\para
The Brauer quiver $Q_{T}^{(2)}$ for $l\ge1$ is as follows. 
For $u=v_2$ or $v_3=v_2$, we have the same underlying quiver 
and the only difference is that the location of the vertex $l$ is either $l=-1$ or $l=0$.
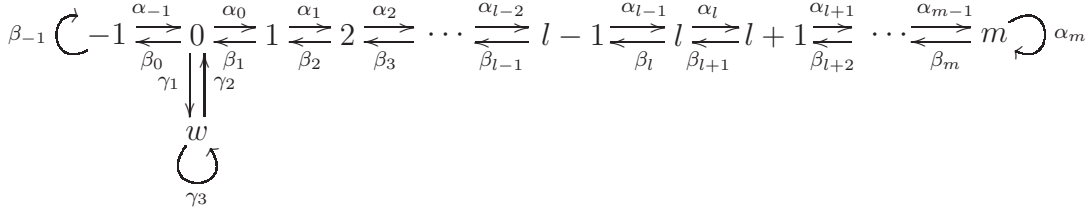
\[\begin{xy}
(-55,0)*[o]+{-1}="a",(-43,0)*[o]+{0}="b",(-33,0)*[o]+{1}="1",(-23,0)*[o]+{2}="2", (-13,0)*[o]+{}="3", (-10,0)*{\cdots}="dot1",(-6,0)*{}="l-2",(1,0)*{}="d1",(7,0)*[o]+{l-1}="l-1",(12,0)*{}="d2",(21,0)*[o]+{l}="l",(29,0)*{}="d3",(34,0)*[o]+{l+1}="l+1",(39,0)*{}="d4",(44,0)*{}="d5", (49,0)*{\cdots}="dot2",(52,0)*{}="m-1",(61,0)*[o]+{}="d6",(63,0)*[o]+{m}="m",
(-43,-13)*[o]+{w}="w"
\ar @<1mm> "a";"b"^{\alpha_{-1}}
\ar @<1mm> "b";"a"^{\beta_{0}}
\ar @<1mm> "b";"1"^{\alpha_{0}}
\ar @<1mm> "1";"b"^{\beta_{1}}
\ar @<1mm> "1";"2"^{\alpha_{1}}
\ar @<1mm> "2";"1"^{\beta_{2}}
\ar @<1mm> "2";"3"^{\alpha_{2}}
\ar @<1mm> "3";"2"^{\beta_{3}}
\ar @<1mm> "l-2";"d1"^{\alpha_{l-2}}
\ar @<1mm> "d1";"l-2"^{\beta_{l-1}}
\ar @<1mm> "d2";"l"^{\alpha_{l-1}}
\ar @<1mm> "l";"d2"^{\beta_{l}}
\ar @<1mm> "l";"d3"^{\alpha_{l}}
\ar @<1mm> "d3";"l"^{\beta_{l+1}}
\ar @<1mm> "d4";"d5"^{\alpha_{l+1}}
\ar @<1mm> "d5";"d4"^{\beta_{l+2}}
\ar @<1mm> "m-1";"d6"^{\alpha_{m-1}}
\ar @<1mm> "d6";"m-1"^{\beta_{m}}
\ar @<-1mm> "b";"w"_{\gamma_{1}}
\ar @<-1mm> "w";"b"_{\gamma_{2}}
\SelectTips{eu}{}
\ar @(dl,ul)"a";"a"^{\beta_{-1}}
\ar @(ur,dr)"m";"m"^{\alpha_m}
\ar @(dl,dr)"w";"w"_{\gamma_{3}}
\end{xy}
\]
Note that we do not use the labels $\alpha_i$ and $\beta_i$ for the $\alpha$-part and the $\beta$-part 
as we did in Definition \ref{D defn V}: 
if $i$ is a positive even integer then $\alpha_i$ belongs to the $\beta$-part and 
$\beta_i$ belongs to the $\alpha$-part. 
Moreover, we denote by $0$ (respectively, $-1$) the vertex $b=c=e$ (respectively, $a$). 
Thus $\alpha_{-1}=\alpha_a$, $\beta_{-1}=\beta_a$, $\alpha_0=\beta_b$, $\beta_0=\alpha_b$. 
The exceptional cycles are $\alpha_{l}\beta_{l+1}$ and $\beta_{l+1}\alpha_l$.
By the definition of $\Gamma^{(2)}(T,v_1,v_2)$, the cycles $\alpha_{-1}\beta_{0}$ and $\beta_o\alpha_{-1}$ are $\alpha$-cycles. 
We denote by $\Gamma^{(2)}(l,m)$ the algebra $\Gamma^{(2)}(T_{l}^{m},v_1,v_2)$.

%%%%
\para
Assume that $m\geq 1$. Then the ideal $I^{(2)}(T_{l}^{m},v_{1},v_{2})$ is generated by the following elements.
\begin{itemize}
\item[(a)]
If $1\leq l\le m-1$ then
\begin{enumerate}
\item $\alpha_{i}\alpha_{i+1}$, for $-1\leq i\leq m-1$, 
\item $\beta_{i}\beta_{i-1}$, for $0\leq i\leq m$,  
\item $\beta_{-1}\alpha_{-1}$ and $\alpha_m\beta_m$,
\item $\alpha_{i}\beta_{i+1}-\beta_{i}\alpha_{i-1}$, for $0\le i\le m-1$ and $i\neq l,$ $l+1$, 
\item $\beta_{-1}-\alpha_{-1}\beta_{0}$ and $\alpha_{m}-\beta_{m}\alpha_{m-1}$,
\item $(\alpha_{l}\beta_{l+1})^{2}-\beta_{l}\alpha_{l-1}$ and $\alpha_{l+1}\beta_{l+2}-(\beta_{l+1}\alpha_{l})^{2}$,
\item $\gamma_{2}\alpha_{0}$, $\beta_{1}\gamma_{1}$, $\gamma_{1}\gamma_{3}$, 
\item $\gamma_{3}\gamma_{2}$, $\gamma_{2}\beta_{0}$ and $\alpha_{-1}\gamma_{1}$,
\item $\beta_{0}\alpha_{-1}-\gamma_{1}\gamma_{2}$,
\item $\gamma_{2}\beta_{0}\alpha_{-1}\gamma_{1}-\gamma_{3}$.
\end{enumerate}
\item[(b)]
If $l=m$, then 
\begin{enumerate}
\item $\alpha_{i}\alpha_{i+1}$, for $-1\leq i\leq m-1$,
\item $\beta_{i}\beta_{i-1}$, for $0\leq i\leq m$,  
\item $\beta_{-1}\alpha_{-1}$ and $\alpha_m\beta_m$,
\item $\alpha_{i}\beta_{i+1}-\beta_{i}\alpha_{i-1}$, for $0 \leq i \leq m-1$, 
\item $\beta_{-1}-\alpha_{-1}\beta_{0}$ and $\alpha_{m}^2-\beta_{m}\alpha_{m-1}$,
\item $\gamma_{2}\alpha_{0}$, $\beta_{1}\gamma_{1}$, $\gamma_{1}\gamma_{3}$, 
\item $\gamma_{3}\gamma_{2}$, $\gamma_{2}\beta_{0}$ and $\alpha_{-1}\gamma_{1}$,
\item $\beta_{0}\alpha_{-1}-\gamma_{1}\gamma_{2}$,
\item $\gamma_{2}\beta_{0}\alpha_{-1}\gamma_{1}-\gamma_{3}$.
\end{enumerate}
\end{itemize}

In (a), we delete the arrows $\alpha_{m}$, $\beta_{-1}$ and $\gamma_{3}$ to obtain the Gabriel quiver, and 
replace $\alpha_{m}$ (respectively, $\beta_{-1}$, $\gamma_{3}$) in the above generators 
with $\beta_{m}\alpha_{m-1}$ (respectively, $\alpha_{-1}\beta_{0}$, $0$). 

In (b), we delete the arrows $\beta_{-1}$ and $\gamma_{3}$ to obtain the Gabriel quiver, and 
replace $\beta_{-1}$ (respectively, $\gamma_{3}$) in the above generators 
with $\alpha_{-1}\beta_{0}$ (respectively, $0$).

\begin{lem}
\label{D V-1}
If $m\ge 1$ and $1\leq l\leq m$, then the algebra $\Gamma^{(2)}(l,m)$ is not cellular.
\end{lem}
\begin{proof}
Let $e=e_{-1}+e_0+e_w+e_1$. By Lemma \ref{C Lem idem}, 
it suffices to show that $e\Gamma^{(2)}(l,m)e$ is not cellular.
If $l\geq 2$, then indecomposable projective $e\Gamma^{(2)}(l,m)e$-modules are
\[\begin{array}{lcl}
P_{-1} &=& Ke_{-1}\oplus K \beta_{0}\oplus K\alpha_{-1}\beta_{0}, \\
P_{0} &=& Ke_0\oplus K\alpha_{-1}\oplus K\gamma_{2} \oplus K\beta_{1}\oplus K\beta_{0}\alpha_{-1}, \\
P_{w} &=& Ke_{w} \oplus K\gamma_{1}\oplus K\gamma_2\gamma_1, \\
P_{1} &=& Ke_{1} \oplus  K\alpha_{0}\oplus K\beta_{1}\alpha_{0}, \\
\end{array} \]
and the Cartan matrix of $e\Gamma^{(2)}(l,m)e$ is 
\[ \bC=\begin{pmatrix}
2 & 1 & 0 & 0 \\
1 & 2 & 1 & 1 \\
0 & 1 & 2 & 0 \\
0 & 1 & 0 & 2 \end{pmatrix}. \]
If $l=1$, then indecomposable projective $e\Gamma^{(2)}(1,m)e$-modules are
\[\begin{array}{lcl}
P_{-1} &=& Ke_{-1}\oplus K \beta_{0}\oplus K\alpha_{-1}\beta_{0}, \\
P_{0} &=& Ke_0\oplus K\alpha_{-1}\oplus K\gamma_{2} \oplus K\beta_{1}\oplus K\beta_{0}\alpha_{-1}, \\
P_{w} &=& Ke_{w} \oplus K\gamma_{1}\oplus K\gamma_2\gamma_1, \\
P_{1} &=& Ke_{1} \oplus  K\alpha_{0}\oplus K\alpha_{1}\beta_{2}\oplus K\beta_{1}\alpha_{0}, \\
\end{array} \]
and the Cartan matrix of $e\Gamma^{(2)}(1,m)e$ is 
\[ \bC=\begin{pmatrix}
2 & 1 & 0 & 0 \\
1 & 2 & 1 & 1 \\
0 & 1 & 2 & 0 \\
0 & 1 & 0 & 3 \end{pmatrix}. \]
 
Now, we find that there is no matrix $\bD$ with entries in $\ZZ_{\ge0}$ that satisfies 
$\bD^T\bD=\bC$ in both cases. 
Thus, $e\Gamma^{(2)}(l,m)e$ is not cellular. 
\end{proof}

If $l=0$ and $m\ge1$, then the ideal $I^{(2)}(T_{0}^{m},v_{1},v_{2})$ is generated by the following elements:
\begin{enumerate}
\item $\alpha_{i}\alpha_{i+1}$, for $-1\leq i\leq m-1$,
\item $\beta_{i}\beta_{i-1}$, for $0\leq i\leq m$,  
\item $\beta_{-1}\alpha_{-1}$ and $\alpha_m\beta_m$,
\item $\alpha_{i}\beta_{i+1}-\beta_{i}\alpha_{i-1}$, for $2\leq i\leq m-1$, 
\item $\beta_{-1}-\alpha_{-1}\beta_{0}$ and $\alpha_{m}-\beta_{m}\alpha_{m-1}$ if $m\ge2$,
\item $(\alpha_{0}\beta_{1})^{2}-\beta_{0}\alpha_{-1}$ and 
$\alpha_{1}-(\beta_{1}\alpha_{0})^2$ if $m=1$, $\alpha_{1}\beta_{2}-(\beta_{1}\alpha_{0})^{2}$ if $m\ge2$,
\item $\gamma_{2}\alpha_{0}$, $\beta_{1}\gamma_{1}$, $\gamma_{1}\gamma_{3}$, $\gamma_{3}\gamma_{2}$,
\item $\gamma_{2}\beta_{0}$ and $\alpha_{-1}\gamma_{1}$,
\item $\beta_{0}\alpha_{-1}-\gamma_{1}\gamma_{2}$,
\item $\gamma_{2}\beta_{0}\alpha_{-1}\gamma_{1}-\gamma_{3}$.
\end{enumerate}
We delete the arrows $\alpha_{m}$, $\beta_{-1}$ and $\gamma_{3}$ 
to obtain the Gabriel quiver, and replace 
$\alpha_m$ (respectively, $\beta_{-1}$, $\gamma_{3}$) in the above generators of 
$I^{(2)}(T_{0}^{m},v_{1},v_{2})$ with $\beta_{m}\alpha_{m-1}$ if $m\ge2$ and 
$(\beta_1\alpha_0)^2$ if $m=1$ (respectively, $\alpha_{-1}\beta_{0}$ , $0$). 

\begin{lem}\label{D V-2}
If $m\ge1$ and $v_3=v_2$, then the algebra $\Gamma^{(2)}(0,m)$ is not cellular.
\end{lem}
\begin{proof}
Let $e=e_{-1}+e_0+e_{w}+e_{1}$. Then, indecomposable projective $e\Gamma^{(2)}(0,m)e$-modules are given by
\[\begin{array}{lcl}
P_{-1} &=& Ke_{-1}\oplus K \beta_{0}\oplus K\alpha_{-1}\beta_{0}, \\
P_{0} &=& Ke_0\oplus K\alpha_{-1}\oplus K\gamma_{2} \oplus K\beta_{1}\oplus K\alpha_0\beta_1\oplus K\beta_1\alpha_0\beta_1\oplus K(\alpha_{0}\beta_{1})^2, \\
P_{w} &=& Ke_{w} \oplus K\gamma_{1}\oplus K\gamma_2\gamma_1, \\
P_{1} &=& Ke_{1} \oplus  K\alpha_{0}\oplus  K\beta_{1}\alpha_{0}\oplus K\alpha_0\beta_1\alpha_0\oplus K(\beta_1\alpha_0)^2. \\
\end{array} \]
Thus, the Cartan matrix of $e\Gamma^{(2)}(0,m)e$ is
\[ \bC=\begin{pmatrix}
2 & 1 & 0 & 0 \\
1 & 3 & 1 & 2 \\
0 & 1 & 2 & 0 \\
0 & 2 & 0 & 3 \end{pmatrix}. \]
Suppose that $\Gamma^{(2)}(0,m)$ is cellular. Then so is $e\Gamma^{(0)}(0,m)e$ by Lemma \ref{C Lem idem}. 
However, it is impossible to find the decomposition matrix $\bD$ which satisfies $\bD^T\bD=\bC$.
\end{proof}

If $l=-1$ and $m\ge1$, then exceptional simple cycles are $\alpha_{-1}\beta_0$ and $\beta_0\alpha_{-1}$. 
In this case, the ideal $I^{(2)}(T_{-1}^{m},v_{1},v_{2})$ is generated by the following elements:
\begin{enumerate}
\item $\alpha_{i}\alpha_{i+1}$, for $-1\leq i\leq m-1$,
\item $\beta_{i}\beta_{i-1}$, for $0\leq i\leq m$,  
\item $\beta_{-1}\alpha_{-1}$ and $\alpha_m\beta_m$,
\item $\alpha_{i}\beta_{i+1}-\beta_{i}\alpha_{i-1}$, for $1\leq i\leq m-1$, 
\item $\beta_{-1}-(\alpha_{-1}\beta_{0})^{2}$ and $\alpha_{m}-\beta_{m}\alpha_{m-1}$,
\item $\alpha_{0}\beta_{1}-(\beta_{0}\alpha_{-1})^{2}$,
\item $\gamma_{2}\alpha_{0}$, $\beta_{1}\gamma_{1}$, $\gamma_{1}\gamma_{3}$ and $\gamma_{3}\gamma_{2}$,
\item $\gamma_2\gamma_1\gamma_2\beta_0$ and $\alpha_{-1}\beta_0\alpha_{-1}\gamma_1$,
\item $\beta_0\alpha_{-1}-\gamma_1\gamma_2$,
\item $(\gamma_2\beta_0\alpha_{-1}\gamma_1)^2-\gamma_3$.
\end{enumerate}
We delete the arrows $\alpha_{m}$, $\beta_{-1}$ and $\gamma_{3}$ to obtain the Gabriel quiver, 
and replace $\alpha_{m}$ (respectively, $\beta_0$, $\gamma_{3}$) in the above generators of 
$I^{(2)}(T_{-1}^{m},v_{1},v_{2})$  with $\beta_{m}\alpha_{m-1}$ 
(respectively, $(\alpha_{-1}\beta_{0})^{2}$, $0$).

\begin{lem}\label{D V-3}
If $m\ge1$ and $u=v_2$, then the algebra $\Gamma^{(2)}(-1,m)$ is not cellular.
\end{lem}
\begin{proof}
Let $e=e_{-1}+e_0+e_{w}+e_{1}$ as before. Then, indecomposable projective $e\Gamma^{(2)}(-1,m)e$-modules are given by
\begin{align*}
P_{-1} &= Ke_{-1}\oplus K \beta_{0}\oplus K\alpha_{-1}\beta_{0} \oplus K\gamma_2\beta_0 \oplus K\beta_0\alpha_{-1}\beta_0\oplus K (\alpha_{-1}\beta_0)^2, \\
P_{0} &= Ke_0\oplus K\alpha_{-1}\oplus K\gamma_{2} \oplus K\beta_{1}\oplus K\beta_0\alpha_{-1}\oplus K\alpha_{-1}\beta_0\alpha_{-1}\oplus K\gamma_2\beta_0\alpha_{-1}\oplus K(\beta_0\alpha_{-1})^2, \\
P_{w} &= Ke_{w} \oplus K\gamma_{1}\oplus K\alpha_{-1}\gamma_1\oplus K\gamma_2\gamma_1\oplus K\gamma_1\gamma_2\gamma_1\oplus K(\gamma_2\gamma_1)^2, \\
P_{1} &= Ke_{1} \oplus  K\alpha_{0}\oplus  K\beta_{1}\alpha_{0}. \\
\end{align*}
Thus, the Cartan matrix of $e\Gamma^{(2)}(-1,m)e$ is
\[ \bC=\begin{pmatrix}
3 & 2 & 1 & 0 \\
2 & 3 & 2 & 1 \\
1 & 2 & 3 & 0 \\
0 & 1 & 0 & 2 \end{pmatrix}, \]
and we cannot find the decomposition matrix $\bD$ again. 
\end{proof}

%%%%
\para
Assume that $m=0$. If $l=0$, then  the ideal $I^{(2)}(T_{0}^{0},v_{1},v_{2})$ is generated by the following elements:
\begin{enumerate}
\item $\alpha_{-1}\alpha_{0}$, $\beta_{0}\beta_{-1}$, $\beta_{-1}\alpha_{-1}$ and $\alpha_{0}\beta_{0}$,
\item $\alpha_0^2-\beta_0\alpha_{-1}$,
\item $\beta_{-1}-\alpha_{-1}\beta_0$,
\item $\gamma_2\alpha_0$, $\alpha_0\gamma_1$, $\gamma_1\gamma_3$ and $\gamma_3\gamma_2$,
\item $\gamma_2\beta_0$ and $\alpha_{-1}\gamma_1$,
\item $\beta_0\alpha_{-1}-\gamma_1\gamma_2$,
\item $\gamma_3-\gamma_2\beta_0\alpha_{-1}\gamma_1$.
\end{enumerate}
By using (iii) and (vii), we delete the arrows $\beta_{-1}$ and $\gamma_3$ {to obtain the Gabriel quiver, and replace $\beta_{-1}$ (respectively, $\gamma_3$) 
in the above generators of $I^{(2)}(T_{0}^{0},v_{1},v_{2})$ with $\alpha_{-1}\beta_0$ (respectively, $0$).

If $l=-1$, then  the ideal $I^{(2)}(T_{-1}^{0},v_{1},v_{2})$ is generated by the following elements:
\begin{enumerate}
\item $\alpha_{-1}\alpha_{0}$, $\beta_{0}\beta_{-1}$, $\beta_{-1}\alpha_{-1}$ and $\alpha_{0}\beta_{0}$,
\item $\alpha_0-(\beta_0\alpha_{-1})^2$,
\item $\beta_{-1}-(\alpha_{-1}\beta_0)^2$,
\item $\gamma_2\alpha_0$, $\alpha_0\gamma_1$, $\gamma_1\gamma_3$ and $\gamma_3\gamma_2$,
\item $\gamma_2\gamma_1\gamma_2\beta_0$ and $\alpha_{-1}\beta_0\alpha_{-1}\gamma_1$,
\item $\beta_0\alpha_{-1}-\gamma_1\gamma_2$,
\item $\gamma_3-(\gamma_2\beta_0\alpha_{-1}\gamma_1)^2$.
\end{enumerate}
By using (ii), (iii) and (vii), we delete the arrows $\alpha_{0}$, $\beta_{-1}$ and $\gamma_3$ 
to obtain the Gabriel quiver, 
and replace $\alpha_{0}$ (respectively, $\beta_{-1}$, $\gamma_3$) in the generators of $I^{(2)}(T_{-1}^{0},v_{1},v_{2})$ 
with $(\beta_0\alpha_{-1})^2$ (respectively, $(\alpha_{-1}\beta_0)^2$, $0$).

\begin{lem}\label{D V-5} For each $l=-1,0$, the algebra $\Gamma^{(2)}(T_{l}^{0},v_1,v_2)$ is cellular.
\end{lem}
\begin{proof} Assume that $l=0$. Then, indecomposable projective modules are 
\[\begin{array}{lcl}
P_{-1} &=& Ke_{-1}\oplus K \beta_{0}\oplus K\alpha_{-1}\beta_{0}, \\
P_{0} &=& Ke_0\oplus K\alpha_{-1}\oplus K\gamma_{2} \oplus K\alpha_0\oplus K\beta_0\alpha_{-1}, \\
P_{w} &=& Ke_{w} \oplus K\gamma_{1}\oplus K\gamma_2\gamma_1. \\
\end{array} \]
Let $\iota$ be the anti-involution of $\Gamma^{(2)}(T_{0}^{0},v_1,v_2)$ induced by 
swapping $\alpha_{-1}$ and $\beta_0$, $\gamma_1$ and $\gamma_2$, and fixing $\alpha_0$. 
Put $\Lambda=\{ \lambda_1<\lambda_2<\lambda_3<\lambda_4<\lambda_5\}$ and define
\[
\cT(\lambda_i)=\left\{\begin{array}{ll}
\{1\} & \text{if $i=1,2,5$}, \\
\{1,2\} & \text{otherwise.}\end{array}\right. 
\]
Then one can construct a cellular basis of $\Gamma^{(2)}(T_{0}^{0},v_1,v_2)$ 
whose cell datum is $(\Lambda, \cT, \cC, \iota)$ as follows.
\begin{align*}
& (c_{1,1}^{\lambda_1})=(e_w),\ 
(c_{i,j}^{\lambda_2})_{i,j\in\cT(\lambda_{3})}=
\begin{pmatrix}
e_0 &  \gamma_1 \\
\gamma_2 & \gamma_2\gamma_1 \end{pmatrix},\ 
(c_{1,1}^{\lambda_3})=(\alpha_0), \\[7pt]
& (c_{i,j}^{\lambda_4})_{i,j\in\cT(\lambda_{4})}=
\begin{pmatrix}
e_{-1} &  \alpha_{-1} \\
\beta_0 & \beta_{0}\alpha_{-1} \end{pmatrix},\ (c_{1,1}^{\lambda_5})=(\alpha_{-1}\beta_0).
\end{align*}
It is clear that (C1) and (C2) hold. To see that (C3) holds, observe that 
\begin{enumerate}
\item $K\alpha_0\oplus K\alpha_{-1}\oplus K\beta_0\alpha_{-1}$ and $K\beta_0\alpha_{-1}$ are submodules of $P_0$.
\item $K\alpha_{-1}\beta_0$ is a submodule of $P_{-1}$.
\end{enumerate}

Assume that $l=-1$. Then, indecomposable projective modules are
\[\begin{array}{lcl}
P_{-1} &=& Ke_{-1}\oplus K \beta_{0}\oplus K\alpha_{-1}\beta_{0}\oplus K\gamma_2\beta_0\oplus K\beta_0\alpha_{-1}\beta_0\oplus K(\alpha_{-1}\beta_0)^2, \\
P_{0} &=& Ke_0\oplus K\alpha_{-1}\oplus K\gamma_{2} \oplus K\beta_0\alpha_{-1}\oplus K\alpha_{-1}\beta_0\alpha_{-1}\oplus K\gamma_2\beta_0\alpha_{-1}\oplus K(\beta_0\alpha_{-1})^2, \\
P_{w} &=& Ke_{w} \oplus K\gamma_{1}\oplus \oplus \alpha_{-1}\gamma_1\oplus K\gamma_2\gamma_1\oplus K\beta_0\alpha_{-1}\gamma_1\oplus K(\gamma_2\gamma_1)^2. \\
\end{array} \]
Let $\iota$ be the anti-involution of $\Gamma^{(2)}(T_{-1}^{0},v_1,v_2)$ induced by 
swapping $\alpha_{-1}$ and $\beta_0$, $\gamma_1$ and $\gamma_2$, and fixing $\alpha_0$. 
Put $\Lambda=\{ \lambda_1<\lambda_2<\lambda_3<\lambda_4<\lambda_5\}$ and define
\[ \cT(\lambda_i)=\left\{\begin{array}{ll}
\{1\} & \text{if $i=1,5$}, \\
\{1,2\} & \text{if $i=2,4$}, \\
\{1,2,3\} & \text{if $i=3$}.\end{array}\right. \]
Noting that $K\alpha_{-1}\beta_0\oplus K\beta_0\alpha_{-1}\beta_0\oplus K(\alpha_{-1}\beta_0)^2$ and $K(\alpha_{-1}\beta_0)^2$ are 
submodules of $P_{-1}$, we can check that the following gives a cellular basis with cell datum $(\Lambda, \cT, \cC, \iota)$.
\begin{align*}
& \qquad (c_{1,1}^{\lambda_1})=(e_w),\quad (c_{i,j}^{\lambda_2})_{i,j\in\cT(\lambda_{2})}=
\begin{pmatrix}
e_0 &  \gamma_1 \\
\gamma_2 & \gamma_2\gamma_1 \end{pmatrix},\\[7pt]
& \qquad (c_{i,j}^{\lambda_3})_{i,j\in\cT(\lambda_{3})}=
\begin{pmatrix}
e_{-1} &  \alpha_{-1} & \alpha_{-1}\gamma_1 \\
\beta_0 & \beta_0\alpha_{-1} & \beta_{0}\alpha_{-1}\gamma_1 \\
\gamma_2\beta_0 & \gamma_2\beta_0\alpha_{-1} & (\gamma_2\gamma_1)^2
 \end{pmatrix},\\[9pt]
& (c_{i,j}^{\lambda_4})_{i,j\in\cT(\lambda_{4})}=
\begin{pmatrix}
\alpha_{-1}\beta_0 &  \alpha_{-1}\beta_0\alpha_{-1} \\
\beta_0\alpha_{-1}\beta_0 & (\beta_{0}\alpha_{-1})^2 \end{pmatrix},\ (c_{1,1}^{\lambda_5})=
((\alpha_{-1}\beta_0)^2).
\end{align*}
Thus, $\Gamma^{(2)}(T_{l}^{0},v_1,v_2)$ is cellular.
\end{proof}

%%%%
\para
It remains to consider non-standard self-injective algebras of domestic type. 
Let $T$ be a Brauer graph with exactly one cycle, which is a loop, and the loop is its direct successor. 
Thus, $T$ is of the following form:
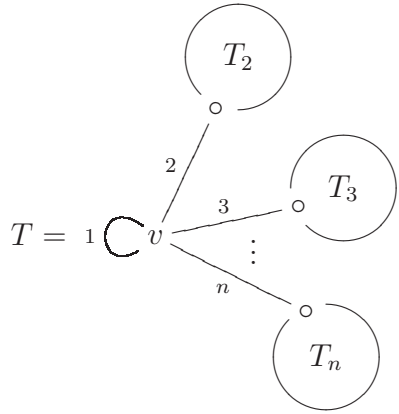
\[ T= \begin{xy}
(0,0)*[o]+{v}="0", (8,17)*[o]+{\circ}="1",
(11,24)*+{T_{2}}*\cir<8mm>{ul_l},
(19,4)*[o]+{\circ}="2", 
(25,6.5)*+{T_{3}}*\cir<8mm>{u_ul},
(13,-1)*[o]+{\vdots}="d1",
(20,-10)*[o]+{\circ}="v1", 
(22.7,-16)*+{T_{n}}*\cir<8mm>{r_ur},
\ar @{-}"0";"1"^{2}
\ar @{-}"0";"2"^{3}
\ar @{-}"0";"v1"_{n}
\ar @{-} @(ul,dl)"0";"0"_{1}
\end{xy} \]
We denote by $Q_T$ the corresponding Brauer quver. Then there are two distinguished simple cycles, say $C_1$ and $C_2$, 
associated with the cyclic ordering around the vertex $v$. 
We assume that $C_1$ is the loop from $1$ to its direct successor $1$. Hence, $Q_T$ has the following subquiver:
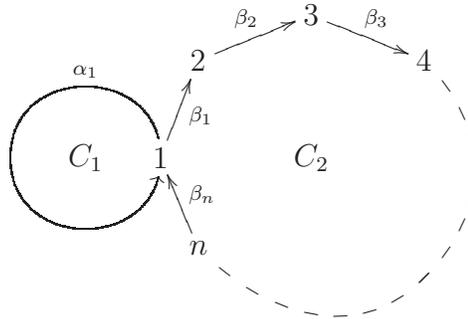
\[
\begin{xy}
(0,11)*[o]+{3}="0",(15,5)*[o]+{4}="1", (-15,5)*[oo]+{2}="3",(-20,-8)*[o]+{1}="4",(-40,-8)*{}="damy",
(-15,-20)*[o]+{n}="n",(15,-24)*{}="a",(-30,-8)*{C_1}="c1",(0,-8)*{C_2}="c2",
\ar  "0";"1"^{\beta_3}
\ar  "3";"0"^{\beta_2}
\ar "4";"3"_{\beta_{1}}
\ar "n";"4"_{\beta_{n}}
\ar @{-}@(u,u)"4";"damy"_{\alpha_1}
\SelectTips{eu}{}
\ar @(d,d)"damy";"4"
\ar @{--}@(dr,dl) "n";"a"
\ar @{--}@(ur,dr) "a";"1"
\end{xy}
\]
For each vertex $i\in (Q_T)_0$, we denote by $A_i$ and $B_i$ the usual simple cycles in $Q_T$ around $i$. 
Moreover, let $B_j' = \beta_j\beta_{j+1}\cdots\beta_n\alpha_1\beta_1\cdots\beta_{j-1}$, 
if $j$ is a vertex on the cycle $C_2$ and $j\ne 1$. Then, we define $\Omega(T) = KQ_T/ J(T)$, 
where $J(T)$ is the ideal generated by the following elements:
\begin{enumerate}
\item $\alpha\beta$ and $\beta\alpha$, for $\alpha$ in the $\alpha$-part and $\beta$ in the $\beta$-part 
such that the middle point of $\alpha\beta$ and the middle point of $\beta\alpha$ are not the vertex $1$,
\item $\beta_n\beta_1$,
\item $A_i-B_i$, for a vertex $i$ which does not lie on the cycle $C_2$,
\item $A_i-B_i'$, for a vertex $i$ on the cycle $C_2$ which is different from the vertex $1$,
\item $A_1^2-A_1B_1$ and $A_1B_1+B_1A_1$.
\end{enumerate}
Then the following statement holds.
\begin{prop}[{\cite[Theorem 4.14]{S}}]
\label{D Prop dom}
Let $A$ be a self-injective algebra. Then, the following statements are equivalent:
\begin{enumerate}
\item[(1)] $A$ is non-standard domestic of infinite type.
\item[(2)] $A\cong \Omega(T)$ for some Brauer graph $T$ with one loop.
\end{enumerate}
\end{prop}
By Proposition \ref{C Prop simple}, we may assume that $T$ is of the form
\[ T(n)=\begin{xy}
(0,0)*[o]+{\circ}="0", (10,0)*[o]+{\circ}="1",(20,0)*[o]+{\circ}="2", (30,0)*[o]+{\cdots}="d1",(40,0)*[o]+{\circ}="v1", (50,0)*[o]+{\circ}="d2", 
\ar @{-}"0";"1"^{2}
\ar @{-}"1";"2"^{3}
\ar @{-}"2";"d1"
\ar @{-}"d1";"v1"^{n-1}
\ar @{-}"v1";"d2"^{n}
\ar @{-} @(ul,dl)"0";"0"_{1}
\end{xy} \]
where $n\geq 1$. We denote the cycles by $C_1=\gamma$ and $C_2=\alpha_1\beta_1$. 
Note that if $n=1$, then one can show $\Omega(T(1))\cong K[X,Y]/(X^2-XY, XY+YX, Y^2)$. 
Assume that $n\geq 2$. Then, $\Omega(T(n)) = KQ_{T(n)}/I$, where $Q_{T(n)}$ is the Brauer quiver
\[
\begin{xy}
(0,0)*[o]+{1}="0",(10,0)*[o]+{2}="1",(20,0)*[o]+{3}="2", (30,0)*[o]+{}="3",(34,0)*{\cdots}="d1",(38,0)*[o]+{}="m-1",
(50,0)*[o]+{n}="m"
\ar @<1mm> "0";"1"^{\alpha_{1}}
\ar @<1mm> "1";"0"^{\beta_{1}}
\ar @<1mm> "1";"2"^{\alpha_{2}}
\ar @<1mm> "2";"1"^{\beta_{2}}
\ar @<1mm> "2";"3"^{\alpha_{3}}
\ar @<1mm> "3";"2"^{\beta_{3}}
\ar @<1mm> "m-1";"m"^{\alpha_{n-1}}
\ar @<1mm> "m";"m-1"^{\beta_{n-1}}
\SelectTips{eu}{}
\ar @(ur,dr) "m";"m"^{\beta_n}
\ar @(dl,ul)"0";"0"^{\gamma}
\end{xy}\]
and $I$ is the ideal in $KQ_{T(n)}$ generated by 
\begin{enumerate}
\item $\alpha_i\alpha_{i+1}$, for $1\leq i\leq n-2$,
\item $\beta_{i+1}\beta_{i}$, for $1\leq i \leq n-1$,
\item $\beta_1\alpha_1$ and $\alpha_{n-1}\beta_n$,
\item $\alpha_{i}\beta_{i}-\beta_{i-1}\alpha_{i-1}$, for $3\leq i\leq n-1$,
\item $\alpha_{2}\beta_{2}-\beta_1\gamma\alpha_1$ and $\beta_n-\beta_{n-1}\alpha_{n-1}$,
\item $\gamma^2-\gamma\alpha_1\beta_1$ and $\gamma\alpha_1\beta_1+\alpha_1\beta_1\gamma$.
\end{enumerate}

\begin{lem}\label{D nondom}
For every $n\geq 1$, the algebra $\Omega(T(n))$ is not cellular.
\end{lem}
\begin{proof} If $n=1$, then $\Omega(T(1)) = K1\oplus KX\oplus KY \oplus KXY$ and 
the Cartan matrix is $(4)$. Suppose that $\Omega(T(1))$ is cellular with cell detum $(\Lambda, \cT,\cC,\iota)$.  
Then, the decomposition matrix is $\bD=(1\ 1\ 1\ 1)^T$ by Proposition \ref{C Prop decom matrix},  
and we may write $\Lambda=\{\lambda_1, \lambda_2, \lambda_3, \lambda_4\}$. 
Moreover, $4=\sum_{i=1}^4 |\cT(\lambda_i)|^2$ implies that we may set $\cT(\lambda_i)=\{1\}$, for $i=1,2,3,4$.
Put $\cC=\{c^{\lambda_i}_{1,1}\ |\ i=1,2,3,4\}$. By the definition of cellular algebras, the anti-involution fixes all elements of $\Omega(T(1))$. Hence, we have
\[ XY=\iota(XY)=\iota(-YX)=-\iota(X)\iota(Y)=-XY, \]
a contradiction. 

Assume that $n\geq 2$. Let $e=e_1+e_2$ and $B=eAe$. Then, indecomposable projective $B$-modules are
\[\begin{array}{lcl}
P_{1} &=& Ke_{1}\oplus K \beta_{1}\oplus K\gamma\oplus K\alpha_{1}\beta_{1}\oplus K\beta_1\gamma \oplus K\gamma^2, \\
P_{2} &=& Ke_2\oplus K\alpha_{1}\oplus K\gamma\alpha_1\oplus K\beta_{1}\gamma\alpha_1, \\
\end{array} \] 
and the Cartan matrix of $B$ is 
\[ \bC=\begin{pmatrix}
4 & 2 \\
2 & 2 \end{pmatrix}.\]
Our goal is to show that $B$ is not cellular. Thus, we assume to the contrary that 
$B$ is a cellular algebra with a cell datum $(\Lambda, \cT,\cC,\iota)$. 
Then, \cite[Proposition 8.2]{KX2} implies that 
we may assume that the anti-involution $\iota$ fixes $e_1$ and $e_2$. Further, 
by (i) and (iii) from Proposition \ref{C Prop decom matrix}, we obtain
\[ \bD=\begin{pmatrix}
1 & 1 \\
1 & 1 \\
1 & 0 \\
1 & 0 \end{pmatrix}\]
modulo rearrangement of rows. We label $\Lambda=\{\lambda_1,\lambda_2,\lambda_3,\lambda_4\}\supset
\Lambda^+=\{\mu_1,\mu_2\}$ in such a way that $[\Delta(\lambda_i):L(\mu_j)]=d_{ij}$. 
As $\dim \Delta(\lambda_i)=|\cT(\lambda_i)|$, we may define 
\[ \cT(\lambda_{1})=\cT(\lambda_{2})=\{1,2\},\quad \cT(\lambda_{3})=\cT(\lambda_{4})=\{1\}.\]
By Lemma \ref{C Lem min}, we may assume that $\lambda_3$ is minimal and $\lambda_4$ is maximal 
in the partial order, and we may extend the partial order to the total order $3<1<2<4$. 
Since $\lambda_3$ is minimal, we have $\lambda_3\in \Lambda^+$. Then, $d_{3,1}=1$ and $d_{3,2}=0$ imply 
$\mu_1=\lambda_3$ and $\Delta(\lambda_4)\cong L(\lambda_3)$ implies that $\lambda_4\not\in\Lambda^+$.
On the other hand, $1<2$ in the linear extension implies that $\Delta(\lambda_2)$ is isomorphic to 
a submodule of $P(\mu_2)$, so that $\Delta(\lambda_2)$ is uniserial, 
$L(\mu_2)$ is its socle and $L(\mu_1)$ is its cosocle. 
Thus, we obtain $\lambda_2\not\in\Lambda^+$ by $\lambda_2\ne \lambda_3=\mu_1$. 
We conclude that $\mu_1=\lambda_3$ and $\mu_2=\lambda_1$. 
Then, since $\Delta(\lambda_1)$ is a factor module of $P(\mu_2)=P_2$, 
$\Delta(\lambda_1)$ is uniserial, $L_2$ is its cosocle and $L_1$ is its socle. 

Now we consider the sequence of two-sided ideals associated with the linear extension $3<1<2<4$:
\[
B=B(\lambda_3) \supset B(\lambda_1) \supset B(\lambda_2) \supset B(\lambda_4) \supset 0.
\]
Since $B(\lambda_4)e_1\cong L_1$ and $B(\lambda_4)e_2=0$, $B(\lambda_4)=K\gamma^2$ follows. 
Moreover, arguing in the similar manner as in the paragraph after (\ref{F DB 2-2}), we obtain 
\[ 
B(\lambda_2)/B(\lambda_4)\equiv  
K(\gamma+\lambda\alpha_1\beta_1)\oplus K\beta_1\gamma \oplus K\gamma\alpha_1\oplus K\beta_1\gamma\alpha_1
\quad \mod B(\lambda_4), 
\]
for some $\lambda\in K$. We compute the linear map on $B(\lambda_2)/B(\lambda_4)$ which is induced by 
the anti-involution $\iota$. As $\iota(e_1)=e_1$ and $\iota(e_2)=e_2$, we may write
\[
\iota(\alpha_1)=x\beta_1+y\beta_1\gamma, \quad \iota(\beta_1)=z\alpha_1+w\gamma\alpha_1, \quad 
\iota(\gamma)=s\gamma+t\alpha_1\beta_1+u\gamma^2,
\]
for some $x,y,z,w,s,t,u\in K$. Since the anti-involution $\iota$ must preserve the relations 
$\gamma^2=\gamma\alpha_1\beta_1$, $\beta_1\alpha_1=0$, $\gamma\alpha_1\beta_1+\alpha_1\beta_1\gamma=0$, 
we have $xw+yz=0$ and $s^2=-szx$. Further, since $c^{\lambda_4}_{1,1}=\gamma^2 \in B(\lambda_4)=K\gamma^2$ must hold, up to a nonzero 
scalar multiple, $\iota(\gamma^2)=\gamma^2$ implies $s^2=1$. We compute
\[
\begin{cases}
\iota(\gamma+\lambda\alpha_1\beta_1) &= s(\gamma-\lambda\alpha_1\beta_1)+t\alpha_1\beta_1+(u-\lambda yz+\lambda xw)\gamma^2,\\
\iota(\beta_1\gamma) &= sz\gamma\alpha_1, \\
\iota(\gamma\alpha_1) &= sx\beta_1\gamma, \\
\iota(\beta_1\gamma\alpha_1) &= szx\beta_1\gamma\alpha_1 = -\beta_1\gamma\alpha_1.
\end{cases}
\]
Then, as $B(\lambda_2)/B(\lambda_4)$ is stable under the anti-involution $\iota$, we must have $t=2\lambda s$ and if we 
represent the linear map on $B(\lambda_2)/B(\lambda_4)$ induced by the anti-involution $\iota$ with respect to 
the basis $(\gamma+\lambda\alpha_1\beta_1, \beta_1\gamma, \gamma\alpha_1, \beta_1\gamma\alpha_1)$ ({\rm mod} $B(\lambda_4)$), then 
\[
\iota(\gamma+\lambda\alpha_1\beta_1, \beta_1\gamma, \gamma\alpha_1, \beta_1\gamma\alpha_1)
=(\gamma+\lambda\alpha_1\beta_1, \beta_1\gamma, \gamma\alpha_1, \beta_1\gamma\alpha_1)M
\]
where the matrix $M$ is given by
\[
M=\begin{pmatrix}
s & 0 & 0 & 0 \\
0 & 0 & sx & 0 \\
0 & sz & 0 & 0 \\
0 & 0 & 0 & -1 \end{pmatrix}.
\]
The characteristic polynomial of $M$ is $\varphi_M(X)=(X-s)(X+1)(X^2+s)$. Since the anti-involution $\iota$ 
fixes $c^{\lambda_2}_{1,1}$ and $c^{\lambda_2}_{2,2}$, $X=1$ is a multiple root of $\varphi_M(X)$, 
but $s=\pm 1$ is not compatible with this requirement. We have reached a contradiction. 
\end{proof}

By Proposition \ref{D prop local}, Proposition \ref{D I-2 cellular}, Proposition \ref{D II}, 
Proposition \ref{D IV}, Lemma \ref{D V-5} and other lemmas in this section, we obtain the following theorem.

\begin{thm}
\label{D main result}
Assume that $\cha K\neq 2$.
\begin{enumerate}
\item Let $A$ be a basic $K$-algebra of infinite type. Then the following are equivalent.
\begin{enumerate}
\item $A$ is a standard domestic self-injective cellular algebra.
\item $A$ is isomorphic to one of the algebras 
$K[X,Y]/(X^2, Y^2)$, $\Lambda'(T(0,1))$, $\Gamma^{(0)}(T(1),v)$, 
$\Lambda(T,v_{1},v_{2})$ where $T$ is a straight line, or $\Gamma^{(2)}(T_{l}^{0},v_1,v_2)$ for $l\in\{-1,0\}$.
\end{enumerate}

\item There is no non-standard domestic self-injective cellular algebra.
\end{enumerate}
\end{thm}

\remark
We review the definitions of the algebras in Theorem \ref{D main result}. The algebra $\Lambda(T,v_1,v_2)$ 
is a Brauer graph algebra whose Brauer graph $T$ has exactly two exceptional vertices such that 
their multiplicities are both two, and the other algebras are as follows. 
\begin{itemize}
\item
$\Lambda'(T(0,1))=KQ/I$ where the quiver $Q$ is
\[
\begin{xy}
(0,0) *[o]+{0}="A", (15,0) *[o]+{1}="B",
\ar @<1mm> "A";"B"^{\gamma_0}
\ar @<1mm> "B";"A"^{\gamma_1}
\SelectTips{eu}{}
\ar @(dl,ul) "A";"A"^{\delta_0}
\ar @(ur,dr) "B";"B"^{\alpha_1}
\end{xy}
\]
and the ideal $I$ is generated by 
$\gamma_0\alpha_1$, $\alpha_1\gamma_1$, $\gamma_1\gamma_0$, $\delta_0^2$, 
$\alpha_1-\gamma_1\delta_0\gamma_0$, $\delta_0\gamma_0\gamma_1-\gamma_0\gamma_1\delta_0$,
\item
$\Gamma^{(0)}(T(1),v)=KQ/I$ where the quiver $Q$ is
\[
\begin{xy}
(0,0) *[o]+{a}="A", (15,0) *[o]+{b}="B", (30,0) *[o]+{1}="C",
\ar @<1mm> "A";"B"^{\delta_a}
\ar @<1mm> "B";"C"^{\gamma_b}
\ar @<1mm> "B";"A"^{\delta_b}
\ar @<1mm> "C";"B"^{\gamma_1}
\end{xy}
\]
and the ideal $I$ is generated by 
$\gamma_b\gamma_1\delta_b\delta_a\gamma_b$, $\gamma_1\delta_b\delta_a\gamma_b\gamma_1$, 
$\delta_a\gamma_b\gamma_1\delta_b\delta_a$, $\delta_b\delta_a\gamma_b\gamma_1\delta_b$, $\gamma_1\gamma_b$,
$\delta_a\gamma_b\gamma_1-\delta_a\delta_b\delta_a$, $\gamma_b\gamma_1\delta_b-\delta_b\delta_a\delta_b$,
\item
$\Gamma^{(2)}(T^0_{-1},v_1,v_2)=KQ/I$ where the quiver $Q$ is
\[
\begin{xy}
(0,0) *[o]+{-1}="A", (15,0) *[o]+{0}="B", (30,0) *[o]+{w}="C",
\ar @<1mm> "A";"B"^{\alpha_{-1}}
\ar @<1mm> "B";"C"^{\gamma_1}
\ar @<1mm> "B";"A"^{\beta_0}
\ar @<1mm> "C";"B"^{\gamma_2}
\end{xy}
\]
and the ideal $I$ is generated by 
$\alpha_{-1}(\beta_0\alpha_{-1})^2$, $\gamma_2(\beta_0\alpha_{-1})^2$, $(\beta_0\alpha_{-1})^2\beta_0$, 
$(\beta_0\alpha_{-1})^2\gamma_1$, $\gamma_2\gamma_1\gamma_2\beta_0$,
$\alpha_{-1}\beta_0\alpha_{-1}\gamma_1$, $\beta_0\alpha_{-1}-\gamma_1\gamma_2$,
\item
$\Gamma^{(2)}(T^0_0,v_1,v_2)=KQ/I$ where the quiver $Q$ is
\[
\begin{xy}
(0,0) *[o]+{-1}="A", (15,0) *[o]+{0}="B", (30,0) *[o]+{w}="C",
\ar @<1mm> "A";"B"^{\alpha_{-1}}
\ar @<1mm> "B";"C"^{\gamma_1}
\ar @<1mm> "B";"A"^{\beta_0}
\ar @<1mm> "C";"B"^{\gamma_2}
\SelectTips{eu}{}
\ar @(ru,lu) "B";"B"_{\alpha_0}
\end{xy}
\]
and the ideal $I$ is generated by
$\alpha_{-1}\alpha_0$, $\alpha_0\beta_0$, $\alpha_0^2-\beta_0\alpha_{-1}$, 
$\alpha_{-1}\gamma_1$, $\gamma_2\beta_0$, $\alpha_0\gamma_1$, 
$\gamma_2\alpha_0$ and $\beta_0\alpha_{-1}-\gamma_1\gamma_2$. 
\end{itemize}

%%%%%%%%%%%%%%%%%%%%%%%%%%%%%%%%%%%%%%%%%%%%%%%%%%%%%%%%%%%%%%%
\section{Self-injective cellular algebras of polynomial growth} 
%%%%%%%%%%%%%%%%%%%%%%%%%%%%%%%%%%%%%%%%%%%%%%%%%%%%%%%%%%%%%%%
In this section, we determine basic self-injective cellular algebras of polynomial growth 
which are not domestic.
By Proposition\;\ref{C Prop decom matrix} and  Proposition\;\ref{C Prop Morita equiv}, they are
weakly symmetric and have nonsingular Cartan matrix.  
In \cite{S}, the following two classes of self-injective $K$-algebras are classified ($\cha K\neq 2$).
\begin{itemize}
\item Basic standard nondomestic weakly symmetric algebras of polynomial growth having nonsingular 
Cartan matrices.
\item Basic non-standard nondomestic self-injective algebras of polynomial growth.
\end{itemize}
%%%

Hence, our task is to check whether an algebra from the two classification results admits a cellular algebra structure. 

Before we recall the two classification results in the next two theorems, assuming that the base field $K$ has an odd characteristic, 
several remarks are in order.

\remark
In the classification of basic standard nondomestic weakly symmetric algebras of polynomial growth 
in \cite[Theorem 5.9]{S}, 
the defining relations for the algebra $A_{13}$ contains $\gamma\beta=0$, but we correct it to $\beta\gamma=0$, following \cite{BiS}.

\remark
There are basic non-standard nondomestic self-injective algebras of polynomial growth in characteristic two. They are 
$\Lambda_3(\lambda)$ and $\Lambda_m$, for $4\le m\le 10$ in \cite[Theorem 5.16]{S}.

\remark
Basic standard nondomestic weakly symmetric algebras of polynomial growth having singular Cartan matrices are also classified. 
See \cite[Theorem 5.12]{S}.

%%%%%%%%%%%%%%%%%%%%%%%%%%%%

\begin{thm}[{\cite[Theorem\;5.13, Theorem\;5.16]{S}}]
\label{P theorem nonstandard nondomestic}
Assume that $\cha K\neq 2$. Then
a basic $K$-algebra $A$ is non-standard nondomestic self-injective of polynomial
 growth if and only if $\cha K=3$ and $A$ is isomorphic to one of the following bound quiver algebras.
\[\Lambda_1:
\begin{xy}
(0,0) *[o]+{1}="A", (10,0) *[o]+{2}="B",
\ar @(lu,ld) "A";"A"_{\alpha}
\ar @<2pt> "A";"B"^{\gamma}
\ar @<2pt> "B";"A"^{\beta}
\end{xy}
\tiny{ \begin{array}{c}
  \alpha^2=\gamma\beta \\
 \beta\alpha\gamma=\beta\alpha^2\gamma \\
 \beta\alpha\gamma\beta=0\\
 \gamma\beta\alpha\gamma=0\\
\end{array}
}
\hspace{10pt}
\Lambda_2:
\begin{xy}
(0,0) *[o]+{1}="A", (10,0) *[o]+{2}="B",
\ar @(lu,ld) "A";"A"_{\alpha}
\ar @<2pt> "A";"B"^{\gamma}
\ar @<2pt> "B";"A"^{\beta}
\end{xy}
\tiny{ \begin{array}{c}
  \alpha^2\gamma=0 \\
  \beta\alpha^2=0\\
  \gamma\beta\gamma=0\\
  \beta\gamma\beta=0\\
   \beta\gamma=\beta\alpha\gamma \\
   \alpha^3=\gamma\beta\\
\end{array}
}
\] 

\end{thm}

%%%%%

\begin{thm}[{\cite[Theorem\;5.3, Theorem\;5.9]{S}}]
\label{P theorem standard nondomestic}
A basic $K$-algebra $A$ is standard nondomestic weakly symmetric of polynomial growth and has 
 nonsingular Cartan matrix if and only if $A$ is isomorphic to one of the following bound quiver algebras.

\[\begin{array}{c}

\begin{array}{c}
A_1(\lambda):\\
{\text {\tiny $\lambda\in K\setminus\{0,1\}$}}\\
\end{array}
\begin{xy}
(0,0) *[o]+{1}="A", (8,0)*[o]+{2}="B", (16,0)*[o]+{3}="C",
\ar @<2pt> "A";"B"^{\alpha}
\ar @<2pt> "B";"A"^{\gamma}
\ar @<2pt> "B";"C"^{\sigma }
\ar @<2pt> "C";"B"^{\beta}
\end{xy}
{\tiny \begin{array}{c}
\alpha\gamma\alpha=\alpha\sigma\beta\\
\beta\gamma\alpha=\lambda\beta\sigma\beta\\
  \gamma\alpha\gamma=\sigma\beta\gamma \\
 \gamma\alpha\sigma=\lambda\sigma\beta\sigma \\
\end{array}
}
\hspace{5pt}
 \begin{array}{c}
A_2(\lambda):\\
{\text {\tiny $\lambda\in K\setminus\{0,1\}$}}\\
\end{array}

\begin{xy}
(0,0) *[o]+{1}="A", (10,0)*[o]+ {2}="B",
\ar @(lu,ld) "A";"A"_{\alpha}
\ar @(ru,rd) "B";"B"^{\beta}
\ar @<2pt> "A";"B"^{\sigma}
\ar @<2pt> "B";"A"^{\gamma}
\end{xy}
\tiny{ \begin{array}{c}
  \alpha^2=\sigma\gamma \\
 \lambda\beta^2=\gamma\sigma \\
 \gamma\alpha=\beta\gamma \\
 \sigma\beta=\alpha\sigma \\
\end{array}
}

\\\\
\begin{array}{lll}

A_3:
\begin{xy}
(0,-3)*[o]+{1}="A", (6,1)*[o]+{2}="B", (6,9)*[o]+{3}="C", (12,-3)*[o]+{4}="D", 
\ar @<2pt>"A";"B"^{\alpha}
\ar @<2pt>"B";"A"^{\beta}
\ar @<2pt>"B";"C"^{\delta}
\ar @<2pt>"C";"B"^{\gamma}
\ar @<2pt>"B";"D"^{\epsilon}
\ar @<2pt>"D";"B"^{\zeta}
\end{xy}
{\tiny \begin{array}{c}
  \beta\alpha+\delta\gamma+\epsilon\zeta=0 \\
  \alpha\beta=0 \\
  \gamma\delta=0 \\
 \zeta\epsilon=0 \\
 \end{array}
}

&

A_4:
\begin{xy}
(0,-3) *[o]+{1}="A", (6,1) *[o]+{2}="B", (6,9) *[o]+{3}="C", (12,-3) *[o]+{4}="D", 
\ar @<2pt>"A";"B"^{\alpha}
\ar @<2pt>"B";"A"^{\beta}
\ar @<2pt>"B";"C"^{\delta}
\ar @<2pt>"C";"B"^{\gamma}
\ar @<2pt>"B";"D"^{\epsilon}
\ar @<2pt>"D";"B"^{\zeta}
\end{xy}
{\tiny \begin{array}{c}
  \beta\alpha+\delta\gamma+\epsilon\zeta=0 \\
  \alpha\beta=0 \\
  \gamma\epsilon=0 \\
 \zeta\delta=0 \\
 \end{array}
}

&

A_5:
\begin{xy}
(0,0) *[o]+{1}="A", (10,0) *[o]+{2}="B",
\ar @(lu,ld) "A";"A"_{\alpha}
\ar @<2pt> "A";"B"^{\gamma}
\ar @<2pt> "B";"A"^{\beta}
\end{xy}
\tiny{ \begin{array}{c}
  \alpha^2=\gamma\beta \\
 \beta\alpha\gamma=0 \\
\end{array}
}\\\\

A_6:
\begin{xy}
(0,0) *[o]+{1}="A", (10,0) *[o]+{2}="B",
\ar @(lu,ld) "A";"A"_{\alpha}
\ar @<2pt> "A";"B"^{\gamma}
\ar @<2pt> "B";"A"^{\beta}
\end{xy}
\tiny{ \begin{array}{c}
  \alpha^3=\gamma\beta \\
 \beta\gamma=0 \\
 \beta\alpha^2=0 \\
 \alpha^2\gamma=0 \\
\end{array}
}
 &
A_7:
\begin{xy}
(0,0) *[o]+{1}="A", (7,0) *[o]+{2}="B", (14,0) *[o]+{3}="C", (21,0) *[o]+{4}="D",
\ar @<2pt> "A";"B"^{\alpha}
\ar @<2pt> "B";"A"^{\beta}
\ar @<2pt> "B";"C"^{\delta}
\ar @<2pt> "C";"B"^{\gamma}
\ar @<2pt> "C";"D"^{\epsilon}
\ar @<2pt> "D";"C"^{\zeta}
\end{xy}
{\tiny \begin{array}{c}
  \beta\alpha=\delta\gamma \\
 \gamma\delta=\epsilon\zeta \\
 \alpha\delta\epsilon=0 \\
 \zeta\gamma\beta=0 \\
\end{array}
}
&
A_8:
\begin{xy}
(0,5) *[o]+{1}="A", (0,-5) *[o]+{2}="B", (10,-5) *[o]+{3}="C", (10,5) *[o]+{4}="D",
\ar "A";"B"_{\sigma}
\ar "B";"C"_{\zeta}
\ar "C";"D"_{\gamma}
\ar "D";"A"_{\delta}
\ar @<2pt> "A";"C"^{\alpha}
\ar @<2pt> "C";"A"^{\beta}
\end{xy}
{\tiny \begin{array}{c}
  \alpha\beta\alpha=\sigma\zeta \\
  \beta\alpha\beta=\gamma\delta\\
  \zeta\beta\alpha=0 \\
 \delta\alpha\beta=0 \\
 \beta\alpha\gamma=0 \\
 \alpha\beta\sigma=0 \\
 \zeta\gamma=0 \\
 \delta\sigma=0\\
\end{array}
}\\\\

A_9:
\begin{xy}
(0,5) *[o]+{1}="A", (0,-5) *[o]+{2}="B", (10,-5) *[o]+{3}="C", (10,5) *[o]+{4}="D",
\ar "A";"B"_{\alpha}
\ar @<-2pt>"B";"C"_{\sigma}
\ar @<-2pt>"C";"B"_{\beta}
\ar @<-2pt> "C";"D"_{\gamma}
\ar @<-2pt> "D";"C"_{\epsilon}
\ar "D";"A"_{\delta}
\end{xy}
{\tiny \begin{array}{c}
  \delta\alpha=\epsilon\beta \\
  \gamma\epsilon=\beta\sigma\\
 \alpha\sigma\beta=0 \\
 \epsilon\gamma\delta=0 \\
 \sigma\gamma\epsilon\gamma=0 \\
 
\end{array}
}
&
A_{10}:
\begin{xy}
(8,10) *[o]+{1}="A", (8,1) *[o]+{2}="B", (16,-5) *[o]+{3}="C", (0,-5) *[o]+{4}="D",
\ar @<2pt>"A";"B"^{\beta}
\ar @<2pt>"B";"A"^{\alpha}
\ar "B";"C"^{\delta}
\ar "C";"D"^{\gamma}
\ar "D";"B"^{\epsilon}
\end{xy}
{\tiny \begin{array}{c}
  \epsilon\alpha\beta=\epsilon\delta\gamma\epsilon \\
  \alpha\beta\delta=\delta\gamma\epsilon\delta\\
 \beta\alpha=0 \\
 (\gamma\epsilon\delta)^2\gamma=0 \\
\end{array}
}

&
A_{11}:
\begin{xy}
(0,0) *[o]+{1}="A", (7,0) *[o]+{2}="B", (14,0) *[o]+{3}="C", (21,0) *[o]+{4}="D",
\ar @<2pt> "A";"B"^{\beta}
\ar @<2pt> "B";"A"^{\alpha}
\ar @<2pt> "B";"C"^{\eta }
\ar @<2pt> "C";"B"^{\gamma}
\ar @<2pt> "C";"D"^{\zeta}
\ar @<2pt> "D";"C"^{\delta}
\end{xy}
{\tiny \begin{array}{c}
  \gamma\alpha\beta=\gamma\eta\gamma \\
 \alpha\beta\eta=\eta\gamma\eta \\
 \beta\alpha=0 \\
 \delta\gamma=0 \\
 \eta\zeta=0\\
 (\gamma\eta)^2=\zeta\delta
\end{array}
}\\\\

A_{12}:
\begin{xy}
(0,-5)*[o]+{1}="A",(8,7)*[o]+{2}="B", (16,-5)*[o]+{3}="C", 
\ar "A";"B"^{\alpha}
\ar "B";"C"^{\gamma}
\ar @<2pt>"C";"A"^{\beta}
\ar @<2pt>"A";"C"^{\delta}
\end{xy}
{\tiny \begin{array}{c}
\delta\beta\delta=\alpha\gamma \\
\gamma\beta\alpha=0 \\
\beta(\delta\beta)^3=0\\
\end{array}
}

&

A_{13}:
\begin{xy}
(0,0) *[o]+{1}="A", (8,0) *[o]+{2}="B", (16,0) *[o]+{3}="C",
\ar @(ul,ur) "B";"B"^{\alpha}
\ar @<2pt> "A";"B"^{\beta}
\ar @<2pt> "B";"A"^{\gamma}
\ar @<2pt> "B";"C"^{\delta }
\ar @<2pt> "C";"B"^{\sigma}
\end{xy}
{\tiny \begin{array}{c}
\alpha^2=\gamma\beta\\
\beta\delta=0\\
 \beta\gamma=0 \\
 \sigma\gamma=0 \\
 \alpha\delta=0 \\
 \sigma\alpha=0 \\
 \alpha^3=\delta\sigma \\
\end{array}
}

&
A_{14}:
\begin{xy}
(0,0) *[o]+{1}="A", (8,0) *[o]+{2}="B", (16,0) *[o]+{3}="C",
\ar @<2pt> "A";"B"^{\alpha}
\ar @<2pt> "B";"A"^{\beta}
\ar @<2pt> "B";"C"^{\delta }
\ar @<2pt> "C";"B"^{\gamma}
\end{xy}
{\tiny \begin{array}{c}
\beta\alpha=(\delta\gamma)^2 \\
\alpha\delta\gamma\delta=0 \\
  \gamma\delta\gamma\beta=0 \\
 \alpha\beta=0 \\
 \end{array}
}

\end{array}\\\\

A_{15}:
\begin{xy}
(0,-5) *[o]+{1}="A",(8,7) *[o]+{2}="B", (16,-5) *[o]+{3}="C", 
\ar @(lu,ld) "A";"A"_{\alpha} 
\ar "A";"B"^{\sigma}
\ar "B";"C"^{\gamma}
\ar @<2pt>"C";"A"^{\beta}
\ar @<2pt>"A";"C"^{\delta}
\end{xy}
{\tiny \begin{array}{c}
\gamma\beta\alpha=0 \\
\alpha^2=\delta\beta \\
\beta\delta=0 \\
\alpha\sigma=0\\
\alpha\delta=\sigma\gamma\\
\end{array}
} 
\hspace{10pt}

A_{16}:
\begin{xy}
(0,-5) *[o]+{1}="A",(8,7) *[o]+{2}="B", (16,-5) *[o]+{3}="C", 
\ar @(lu,ld) "A";"A"_{\alpha} 
\ar "B";"A"_{\sigma}
\ar "C";"B"_{\gamma}
\ar @<2pt>"A";"C"^{\beta}
\ar @<2pt>"C";"A"^{\delta }
\end{xy}
{\tiny \begin{array}{c}
\alpha\beta\gamma=0 \\
\alpha^2=\beta\delta \\
\delta\beta=0 \\
\sigma\alpha=0\\
\delta\alpha=\gamma\sigma\\
\end{array}
}\\

\end{array}
\] 
\end{thm}

%%%%%%%%%%%%%%%%%%%%%%%%%%%%

The main result of this section is the following.

%%%

\begin{thm}
\label{P main}
Assume that $\cha K\neq 2$.
\begin{enumerate} 
\item Let $A$ be a basic $K$-algebra. Then the following are equivalent.
\begin{enumerate}
\item $A$ is a standard nondomestic self-injective cellular algebra
 of polynomial growth.
\item $A$ is isomorphic to either $A_1(\lambda)$, $A_2(\lambda)$ 
$(\lambda\in K\setminus\{0,1\})$, $A_4$, $A_7$ or $A_{11}$.
\end{enumerate}

\item There is no non-standard nondomestic self-injective cellular algebra
 of polynomial growth. 
 \end{enumerate}
\end{thm}

%%%
\subsection{Proof of Theorem\;\ref{P main}\;(i)}
In this subsection, we prove Theorem \ref{P main}(i).
By Proposition \ref{C Prop simple}(iii), we may exclude the 
algebras $A_8$, $A_9$, $A_{10}$, $A_{12}$, $A_{15}$ and $A_{16}$ 
immediately, so that it suffices to consider 
$A_1(\lambda)$, $A_2(\lambda)$, $A_3$, $A_4$, $A_5$, $A_6$, $A_7$, $A_{11}$, $A_{13}$ and $A_{14}$.
%%%%%%%%%%%%%%%%%%%%%%%%%%%%%%%%%%%%%%%%%%%%%%%%%%%%%%%%%%%%%%%
%%%%%%%%%%%%%%%%%%%%%%%%%%%%%%%% A_1 %%%%%%%%%%%%%%%%%%%%%%%%%%
\begin{lem}
\label{P A1}
$A_1(\lambda)$ is cellular.
\end{lem}
\begin{proof}
Let $A=A_1(\lambda)$ with $\lambda\in K\setminus\{0,1\}$.
Since $\lambda\neq 1$, we have 
\begin{enumerate}
\item $\beta\sigma\beta\gamma=\beta\gamma\alpha\gamma=\lambda\beta\sigma\beta\gamma=0.$
\item $\alpha\gamma\alpha\gamma\alpha=\alpha\sigma\beta\gamma\alpha=\lambda\alpha\sigma\beta\sigma\beta=\lambda\alpha\gamma\alpha\sigma\beta=\lambda\alpha\gamma\alpha\gamma\alpha=0.$
\item $\alpha\sigma\beta\sigma=\alpha\gamma\alpha\sigma=\lambda\alpha\sigma\beta\sigma=0.$
\end{enumerate}
Then, we may read the radical series of $P_1$ from the $K$-basis
\[e_1, \; \gamma, 
\begin{array}{cc}
 \alpha\gamma\\
 \beta\gamma\\
\end{array},\;
\gamma\alpha\gamma=\sigma\beta\gamma,\;
 \alpha\gamma\alpha\gamma=\alpha\sigma\beta\gamma,
  \]
and we see that $P_1$ has submodules 
$K\alpha\gamma\oplus K\beta\gamma\oplus K\gamma\alpha\gamma\oplus K(\alpha\gamma)^2$ 
and $K(\alpha\gamma)^2$. Similarly, $P_2$ and $P_3$ have the $K$-bases
\[e_2, 
\begin{array}{cc}
 \alpha\\ 
 \beta\\
\end{array}, 
\begin{array}{cc}
 \gamma\alpha\\
 \sigma\beta\\
\end{array},
\begin{array}{cc} 
  \alpha\gamma\alpha=\alpha\sigma\beta\\
  \beta\gamma\alpha=\lambda\beta\sigma\beta\\
\end{array},\;  
\gamma\alpha\gamma\alpha,
\]
\[e_3,\; \sigma, 
\begin{array}{cc}
 \alpha\sigma\\
 \beta\sigma\\
\end{array},\;
\gamma\alpha\sigma=\lambda\sigma\beta\sigma,
 \beta\gamma\alpha\sigma=\lambda\beta\sigma\beta\sigma,\\
\]

\medskip
\noindent
and $P_3$ has submodules 
$K\alpha\sigma\oplus K\beta\sigma\oplus K\gamma\alpha\sigma\oplus K\beta\gamma\alpha\sigma$ 
and $K\beta\gamma\alpha\sigma$.

Let $\iota$ be the anti-involution induced by swapping $\alpha$ and $\gamma$,
$\beta$ and $\sigma$, respectively. It is straightforward to check that it is well-defined.
Take a totally ordered set $\Lambda=\{\lambda_1<\lambda_2<\lambda_3<\lambda_4<\lambda_5<\lambda_6\}$ 
and define
\[\cT(\lambda_1)=\cT(\lambda_5)=\cT(\lambda_6)=\{1\},\ \cT(\lambda_2)=\cT(\lambda_3)=\{1,2\},\ \cT(\lambda_4)=\{1,2,3\}.\]
Then we can construct a cellular basis of $A$ as follows.
\begin{align*}
& (c_{1,1}^{\lambda_1})=(e_2),\  
(c_{i,j}^{\lambda_2})_{i,j\in \cT(\la_2)}
    =\begin{pmatrix}
    e_1 & \alpha \\
    \gamma & \gamma\alpha \\
    \end{pmatrix},\ 
(c_{i,j}^{\lambda_3})_{i,j\in \cT(\la_3)}
    =\begin{pmatrix}
    e_3 & \beta \\
    \sigma & \sigma\beta \\
    \end{pmatrix},\\[7pt]
& (c_{i,j}^{\lambda_4})_{i,j\in \cT(\la_4)}
    =\begin{pmatrix}
    \alpha\gamma & \alpha\sigma & \alpha\gamma\alpha \\
    \beta\gamma & \lambda\beta\sigma & \beta\gamma\alpha \\
    \gamma\alpha\gamma & \gamma\alpha\sigma & \gamma\alpha\gamma\alpha \\
    \end{pmatrix},\ 
 (c_{1,1}^{\lambda_5})=(\alpha\gamma\alpha\gamma),\ 
 (c_{1,1}^{\lambda_6})=(\beta\gamma\alpha\sigma).\\
 \end{align*}
Hence, $A_1(\la)$ is cellular.
\end{proof}

%%%%%%%%%%%%%%%%%%%%%%%%%%%%%%%%%%%%%%%%%%%%%%%%%%%%%%%%%%%%%%%%%%%%%%%%%%%%%%%%%%%%%
%%%%%%%%%%%%%%%%%%%%%%%%%%%%%%%%%%%%%%%%% A_2 %%%%%%%%%%%%%%%%%%%%%%%%%%%%%%%%%%%%%%%

\begin{lem}
\label{P A2}
$A_2(\lambda)$ is cellular.
\end{lem}
\begin{proof}
Let $A=A_2(\lambda)$ with $\lambda\in K\setminus\{0,1\}$.
Since $\lambda\neq 1$, we have 
\begin{enumerate}
\item $\beta^2\gamma=\beta\gamma\alpha=\gamma\alpha^2=\gamma\sigma\gamma=\lambda\beta^2\gamma=0.$
\item $\alpha^2\sigma=\sigma\gamma\sigma=\lambda\sigma\beta^2=
\lambda\alpha\sigma\beta=\lambda\alpha^2\sigma=0.$
\end{enumerate}
Then, $P_1$ has the following $K$-basis.
\[e_1, 
\begin{array}{cc}
 \alpha\\
 \gamma\\
\end{array},
\begin{array}{cc}
 \alpha^2=\sigma\gamma\\
 \gamma\alpha=\beta\gamma\\
\end{array},\;
\alpha^3=\sigma\gamma\alpha=\alpha\sigma\gamma.
  \]
On the other hand, $P_2$ has the following $K$-basis.
\[e_2, 
\begin{array}{cc}
 \sigma\\
 \beta\\
\end{array},
\begin{array}{cc}
 \alpha\sigma=\sigma\beta\\
 \gamma\sigma=\lambda\beta^2\\
\end{array},\;
\gamma\alpha\sigma=\gamma\sigma\beta=\lambda\beta^3.
  \]
Let $\iota$ be the anti-involution induced by fixing $\alpha$ and $\beta$, and swapping 
$\sigma$ and $\gamma$. Take a totally ordered set
 $\Lambda=\{\lambda_1<\lambda_2<\lambda_3<\lambda_4<\lambda_5<\lambda_6\}$ 
and define
\[\cT(\lambda_1)=\cT(\lambda_2)=\cT(\lambda_5)=\cT(\lambda_6)=\{1\},\ 
\cT(\lambda_3)=\cT(\lambda_2)=\{1,2\}.\]
Then one can construct a cellular basis of $A$ as follows.
\begin{align*}
& (c_{1,1}^{\lambda_1})=(e_1),\
 (c_{1,1}^{\lambda_2})=(e_2),\quad \  
(c_{i,j}^{\lambda_3})_{i,j\in \cT(\la_3)}
    =\begin{pmatrix}
    \alpha & \sigma \\
    \gamma & \beta \\
    \end{pmatrix},\\[7pt] 
& (c_{i,j}^{\lambda_4})_{i,j\in \cT(\la_4)}
    =\begin{pmatrix}
    \alpha^2 & \alpha\sigma \\
    \gamma\alpha & \beta^2 \\
    \end{pmatrix},\ 
(c_{1,1}^{\lambda_5})
    =(\alpha^3),\ 
 (c_1^{\lambda_6})=(\beta^3).
 \end{align*}
Hence, $A_2(\la)$ is cellular.
\end{proof}
%%%%%%%%%%%%%%%%%%%%%%%%%%%%%%%%%%%%%%%%%%%%%%%%%%%%%%%%%%%%%%%%%%%%%%%%%%%%%%%%%%%%%
%%%%%%%%%%%%%%%%%%%%%%%%%%%%%%%%%%%%%% A_3 %%%%%%%%%%%%%%%%%%%%%%%%%%%%%%%%%%%%%%%%%%
\begin{lem}
\label{P A3}
$A_3$ is not cellular.
\end{lem}
\begin{proof}
Let $A=A_3$. We put $x=\beta\alpha$, $y=\delta\gamma$ and $z=\epsilon\zeta$.
Since $x+y+z=0$ and $x^2=y^2=z^2=0$, we obtain
\[xy=yz=zx=-yx=-zy=-xz.\]
This shows that $xy\in \Soc A$. 
Hence, we have
\begin{align*}
P_1 &=
Ke_1\oplus K\beta\oplus K\gamma\beta\oplus K\zeta\beta\oplus Ky\beta\oplus K\alpha y \beta, \\
P_2 &=
Ke_2\oplus K\alpha\oplus K\gamma\oplus K\zeta\oplus Kx\oplus Ky\oplus 
K\gamma x\oplus K\zeta y\oplus K\alpha z\oplus Kxy, \\
P_3 &=
Ke_3\oplus K\delta\oplus K\alpha\delta\oplus K\zeta\delta\oplus Kx\delta\oplus K\gamma x \delta, \\
P_4 &=
Ke_4\oplus K\epsilon\oplus K\alpha\epsilon\oplus K\gamma\epsilon\oplus Kx\epsilon\oplus K\zeta x \epsilon,
\end{align*}
and the Cartan matrix is given by
\[\bC=
\begin{pmatrix}
2 & 2 & 1 & 1\\
2 & 4 & 2 & 2\\
1 & 2 & 2 & 1\\
1 & 2 & 1 & 2\\ 
\end{pmatrix}.
\]
 Suppose that $A$ is a cellular algebra. Then, the decomposition matrix is either
\[ \bD=\begin{pmatrix}
1 & 1 & 1 & 1\\
1 & 1 & 0 & 0\\
0 & 1 & 1 & 0\\
0 & 1 & 0 & 1\\ 
\end{pmatrix} \quad \text{or} \quad
\begin{pmatrix}
1 & 1 & 1 & 0\\
1 & 1 & 0 & 1\\
0 & 1 & 1 & 1\\
0 & 1 & 0 & 0\\ 
\end{pmatrix}.\] 
modulo rearrangement of rows, by Proposition \ref{C Prop decom matrix}.
This contradicts Lemma\;\ref{C Lem min}.
\end{proof}
%%%%%%%%%%%%%%%%%%%%%%%%%%%%%%%%%%%%%%%%%%%%%%%%%%%%%%%%%%%%%%%%%%%%%%%%%%%%%%%%%%%%

%%%%%%%%%%%%%%%%%%%%%%%%%%%%%%%%%%%%%%%%%%%%%%%%%%%%%%%%%%%%%%%%%%%%%%%%%%%%%%%%%%%%
%%%%%%%%%%%%%%%%%%%%%%%%%%%%%%%%%%%%%%%% A_4 %%%%%%%%%%%%%%%%%%%%%%%%%%%%%%%%%%%%%%%

\begin{lem}
\label{P A4}
$A_4$ is cellular.
\end{lem}
\begin{proof}
Let $A=A_4$. We put $x=\beta\alpha$, $y=\delta\gamma$ and $z=\epsilon\zeta$.
Since $x+y+z=0$ and $x^2=yz=zy=0$, we obtain
\[xy=yx=z^2=-y^2=-xz=-zx.\]
This shows that $xy\in \Soc A$.
Then, we may read the radical series of $P_1$ from the following $K$-basis.
\[
e_1,\; \beta,
\begin{array}{c}
\gamma\beta\\
\zeta\beta\\
\end{array},\;
y\beta=-z\beta,\; \alpha y \beta=-\alpha z \beta.
\] 
Similarly, $P_2$, $P_3$ and $P_4$ have the following $K$-bases.

\[e_2,
\begin{array}{c}
\alpha\\
\gamma\\
\zeta\\
\end{array},
\begin{array}{c}
x\\
y\\
\end{array},
\begin{array}{c}
\zeta x=-\zeta z\\
\gamma y=-\gamma x\\
\alpha z=-\alpha y\\
\end{array},\;
 xy=yx=z^2.
\] 

\[
e_3,\; \delta,
\begin{array}{c}
\alpha\delta\\
\gamma\delta\\
\end{array},
x\delta=-y\delta,\; \gamma x \delta=-\gamma y \delta.
\]

\[
e_4,\; \epsilon,
\begin{array}{c}
\alpha\epsilon\\
\zeta\epsilon\\
\end{array},
x\epsilon=-z\epsilon,\; \zeta x \epsilon=-\zeta z \epsilon.
\]

Let $\iota$ be the anti-involution induced by swapping $\alpha$ and $\beta$, 
$\gamma$ and $\delta$, $\epsilon$ and $\zeta$, respectively.  
Take a totally ordered set
 $\Lambda=\{\lambda_1<\lambda_2<\lambda_3<\lambda_4<\lambda_5<\lambda_6\}$ and define
\[\cT(\lambda_1)=\cT(\lambda_6)=\{1\},\ \cT(\lambda_2)=\cT(\lambda_5)=\{1,2\},\
\cT(\lambda_3)=\cT(\lambda_4)=\{1,2,3\}.\]
Then we can construct a cellular basis of $A$ as follows.
\begin{align*}
& (c_{1,1}^{\lambda_1})=(e_3),\
 (c_{i,j}^{\lambda_2})_{i,j\in \cT(\la_2)}
    =\begin{pmatrix}
    e_2 & \delta \\
    \gamma & \gamma\delta \\
    \end{pmatrix},\\[7pt]
& (c_{i,j}^{\lambda_3})_{i,j\in \cT(\la_3)}
    =\begin{pmatrix}
    e_1 & \alpha & \alpha\delta \\
    \beta & \beta\alpha & \beta\alpha\delta \\
    \gamma\beta & \gamma\beta\alpha & \gamma\beta\alpha\delta \\
    \end{pmatrix},\   
  (c_{i,j}^{\lambda_4})_{i,j\in \cT(\la_4)}
    =\begin{pmatrix}
    e_4 & \zeta & \zeta\beta \\
    \epsilon & \epsilon\zeta & \epsilon\zeta\beta \\
    \alpha\epsilon & \alpha\epsilon\zeta & \alpha\epsilon\zeta\beta \\
    \end{pmatrix},\\[7pt]
& (c_{i,j}^{\lambda_5})_{i,j\in \cT(\la_5)}
    =\begin{pmatrix}
    \zeta\epsilon & \zeta\epsilon\zeta \\
    \epsilon\zeta\epsilon & \epsilon\zeta\epsilon\zeta \
    \end{pmatrix},\
  (c_{1,1}^{\lambda_6})=(\zeta\epsilon\zeta\epsilon).
 \end{align*}
Hence, $A_4$ is cellular.
\end{proof}

%%%%%%%%%%%%%%%%%%%%%%%%%%%%%%%%%%%%%%%%%%%%%%%%%%%%%%%%%%%%%%%%%%%%%%%%%%%%%%%%%%%%%
%%%%%%%%%%%%%%%%%%%%%%%%%%%%%%%%%%%%%%%%% A_5 %%%%%%%%%%%%%%%%%%%%%%%%%%%%%%%%%%%%%%%

\begin{lem}
\label{P A5}
$A_5$ is not cellular.
\end{lem}
\begin{proof}
Suppose that $A=A_5$ is a cellular algebra with cell datum $(\Lambda,\cT,\cC,\iota)$. Then, 
explicit computation shows
\begin{align*}
P_1 &=
Ke_1\oplus K\alpha\oplus K\beta\oplus K\beta\alpha\oplus K\alpha^2\oplus K\alpha^3
\oplus K\beta\alpha^2\oplus K\alpha^4, \\ 
P_2 &=
Ke_2\oplus K\gamma\oplus K\alpha\gamma\oplus K\beta\gamma\oplus K\alpha^2\gamma\oplus K\beta\alpha^2\gamma.
\end{align*}
Thus, the Cartan matrix is given by
\[
\bC=\begin{pmatrix}5& 3 \\3 & 3\\\end{pmatrix}
\]
and the decomposition matrix is either
 \[
\bD=\bordermatrix{
& \mu_1 & \mu_2 \cr
\lambda_1 & 2 & 1 \cr
\lambda_2 & 1 & 1 \cr
\lambda_3 & 0 & 1 \cr
}
\quad \text{or}\quad
\bordermatrix{
& \mu_1 & \mu_2 \cr
\lambda_1 & 1 & 1 \cr
\lambda_2 & 1 & 1 \cr
\lambda_3 & 1 & 1 \cr
\lambda_4 & 1 & 0 \cr
\lambda_5 & 1 & 0 \cr
}. 
\] 
modulo rearrangement of rows. 
The first case contradicts with Lemma\;\ref{C Lem min}. In the second case,
we may assume that $\la_4$ is minimal and 
$\la_5$ is maximal in the partial order, and we may extend the partial order to the total order 
$4<3<2<1<5$, by Lemma\;\ref{C Lem min}.
Then Proposition\;\ref{C Prop decom matrix} implies that
$P(\mu_1)\supset P(\mu_1)_{\lambda_5}\cong \D(\lambda_5)\cong L(\mu_1)$ and $P(\mu_2)_{\lambda_5}=0$, 
so that we have $P(\mu_1)_{\lambda_5}=K \alpha^4$. Proposition\;\ref{C Prop decom matrix} 
also implies that 
\[P(\mu_1)_{\lambda_1}/P(\mu_1)_{\lambda_5}\cong \D(\lambda_1)
\cong P(\mu_2)_{\lambda_1}/P(\mu_2)_{\lambda_5}=P(\mu_2)_{\lambda_1}.\]
Since $[\D(\lambda_1):L(\mu_1)]=[\D(\lambda_1):L(\mu_2)]=1$, we have
\[P(\mu_1)_{\lambda_1}=
K(xe_1+y\alpha+z\alpha^2+w\alpha^3)
\oplus K(x'\beta+y'\beta\alpha+z'\beta\alpha^2)\oplus K\alpha^4,\]
for some $x,y,z,w,x',y',z'\in K$.
Arguing in the similar way as in the paragraph after (\ref{F DB 2-2}), we see
that $x=y=z=0$. In particular, $L_1$ appears in $\Soc \D(\lambda_1)$.
However, $\Soc \D(\lambda_1)\cong \Soc P(\mu_2)_{\lambda_1}\cong L_2$. This is a contradiction.
\end{proof}

%%%%%%%%%%%%%%%%%%%%%%%%%%%%%%%%%%%%%%%%%%%%%%%%%%%%%%%%%%%%%%%
%%%%%%%%%%%%%%%%%%%%%%%%%%%% A6 %%%%%%%%%%%%%%%%%%%%%%%%%%%%%%%

\begin{lem}
\label{P A6}
$A_6$ is not cellular.
\end{lem}
\begin{proof}
Suppose that $A=A_6$ is a cellular algebra with cell datum $(\Lambda,\cT,\cC,\iota)$. 
Then, explicit computation shows
\begin{align*}
P_1 &=
Ke_1\oplus K\alpha\oplus K\beta\oplus K\beta\alpha\oplus K\alpha^2\oplus K\alpha^3\oplus K\alpha^4, \\
P_2 &=
Ke_2\oplus K\gamma\oplus K\alpha\gamma\oplus K\beta\alpha\gamma,
\end{align*}
and the Cartan matrix is given by
\[
\bC=\begin{pmatrix} 5 & 2 \\ 2 & 2 \end{pmatrix}.
\]
Then, 
Proposition\;\ref{C Prop decom matrix} implies that the decomposition matrix is either
 \[\bD_1=\bordermatrix{
& \mu_1 & \mu_2 \cr
\lambda_1 & 2 & 1 \cr
\lambda_2 & 1 & 0 \cr
\lambda_3 & 0 & 1 \cr
}
\quad \text{or} \quad 
\bD_2=\bordermatrix{
& \mu_1 & \mu_2 \cr
\lambda_1 & 1 & 1 \cr
\lambda_2 & 1 & 1 \cr
\lambda_3 & 1 & 0 \cr
\lambda_4 & 1 & 0 \cr
\lambda_5 & 1 & 0 \cr
}
\]
modulo rearrangement of rows.

Assume that $\bD=\bD_1$.
By Lemma\;\ref{C Lem min}, either $\lambda_2<\lambda_1<\lambda_3$ or
$\lambda_3<\lambda_1<\lambda_2$ occurs.
If $\lambda_2<\lambda_1<\lambda_3$, then we have $P_1\supset \Delta(\lambda_1)^{\oplus 2}$.
However, $P_1$ has simple socle, a contradiction.
Therefore, we assume that $\lambda_3<\lambda_1<\lambda_2$. 
Then Proposition\;\ref{C Prop decom matrix} shows that
$\lambda_1\geq \mu_1,\mu_2$, so that $\lambda_1\in \Lambda^+$. This contradicts 
$[\Delta(\lambda_1):L(\lambda_1)]=1$.
 
Assume that $\bD=\bD_2$. Then it suffices to consider 
the following three cases for the linear extension of the partial order, without loss of generality.
\begin{itemize}
\item[(a)]
$3<1<2<4<5$.
\item[(b)]
$3<1<4<2<5$.
\item[(c)]
$3<4<1<2<5$.
\end{itemize}
Note that we have $\Delta(\lambda_5)\simeq P(\mu_1)_{\lambda_5}=K \alpha^4$
and $P(\mu_2)_{\lambda_5}=0$ in the three cases.  

In the cases (b) and (c), we have
\[P(\mu_1)_{\lambda_2}=K(xe_1+y\alpha+z\alpha^2+w\alpha^3)\oplus K(x'\beta+y'\beta\alpha)\oplus 
K\alpha^4,\]
for some $x,y,z,w,x',y'\in K$. It is easy to check that $x=y=z=0$. Therefore,
$L_1$ appears in $\Soc \Delta(\lambda_2)\cong P(\mu_2)_{\lambda_2}.$
This is a contradiction.

In the case (a), we have $P(\mu_1)_{\lambda_4}=K\alpha^3\oplus K\alpha^4$
 and $P(\mu_2)_{\lambda_4}=0$. We also obtain
\[P(\mu_1)_{\lambda_2}=K(xe_1+y\alpha+z\alpha^2)\oplus K(x'\beta+y'\beta\alpha)\oplus K\alpha^3\oplus K\alpha^4,\]
for some $x,y,z,x',y'\in K$. Then one sees that $x=y=0$. Thus 
 $L_1$ appears in $\Soc \Delta(\lambda_2)\cong \Soc P(\mu_2)_{\lambda_2}\cong L_2.$
This is a contradiction.
\end{proof}

%%%%%%%%%%%%%%%%%%%%%%%%%%%%%%%%%%%%%%%%%%%%%%%%%%%%%%%%%%%%%%%%%%%%%%%%%%%%%%%%%%
%%%%%%%%%%%%%%%%%%%%%%%%%%%%%%%%%% A7 %%%%%%%%%%%%%%%%%%%%%%%%%%%%%%%%%%%%%%%%%%%%
\begin{lem}
\label{P A7}
$A_7$ is cellular.
\end{lem}
\begin{proof}
Let $A=A_7$. Then we may read the radical series of $P_1$ from the following $K$-basis.

\[
e_1,\; \beta,
\begin{array}{c}
\alpha\beta\\
\gamma\beta\\
\end{array},
\beta\alpha\beta=\delta\gamma\beta,\; \alpha \beta\alpha\beta.
\]
Similarly, $P_2$ has the $K$-basis
\[e_2,
\begin{array}{c}
\alpha\\
\gamma\\
\end{array}
,
\begin{array}{c}
\beta\alpha=\delta\gamma\\
\zeta\gamma\\
\end{array},
\begin{array}{c}
\alpha\beta\alpha\\
\gamma\beta\alpha=\epsilon\zeta\gamma\\
\end{array},\;
\beta\alpha\beta\alpha=\delta\gamma\beta\alpha, \] 
$P_3$ has the $K$-basis
\[e_3,
\begin{array}{c}
\delta\\
\zeta\\
\end{array},\;
\begin{array}{c}
\alpha\delta\\
\gamma\delta=\epsilon\zeta\\
\end{array},\;
\begin{array}{c}
\beta\alpha\delta=\delta\gamma\delta\\
\zeta\gamma\delta\\
\end{array},\;
\gamma\beta\alpha\delta=\epsilon\zeta\gamma\delta,
\]
and $P_4$ has the $K$-basis
\[
e_4,\; \epsilon,
\begin{array}{c}
\delta\epsilon\\
\zeta\epsilon\\
\end{array},\;
\gamma\delta\epsilon=\epsilon\zeta\epsilon, \;
\zeta \gamma\delta\epsilon.
\] 

Let $\iota$ be the anti-involution induced by swapping $\alpha$ and $\beta$, 
$\gamma$ and $\delta$, $\epsilon$ and $\zeta$, respectively. Take a totally ordered set
 $\Lambda=
 \{\lambda_1<\lambda_2<\lambda_3<\lambda_4<\lambda_5<\lambda_6\}$ and define
\[\cT(\lambda_1)=\cT(\lambda_6)=\{1\},\ \cT(\lambda_2)=\cT(\lambda_5)=\{1,2\},\
\cT(\lambda_3)=\cT(\lambda_4)=\{1,2,3\}.\]
Then we obtain a cellular basis of $A$ as follows.
\begin{align*}
& (c_{1,1}^{\lambda_1})=(e_4),\  
  (c_{i,j}^{\lambda_2})_{i,j\in \cT(\la_2)}
    =\begin{pmatrix}
    e_3 & \epsilon \\
    \zeta & \zeta\epsilon \\
    \end{pmatrix},\\[7pt] 
& (c_{i,j}^{\lambda_3})_{i,j\in \cT(\la_3)}
    =\begin{pmatrix}
    e_2 & \delta & \delta\epsilon \\
    \gamma & \gamma\delta & \gamma\delta\epsilon \\
    \zeta\gamma & \zeta\gamma\delta & \zeta\gamma\delta\epsilon \\
    \end{pmatrix},\   
 (c_{i,j}^{\lambda_4})_{i,j\in \cT(\la_4)}
    =\begin{pmatrix}
    e_1 & \alpha & \alpha\delta \\
    \beta & \beta\alpha & \beta\alpha\delta \\
    \gamma\beta & \gamma\beta\alpha & \gamma\beta\alpha\delta \\
    \end{pmatrix},\\[9pt] 
& (c_{i,j}^{\lambda_5})_{i,j\in \cT(\la_5)}
    =\begin{pmatrix}
    \alpha\beta & \alpha\beta\alpha \\
    \beta\alpha\beta & \beta\alpha\beta\alpha \
    \end{pmatrix},\
 (c_{1,1}^{\lambda_6})=(\alpha\beta\alpha\beta).
 \end{align*}
Hence, $A_7$ is cellular.
 \end{proof}
%%%%%%%%%%%%%%%%%%%%%%%%%%%%%%%%%%%%%%%%%%%%%%%%%%%%%%%%%%%%%%%%%
%%%%%%%%%%%%%%%%%%%%%%%%%%%% A11 %%%%%%%%%%%%%%%%%%%%%%%%%%%%%%%%
\begin{lem}
\label{P A11}
$A_{11}$ is cellular.
\end{lem}
\begin{proof}
Let $A=A_{11}$. Then, we may read the radical series of $P_1$ from the following $K$-basis.
\[
e_1,\; \alpha,\; \gamma\alpha,\; \eta\gamma\alpha,\; \beta\eta\gamma\alpha.
\]
Similarly, $P_2$ has the $K$-basis
\[e_2,
\begin{array}{c}
\beta\\
\gamma\\
\end{array},
\begin{array}{c}
\alpha\beta \\
\eta\gamma \\
\end{array},
\begin{array}{c}
\gamma\alpha\beta=\gamma\eta\gamma\\
\beta\eta\gamma\\
\end{array},\;
\eta\gamma\alpha\beta=\eta\gamma\eta\gamma=\alpha\beta\eta\gamma, \] 
$P_3$ has the $K$-basis
\[e_3,
\begin{array}{c}
\eta\\
\delta\\
\end{array},
\begin{array}{c}
\beta\eta\\
\gamma\eta\\
\end{array},\;
\alpha\beta\eta=\eta\gamma\eta,
\gamma\alpha\beta\eta=\gamma\eta\gamma\eta=\zeta\delta,
\]
and $P_4$ has the $K$-basis $e_4,\zeta,\delta\zeta$.

Let $\iota$ be the anti-involution induced by swapping 
$\alpha$ and $\beta$, $\gamma$ and $\eta$, $\delta$ and $\zeta$, respectively.
Take a totally ordered set
 $\Lambda=
 \{\lambda_1<\lambda_2<\lambda_3<\lambda_4<\lambda_5<\lambda_6\}$ and define
\[\cT(\lambda_1)=\cT(\lambda_6)=\{1\},\ \cT(\lambda_2)=\cT(\lambda_4)=\cT(\lambda_5)=\{1,2\},\
\cT(\lambda_3)=\{1,2,3\}.\]
Then we can construct a cellular basis of $A$ as follows.
\begin{align*}
& (c_{1,1}^{\lambda_1})=(e_2),\  
  (c_{i,j}^{\lambda_2})_{i,j\in \cT(\la_2)}
    =\begin{pmatrix}
    e_1 & \beta \\
    \alpha & \alpha\beta \\
    \end{pmatrix},\ 
 (c_{i,j}^{\lambda_3})_{i,j\in \cT(\la_3)}
    =\begin{pmatrix}
    e_3 & \gamma & \gamma\alpha \\
    \eta & \eta\gamma & \eta\gamma\alpha \\
    \beta\eta & \beta\eta\gamma & \beta\eta\gamma\alpha \\
    \end{pmatrix},\\[7pt]   
& (c_{i,j}^{\lambda_4})_{i,j\in \cT(\la_4)}
    =\begin{pmatrix}
    \gamma\eta & \gamma\eta\gamma \\
    \eta\gamma\eta & \eta\gamma\eta\gamma \\
    \end{pmatrix},\ 
 (c_{i,j}^{\lambda_5})_{i,j\in \cT(\la_5)}
    =\begin{pmatrix}
    e_4 & \delta \\
    \zeta & \zeta\delta \\
    \end{pmatrix},\ 
 (c_{1,1}^{\lambda_6})=(\delta\zeta).
 \end{align*}
Hence, $A_{11}$ is cellular.
 \end{proof}
%%%%%%%%%%%%%%%%%%%%%%%%%%%%%%%%%%%%%%%%%%%%%%%%%%%%%%%%%%%%%%%%%%%%%%%%%%%
%%%%%%%%%%%%%%%%%%%%%%%%%%%%%% A13 A14 %%%%%%%%%%%%%%%%%%%%%%%%%%%%%%%%%%%%
\begin{lem}
\label{P A13A14}
Neither $A_{13}$ nor $A_{14}$ is cellular.
\end{lem}
\begin{proof}
Let $B$ be the following bound quiver algebra.
\[\begin{xy}
(0,0) *[o]+{1}="A", (10,0) *[o]+{2}="B",
\ar @(ru,rd) "B";"B"^{z}
\ar @<2pt> "A";"B"^{x}
\ar @<2pt> "B";"A"^{y}
\end{xy}
\tiny{ \begin{array}{c}
  z^2=yx \\
 xy=0 \\
\end{array}
}\]
Then we have a surjective algebra homomorphism $B\twoheadrightarrow (e_1+e_2) A_{13} (e_1+e_2)$
given by 
$x\mapsto \beta$, $y\mapsto \gamma$ and $z\mapsto \alpha$. 
We also have a surjective algebra homomorphism $B\twoheadrightarrow (e_1+e_2) A_{14} (e_1+e_2)$
given by 
$x\mapsto \alpha$, $y\mapsto \beta$ and $z\mapsto \delta\gamma$. 
By comparing dimensions, one sees that these are isomorphism of algebras.
By Lemma \ref{F Lemma DB 2}, $B$ is not cellular. Therefore, 
 the assertion follows by Lemma\;\ref{C Lem idem}.
\end{proof}
%%%%%%%%%%%%%%%%%%%%%%%%%%%%%%%%%%%%%%%%%%%%%%%%%%%%%%%%%%%%%%%%%%%%%%%%%%%%%%
%%%%%%%%%%%%%%%%%%%%%%%%%%%%%%%% main b %%%%%%%%%%%%%%%%%%%%%%%%%%%%%%%%%%%%%%

\subsection{Proof of Theorem\;\ref{P main}\;(ii)}
In this subsection, we prove Theorem\;\ref{P main}\;(b).
By Theorem\;\ref{P theorem nonstandard nondomestic}, 
it is sufficient to show that the algebras $\Lambda_1$ and $\Lambda_2$ are not cellular. 
%%%%%%%%%%%%%%%%%%%%%%%%%%%%%%%%%%%%%%%%%%%%%%%%%%%%%%%%%%%%%%%%%%%%%%%%%%%%%%%
%%%%%%%%%%%%%%%%%%%%%%%%%%%%%%%%% lambda1 %%%%%%%%%%%%%%%%%%%%%%%%%%%%%%%%%%%%%
\begin{lem}
\label{P lambda1}
$\Lambda_1$ is not cellular.
\end{lem}
\begin{proof}
Suppose that $A=\Lambda_1$ is a cellular algebra with cell datum $(\Lambda,\cT,\cC,\iota)$.
Then,
\begin{align*}
P_1 &=
Ke_1\oplus K\alpha\oplus K\beta\oplus K\beta\alpha\oplus K\alpha^2\oplus K\alpha^3
\oplus K\beta\alpha^2\oplus K\alpha^4, \\
P_2 &=
Ke_2\oplus K\gamma\oplus K\alpha\gamma\oplus K\beta\gamma\oplus K\alpha^2\gamma
\oplus K\beta\alpha\gamma,
\end{align*}
and the Cartan matrix is given by
\[ \bC=\begin{pmatrix}
5& 3 \\
3 & 3\\
\end{pmatrix}.
\]
Let $\bD$ be the decomposition matrix. 
By the same argument we used in the proof of Lemma\;\ref{P A5}, we may assume that 
 \[\bD=\bordermatrix{
& \mu_1 & \mu_2 \cr
\lambda_1 & 1 & 1 \cr
\lambda_2 & 1 & 1 \cr
\lambda_3 & 1 & 1 \cr
\lambda_4 & 1 & 0 \cr
\lambda_5 & 1 & 0 \cr
} 
\] 
and that we may choose the linear extension $4<3<2<1<5$ of the partial order, and 
$P(\mu_1)_{\lambda_5}=K \alpha^4$, $P(\mu_2)_{\lambda_5}=0$.
 Proposition\;\ref{C Prop decom matrix} also implies that 
\[P(\mu_1)_{\lambda_1}/P(\mu_1)_{\lambda_5}\cong \D(\lambda_1)
\cong P(\mu_2)_{\lambda_1}/P(\mu_2)_{\lambda_5}=P(\mu_2)_{\lambda_1}.\]
Since $[\D(\lambda_1):L(\mu_1)]=[\D(\lambda_1):L(\mu_2)]=1$, it is easy to check that
\[P(\mu_1)_{\lambda_1}
=K(x\beta+y\beta\alpha+z\beta\alpha^2)\oplus K\alpha^3\oplus K\alpha^4,\]
for some $x,y, z\in K$. As $\Delta(\lambda_1)$ is isomorphic to a submodule of $P_2$, 
it is uniserial and $\Soc \Delta(\la_1)\cong L_2$. However, 
$\alpha^3\in \Soc P(\mu_1)_{\lambda_1}/P(\mu_1)_{\lambda_5}$ implies that 
$\Soc \Delta(\la_1)\cong L_1$. This is a contradiction.
\end{proof}
%%%%%%%%%%%%%%%%%%%%%%%%%%%%%%%%%%%%%%%%%%%%%%%%%%%%%%%%%%%%%%%%%%%%%%%%%%%%%%
%%%%%%%%%%%%%%%%%%%%%%%%%%%%%% lambda2 %%%%%%%%%%%%%%%%%%%%%%%%%%%%%%%%%%%%%%%
\begin{lem}
$\Lambda_2$ is not cellular.
\end{lem}
\begin{proof}
Suppose that $A=\Lambda_2$ is a cellular algebra with cell datum $(\Lambda,\cT,\cC,\iota)$. Then,
\begin{align*}
P_1 &=
Ke_1\oplus K\alpha\oplus K\beta\oplus K\beta\alpha\oplus K\alpha^2\oplus K\alpha^3\oplus K\alpha^4, \\
P_2 &=
Ke_2\oplus K\gamma\oplus K\alpha\gamma\oplus K\beta\gamma,
\end{align*}
and the Cartan matrix is given by
\[ \bC=\begin{pmatrix}
5 & 2 \\
2 & 2\\
\end{pmatrix}. \]
Then, modulo rearrangement of rows, the decomposition matrix is given by
  \[\bD=\bordermatrix{
& \mu_1 & \mu_2 \cr
\lambda_1 & 1 & 1 \cr
\lambda_2 & 1 & 1 \cr
\lambda_3 & 1 & 0 \cr
\lambda_4 & 1 & 0 \cr
\lambda_5 & 1 & 0 \cr
}. 
\]
Thus it is sufficient to consider the following three linear extensions of the partial order, 
without loss of generality.
\begin{itemize}
\item[(a)] $3<1<2<4<5$. 
\item[(b)] $3<1<4<2<5$. 
\item[(c)] $3<4<1<2<5$.
\end{itemize}
We have $\Delta(\lambda_5)\simeq P(\mu_1)_{\lambda_5}=K \alpha^4$
and $P(\mu_2)_{\lambda_5}=0$ in the three cases.
  
In the cases (b) and (c), we have
\[P(\mu_1)_{\lambda_2}=
K(xe_1+y\alpha+z\alpha^2+w\alpha^3)\oplus K(x'\beta+y'\beta\alpha)\oplus K\alpha^4,\]
for some $x,y,z,w,x',y'\in K$.
One can easily check that $x=y=z=0$.
Thus, $L_1$ appears in $\Soc \Delta(\lambda_2)\cong \Soc P(\mu_2)_{\la_2}\cong L_2$.
This is a contradiction.

In the case (a), we have $P(\mu_1)_{\lambda_4}=K\alpha^3\oplus K\alpha^4$
 and $P(\mu_2)_{\lambda_4}=0$. We also obtain
\[P(\mu_1)_{\lambda_2}=K(xe_1+y\alpha+z\alpha^2)
\oplus K(x'\beta+y'\beta\alpha)\oplus K\alpha^3\oplus K\alpha^4,\]
for some $x,y,z,x',y'\in K$.
It is easy to check that $x=y=0$. Hence, $L_1$ appears in 
$\Soc \Delta(\lambda_2)\cong \Soc P(\mu_2)_{\lambda_2}\cong L_2$.
This is a contradiction.
\end{proof}

%%%%%%%%%%%%%%%%%%%%%%%%%%%%%%%%%%%%%%%%%%%%%%%%%%%%%%%%%%%%%%%

%%%%%%%%%%%%%%%%%%%%%%%%%%%%%%%%%%%%%%%%%%%%%%%%%%%%%%%%%%%%%%%

%%%%%%%%%%%%%%%%%%%%%%%%%%%%%%%%%%%%%%%%%%%%%%%%%%%%%%%%%%%%%%%

%%%%%%%%%%%%%%%%%%%%%%%%%%%%%%%%%%%%%%%%%%%%%%%%%%%%%%%%%%%%%%%
\end{document}